\documentclass{amsart}
\usepackage{amsmath ,amssymb ,amsthm, mathrsfs, stmaryrd, dsfont}
\usepackage{bm}
\usepackage[all]{xy}
\usepackage[pdftex]{graphicx} 
\usepackage[pdftex ,pdfborder={0 0 0}]{hyperref} 
\usepackage{tikz, tikz-cd}
\usepackage{shadethm}
\usepackage{graphicx, xcolor}
\usepackage{url}
\DeclareMathAlphabet\mathbfcal{OMS}{cmsy}{b}{n}
 \DeclareMathAlphabet{\mathpzc}{OT1}{pzc}{m}{it}

\theoremstyle{definition}
\newtheorem{defin}{Definition}[section]

\newtheorem{thm}[defin]{Theorem}

\newtheorem{lem}[defin]{Lemma}
\newtheorem{prop}[defin]{Proposition}
\newtheorem{cor}[defin]{Corollary}
\newtheorem*{thmA}{Theorem A}
\newtheorem*{thmB}{Theorem B}
\newtheorem*{thmC}{Theorem C}
\newtheorem*{thmD}{Theorem D}
\newtheorem*{thmE}{Theorem E}

\theoremstyle{remark}
\newtheorem{rem}[defin]{Remark}
\newtheorem{eg}[defin]{Example}
\newtheorem{quest}[defin]{Question}

\newcommand{\Ric}{\mathrm{Ric}}
\newcommand{\tilFut}{\widetilde{\mathscr{F}}}
\newcommand{\Fut}{\mathscr{F}}

\newcommand{\dbar}{\bar{\partial}}

\newcommand{\sqddbar}{\sqrt{-1} \partial \bar{\partial}}
\newcommand{\sqdbard}{\sqrt{-1} \bar{\partial} \partial}
\newcommand{\M}{\mathcal{M}}

\title[Constant $\mu$-scalar curvature K\"ahler metric]{Constant $\mu$-scalar curvature K\"ahler metric -- formulation and foundational results}

\author{Eiji Inoue}

\address{Graduate School of Mathematical Sciences, the University of Tokyo \endgraf 3-8-1 Komaba, Meguro, Tokyo 153-8914, Japan. }

\email{eijinoe@ms.u-tokyo.ac.jp}

\begin{document}

\begin{abstract}
We introduce $\mu$-scalar curvature for a K\"ahler metric with a moment map $\mu$ and start up a study on constant $\mu$-scalar curvature K\"ahler metric as a generalization of both cscK metric and K\"ahler-Ricci soliton and as a continuity path to/from extremal metric. 
We exhibit some fundamental constraints to the existence of constant $\mu$-scalar curvature K\"ahler metrics by investigating a variant of Tian--Zhu's volume functional, which is closely related to Perelman's W-functional. 
A new K-energy is also introduced as an approach to the uniqueness problem of constant $\mu$-scalar curvature K\"ahler metrics and as a prelude to new K-stability concept. 
\end{abstract}

\maketitle

\tableofcontents

\section{Introduction}

In this article, we propose a new variant of scalar curvature of K\"ahler metric, which we call $\mu$-scalar curvature, motivated by a version of Donaldson-Fujiki moment map picture for a weighted measure $e^{\theta_\xi} \omega^n$ associated to a holomorphic vector field $\xi^J$. 
Here we simply begin with the definition of $\mu$-scalar curvature and the main results of this article. 

\subsubsection*{Setup}

Let $X$ be a K\"ahler manifold and $\omega$ be a K\"ahler form on $X$. 
We call a smooth real vector field $\xi$ on $X$ \textit{$\dbar$-Hamiltonian} with respect to $\omega$ if the complexified vector field $\xi^J := J \xi + \sqrt{-1} \xi$ is holomorphic ($\Leftrightarrow L_\xi J = 0$) and $i_{\xi^J} \omega$ is $\dbar$-exact. 
Note that $i_{\xi^J} \omega$ is $\dbar$-closed for any holomorphic $\xi^J$. 
As $\xi^J$ is holomorphic, we have $i_{\xi^J} (\omega + \sqddbar \phi) = i_{\xi^J} \omega + \sqrt{-1} \dbar \xi^J \phi$, so that the $\dbar$-Hamiltonian property does not depend on the choice of the K\"ahler form $\omega$ in the fixed K\"ahler class $[\omega]$. 
Moreover, it is known by \cite{LS} that a vector field $\xi$ preserving $J$ on a compact K\"ahler manifold is $\dbar$-Hamiltonian with respect to $[\omega]$ if and only if it has a fixed point, thus in particular the $\dbar$-Hamiltonian property is even independent of the K\"ahler class $[\omega]$. 
We call a function $\theta$ satisfying $\sqrt{-1} \dbar \theta = i_{\xi^J} \omega$ a \textit{$\dbar$-Hamiltonian potential} with respect to $\omega$, which is complex-valued in general. 
We call a $\dbar$-Hamiltonian vector field $\xi$ \textit{properly $\dbar$-Hamiltonian} if $\xi$ generates a closed torus, i.e., the closure $\overline{\exp \mathbb{R} \xi} \subset \mathrm{Aut} (X)$ is compact. 

We define the \textit{$\mu_\xi$-scalar curvature} $s_\xi (\omega)$ of a K\"ahler metric $\omega$ and a $\dbar$-Hamiltonian vector $\xi$ by 
\begin{equation}
s_\xi (\omega) = (s (\omega) + \bar{\Box} \theta) + (\bar{\Box} \theta - \xi^J \theta), 
\end{equation} 
where $\theta$ is a $\dbar$-Hamiltonian potential of $\xi$ with respect to $\omega$ and $s (\omega)$ denotes the K\"ahlerian scalar curvature: $s (\omega) := - g^{k \bar{l}} \partial_k \dbar_l \log \det g$. 
We can take a real-valued $\theta$ iff $\omega$ is $\xi$-invariant since we have $\sqrt{-1} (d \mathrm{Re} \theta - J d \mathrm{Im} \theta) = \sqrt{-1} \dbar \theta - \overline{\sqrt{-1} \dbar \theta} = i_{\xi^J - \bar{\xi}^J} \omega = 2 \sqrt{-1} i_\xi \omega$. 
In this case, $\theta$ is $\xi$-invariant, so $\xi^J \theta$ is also real valued. 

A version of Donaldson-Fujiki moment map picture characterizes $\mu$-scalar curvature. 
As we will see this motivative interpretation in section \ref{section: Donaldson-Fujiki}, here we instead simply observe how the individual terms of the $\mu$-scalar curvature arise. 
The first term $s (\omega) + \bar{\Box} \theta$ is just the trace of the Bakry--Emery Ricci curvature $\Ric (\omega) - \sqddbar \theta$, which is well-studied in Riemannian and metric measure geometry. 
The second term $\bar{\Box} \theta - \xi^J \theta$ often arises as the Lie derivative of the weighted measure: 
\[ L_{\xi^J} (e^{\theta} \omega^n) = -(\bar{\Box} \theta - \xi^J \theta) e^{\theta} \omega^n. \]
Another important aspect is that this second term is a $\bar{\partial}$-Hamiltonian potential of the Bakry--Emery Ricci curvature $\mathrm{Ric} (\omega) - \sqrt{-1} \partial \bar{\partial} \theta$, i.e. $\sqrt{-1} \bar{\partial} (\bar{\Box} \theta - \xi^J \theta) = i_{\xi^J} (\mathrm{Ric} (\omega) - \sqrt{-1} \partial \bar{\partial} \theta)$. 

We may also regard $s_\xi (\omega)$ as the trace of the following `complex analogy of ($m$=1)-Bakry--Emery curvature' for an integrable complex structure $J$: 
\[ \Ric (\omega) + 2 \sqddbar \theta - \sqrt{-1} \partial \theta \wedge \bar{\partial} \theta. \]
We can easily see that this $(1,1)$-form does not change by simultaneously replacing the equivariant form $\omega + \theta$ with $\tilde{\omega} + \tilde{\theta} = c (\omega + \theta)$ and $\xi$ with $\tilde{\xi} = c^{-1} \xi$ for a positive constant $c > 0$, so that we have 
\begin{align} 
\label{multiple}
s_{c^{-1} \xi} (c \omega) 
&= c^{-1} s_{\xi} (\omega)
\end{align}
for every positive constant $c > 0$. 
(Note $s_{\xi} (c \omega) \neq c^{-1} s_{\xi} (\omega)$ when $\xi \neq 0$. )

Next, introducing a parameter $\lambda \in \mathbb{R}$, we define the \textit{$\mu^\lambda_\xi$-scalar curvature} of a K\"ahler metric $\omega$ by 
\begin{equation}
s^\lambda_\xi (\omega) = (s (\omega) + \bar{\Box} \theta) + (\bar{\Box} \theta - \xi^J \theta) - \lambda \theta. 
\end{equation}
We call a K\"ahler metric $\omega$ a \textit{constant $\mu_\xi^\lambda$-scalar curvature K\"ahler metric} (\textit{$\mu^\lambda_\xi$-cscK metric} for short) if $s^\lambda_\xi (\omega)$ is constant. 
We may also use variant terminologies such as \textit{$\mu^\lambda$-cscK metric} or \textit{$\mu$-cscK metric} when the abbreviated parameters are determined/unimportant in the context. 
Since we have $\mathrm{Im} (s^\lambda_\xi (\omega)) = \Delta \mathrm{Im} \theta - 2 J \xi (\mathrm{Im} \theta) - \lambda \mathrm{Im} \theta$, we get 
\begin{align*} 
\int_X \mathrm{Im} (s^\lambda_\xi (\omega)) \mathrm{Im} \theta ~ e^{\mathrm{Re} \theta} \omega^n 
&= \int_X (\Delta \mathrm{Im} \theta -2 J\xi (\mathrm{Im} \theta) - \lambda \mathrm{Im} \theta) \mathrm{Im} \theta ~ e^{\mathrm{Re} \theta} \omega^n 
\\
&= \int_X |d \mathrm{Im} \theta|^2 e^{\mathrm{Re} \theta} \omega^n - \lambda \int_X (\mathrm{Im} \theta)^2 e^{\mathrm{Re} \theta} \omega^n
\end{align*}
by $(d \mathrm{Im} \theta, \mathrm{Re} \theta) = 2\xi (\mathrm{Re} \theta) = 2 J \xi (\mathrm{Im} \theta)$. 
Thus we automatically obtain $L_\xi \omega = 0$ for any $\mu^\lambda_\xi$-cscK metric $\omega$ with $\lambda \le 2\lambda_1$ for the positive first eigenvalue $\lambda_1$ of $\frac{1}{2} (\Delta - \nabla \mathrm{Re} \theta)$. 
In this article, we are mainly interested in this case, especially the case $\lambda \le 0$. 
So from now on we always assume the $\xi$-invariance of the K\"ahler metric $\omega$. 
Hence $s^\lambda_\xi (\omega)$ is real-valued. 

A cscK metric obviously gives an example of $\mu$-cscK metric for $\xi = 0$ and every $\lambda$. 
We will see in section \ref{KR soliton} that a K\"ahler-Ricci soliton $\Ric (\omega) - L_{\xi^J} \omega = \lambda \omega$ gives an example of $\mu^\lambda_\xi$-cscK metric for $\lambda > 0$, which satisfies $\lambda \le 2 \lambda_1$. 
If $\omega$ is a $\mu^\lambda_\xi$-cscK metric, then for any positive constant $c > 0$, 
$\tilde{\omega} = c \omega$ is a $\mu^{\tilde{\lambda}}_{\tilde{\xi}}$-cscK metric for $\tilde{\xi} := c^{-1} \xi$ and $\tilde{\lambda} := c^{-1} \lambda$. 
The product $(X \times Y, \omega_X + \omega_Y)$ of a $\mu^\lambda_{\xi_X}$-cscK metric $\omega_X$ on $X$ and a $\mu^\lambda_{\xi_Y}$-cscK metric $\omega_Y$ on $Y$ gives a $\mu^\lambda_{\xi_X + \xi_Y}$-cscK metric. 

\subsubsection*{Main results}

Now we collect the main results in this article. 
The first three results, except for Theorem B (3), are analogous to well-known foundational results on cscK metric and K\"ahler-Ricci soliton (cf. \cite{Sze}, \cite{TZ}. \cite{AS} is a good survey. ). 

In the following, $X$ denotes a compact K\"ahler manifold. 
Let us firstly recall the reduced automorphism group (cf. \cite{Gau}). 
Put 
\begin{equation}
\mathfrak{h}_0 (X) := \{ \xi \in \mathcal{X}_\mathbb{R} (X) ~|~ \xi \text{ is $\dbar$-Hamiltonian. } \}, 
\end{equation}
which is naturally a complex vector space by putting $\sqrt{-1} \xi := J \xi$. 
Denote by $\mathrm{Aut}^0 (X/\mathrm{Alb})$ the connected subgroup of the group $\mathrm{Aut} (X)$ of biholomorphisms associated to $\mathfrak{h}_0 (X)$. 
It is known by \cite{LS} that $\mathfrak{h}_0 (X)$ is the space of vector fields tangent to the Jacobi map $A_x: X \to \mathrm{Alb} (X)$, or equivalently, the space of vector fields with non-empty zero set. 
This group $\mathrm{Aut}^0 (X/\mathrm{Alb})$ is called the \textit{reduced automorphism group} of $X$. 
Similarly, we put $\mathfrak{h}_{0, \xi} (X) := \{ \zeta \in \mathfrak{h}_0 (X) ~|~ [\xi, \zeta] = 0 \}$ for a vector $\xi \in \mathfrak{h}_0 (X)$ and denote by $\mathrm{Aut}^0_\xi (X/ \mathrm{Alb})$ the connected subgroup of $\mathrm{Aut} (X)$ associated to $\mathfrak{h}_{0, \xi} (X)$. 

Note that for a line bundle $L$ on $X$, the identity component $\mathrm{Aut}^0 (X, L)$ of the group of biholomorphisms lifting to $L$ is contained in the reduced automorphism group $\mathrm{Aut}^0 (X/ \mathrm{Alb})$. 
Moreover, it is known that $\mathrm{Aut}^0 (X, L^{\otimes n})$ coincides with $\mathrm{Aut}^0 (X/ \mathrm{Alb})$ for some positive integer $n$ (because $\mathrm{Aut}^0 (X/ \mathrm{Alb})$ is linear algebraic). 
If $X$ has no holomorphic 1-form, or equivalently $b^1 (X) = 0$, then the reduced automorphism group $\mathrm{Aut}^0 (X/\mathrm{Alb})$ coincides with the identity component $\mathrm{Aut}^0 (X)$ of the group of biholomorphisms of $X$. 

\begin{thmA}[Reductiveness]
Let $\omega$ be a constant $\mu^\lambda_\xi$-scalar curvature K\"ahler metric on a compact K\"ahler manifold $X$. 
Then
\begin{enumerate}
\item the group $\mathrm{Aut}^0_\xi (X/\mathrm{Alb})$ is the complexification of the compact connected subgroup ${^\nabla \mathrm{Isom}}^0_\xi (X, \omega)$ associated to the Lie algebra of Hamiltonian Killing vector fields with respect to the $\mu^\lambda_\xi$-cscK metric $\omega$ compatible with $\xi$. 
Especially, it is reductive. (Corollary \ref{reductiveness})

\item When $\lambda \le 0$, $\mathrm{Aut}^0_\xi (X/\mathrm{Alb})$ is maximal among reductive subgroups of $\mathrm{Aut}^0 (X/\mathrm{Alb})$. 
This fails in general when $\lambda \gg 0$. (Corollary \ref{K-optimal vectors are finite}) 
\end{enumerate}
\end{thmA}

When $b^1 (X) = 0$, we can replace $\mathrm{Aut}^0_\xi (X/\mathrm{Alb})$ by the identity component $\mathrm{Aut}^0_\xi (X)$ of the group of biholomorphisms preserving $\xi$ and ${^\nabla \mathrm{Isom}}^0_\xi (X, \omega)$ by the identity component $\mathrm{Isom}_\xi^0 (X, \omega)$ of the group $\mathrm{Isom}_\xi (X, \omega)$ of isometries preserving $\xi$. 

\begin{thmB}[$\mu$-Futaki invariant and $\mu$-volume functional]
Let $X$ be a compact K\"ahler manifold, $[\omega]$ be a K\"ahler class and $\xi$ be a properly $\dbar$-Hamiltonian vector field on $X$. 
\begin{enumerate}
\item There is a $\mathbb{C}$-linear functional $\Fut_\xi^\lambda: \mathfrak{h}_0 (X) \to \mathbb{C}$ depending only on the quadruple $(X, [\omega], \xi, \lambda)$ such that $\Fut_\xi^\lambda$ vanishes if the K\"ahler class $[\omega]$ admits a $\mu^\lambda_\xi$-cscK metric. (Proposition \ref{Futaki})

\item For any compact Lie subgroup $K \subset \mathrm{Aut}^0 (X/\mathrm{Alb})$ and $\lambda \in \mathbb{R}$, there always exists a vector $\xi \in \mathfrak{k}$ such that $\Fut^\lambda_\xi|_{\mathfrak{k}^c}$ vanishes, regardless of the existsnce of $\mu^\lambda$-cscK metrics. (Corollary \ref{Futaki vanish})

\item We have the uniqueness of such $\xi$ for $\lambda \ll 0$. 
Moreover, the value 
\[ \lambda_{\mathrm{freeze}} := \sup \{ \lambda \in \mathbb{R} ~|~ \forall \lambda' < \lambda \quad \Fut^\lambda_{\xi_{\lambda'}}|_{\mathfrak{k}^c} = 0 \text{ for a unique } \xi_{\lambda'} \}  \]
is never $\pm \infty$. 
(Proposition \ref{extremal background} and its remark)
\end{enumerate}
\end{thmB}

The second claim in the above is a partial generalization of a volume minimization result in \cite{TZ}, except for the uniqueness. 
Indeed, we will see in section 5.2 that vectors $\xi$ with $\Fut_\xi^\lambda \equiv 0$ are not unique for $\lambda \gg 0$ as claimed in the above (3). 
As for general theory, we are mainly interested in the case $\lambda \le 0$ since in this case we have a nice compactness/finiteness results as in Corollary \ref{compact} and Corollary \ref{K-optimal vectors are finite}. 
The author suspects the above value $\lambda_{\mathrm{freeze}}$ is always slightly positive. 

In section 4, we prove the following extension result, which generalizes the result of \cite{Chen2}. 
This is the first step for studying the uniqueness of $\mu^\lambda_\xi$-cscK metrics and `$\mu$K-stability'. 

\begin{thmC}[$\mu$K-energy and geodesic]
Let $(X, [\omega])$ be a compact K\"ahler manifold with $T$-action. 
Fix a vector $\xi \in \mathfrak{t}$ and $\lambda \in \mathbb{R}$. 
\begin{enumerate}
\item There is a functional $\M_\xi^\lambda$ on the space of $\xi$-invariant smooth K\"ahler metrics in the K\"ahler class $[\omega]$ such that the critical points of $\M_\xi^\lambda$ are precisely $\mu^\lambda_\xi$-cscK metrics and that $\M^\lambda_\xi$ is convex along smooth geodesics. 

\item There is a canonical extension of this functional $\mathcal{M}^\lambda_\xi$ to the space $\overline{\mathcal{H}^{1,1}_{\omega, \xi}}$ of $\xi$-invariant \textit{sub-K\"ahler metrics with $C^{1,1}$-potentials}. 
Here a sub-K\"ahler metric with $C^{1,1}$-potentials means a $(1,1)$-form $\omega_\phi = \omega + \sqddbar \phi$ with $L^\infty$-coefficients given by a smooth K\"ahler metric $\omega$ and a $C^{1,1}$-smooth $\omega$-psh function $\phi$, i.e., a $C^{1,1}$-smooth function satisfying $\omega_\phi \ge 0$ as a current. 
\end{enumerate}
\end{thmC}

The following result illustrates an intriguing special aspect of $\mu$-scalar curvature. 
From this result, our parameter $\lambda$ can be thought as a continuity path connecting $\mu^0$-cscK metric/K\"ahler-Ricci soliton and extremal metric. 

\begin{thmD}[Behavior of K-optimal vectors]
Fix a compact subgroup $K \subset \mathrm{Aut}^0 (X/\mathrm{Alb})$. 
\begin{enumerate}
\item 
Let $\{ (\xi_i, \lambda_i) \in \mathfrak{k} \times \mathbb{R} \}_{i \in \mathbb{N}}$ be a sequence satisfying $\Fut^{\lambda_i}_{\xi_i}|_{\mathfrak{k}^c} \equiv 0$ and $\lambda_i \to - \infty$. 
(Note such a sequence always exists by Theorem B. )
Then the rescaled sequence $\lambda_i \xi_i \in \mathfrak{k}$ converges to the extremal vector $\xi_{\mathrm{ext}}$ which is uniquely characterized by the property 
\[ \check{\Fut}^0_{\xi_{\mathrm{ext}}} (\zeta) := \int_X \Big{(} (s (\omega) - \underline{s}) - (\theta_{\xi_{\mathrm{ext}}} - \underline{\theta}_{\mathrm{ext}}) \Big{)} \theta_\zeta \omega^n = 0 \]
for every $\zeta \in \mathfrak{k}$, where we put $\underline{s} = \int_X s (\omega) \omega^n / \int_X \omega^n$ and $\underline{\theta}_{\mathrm{ext}} := \int_X \theta_{\xi_{\mathrm{ext}}} \omega^n / \int_X \omega^n$. 
(Section 2.2 and Corollary \ref{compact})

\item If there are $\mu^{\lambda_i}_{\xi_i}$-cscK metrics $\omega_i$ with a uniform $C^{3, \alpha}$-bound of the K\"ahler potentials $\phi_i$ of $\omega_i = \omega + \sqddbar \phi_i$ and a uniform lower bound $C \omega \le \omega_i$, then $\omega_i$ subconverges to an extremal metric $\omega_{\mathrm{ext}}$ on $X$. (Section 2.2) 

\item Conversely, if there is an extremal metric on $\omega_{\mathrm{ext}}$ in a K\"ahler class $[\omega]$, then there are constants $\lambda_-$ and $\lambda_+$ such that there is a family of $\mu$-cscK metrics $\{ \omega_\lambda \}_{\lambda \in (-\infty, \lambda_-) \cup (\lambda_+, \infty)}$ where for each $\lambda \in (-\infty, \lambda_-) \cup (\lambda_+, \infty)$ the metric $\omega_\lambda$ is a $\mu^\lambda_{\xi_\lambda}$-cscK metric in the K\"ahler class $[\omega]$ for some vector field $\xi_\lambda$ such that $\omega_\lambda$ converges to $\omega_{\mathrm{ext}}$ smoothly as $\lambda \to \pm \infty$. (Theorem \ref{extremal propagation})
\end{enumerate}
\end{thmD}

By Theorem B (3), the vector field $\xi_\lambda$ in the above Theorem D (3) is unique for each $\lambda \ll 0$, while we may have other solutions for $\lambda \gg 0$.

In section 5.1, we prove the following result analogous to one of the main results in \cite{LS} on extremal metric. 

\begin{thmE}[Perturbation of K\"ahler class]
Let $\omega$ be a $\mu^\lambda_\xi$-cscK metric on a compact K\"ahler manifold $X$. 
Suppose we have $\lambda < 2 \lambda_1$ for the first eigenvalue $\lambda_1 > 0$ of the operator $\bar{\Box} - J\xi = \frac{1}{2} (\Delta - \nabla \mathrm{Re} \theta_\xi)$ restricted to the space $C^\infty_\xi (X, \mathbb{R})$ of $\xi$-invariant real-valued functions. 
Then there exists a neighbourhood $U$ of $[\omega]$ in the K\"ahler cone and a positive constant $\epsilon > 0$ such that for every K\"ahler class $[\tilde{\omega}] \in U$ and $\tilde{\lambda} \in (\lambda - \epsilon, \lambda + \epsilon)$, there exists a vector $\tilde{\xi}$ and a constant $\mu^{\tilde{\lambda}}_{\tilde{\xi}}$-scalar curvature K\"ahler metric $\tilde{\omega}_{\tilde{\lambda}}$ in the K\"ahler class $[\tilde{\omega}]$. 
\end{thmE}

It follows that if a K\"ahler class $[\omega]$ admits a cscK metric, then a small perturbation of $[\omega]$ admits both extremal metric and $\mu^\lambda$-cscK metrics for $\lambda < 2 \lambda_1$. 
For example, the K\"ahler class $c_1 (X)$ of $\mathbb{C}P^2$ blown up at three points $(1:0:0), (0:1:0), (0:0:1)$ admits cscK metric, but a small perturbation of $c_1 (X)$ does not admit cscK metrics (cf. \cite[Example 3.2]{LS}). 
This example shows that there exists a non-trivial path of $\mu^\lambda$-cscK metrics connecting extremal metric and $\mu^0$-cscK metric. 

In section 6.2, we give more explicit examples of $\mu$-cscK metrics on ruled surfaces, using Calabi ansatz method. 
It turns out that there is a K\"ahler class that does not admit extremal metrics, but do admit $\mu^0$-cscK metrics. 
In particular, we observe that there exists a path of $\mu^\lambda$-cscK metrics connecting K\"ahler--Ricci soliton and extremal metric on $X = \mathbb{C}P^2 \# \overline{\mathbb{C}P^2}$ in the K\"ahler class $c_1 (X)$. 

\subsubsection*{Relation with Lahdili's work}
\label{Lahdili}

Just after uploading the first version of this article on arXiv, the author was informed that the $\mu$-scalar curvature is a special part of weighted scalar curvature introduced in \cite{Lah}. 
As there are some overlaps on the results, especially on $\mu$K-energy, we collect them here. 
Proposition \ref{expression} in this article should correspond to Theorem 5 in \cite{Lah} and its corollary is mentioned in Remark 4. 
Proposition \ref{convexity} is also covered in the proof of Proposition 1 in his article. 
Proposition \ref{Futaki} corresponds to Proposition 2 in \cite{Lah}, but the statement is slightly different. 
Computing with $\mu$-Lichnerowicz operator, we can extend the domain of $\Fut_\xi^\lambda$ to $\mathfrak{h}_0 (X)$ from the centralizer $\mathfrak{h}_{0, \xi} (X)$ of $\xi$, where the latter case is considered in \cite{Lah}. 
While it is natural to consider only a torus equivariant test configurations to formulate weighted K-stability (or $\mu$K-stability), our slight deviation to non-equivariant direction $\mathfrak{h}_{0, \xi} (X)^\bot \subset \mathfrak{h}_0 (X)$ enables us to conclude the maximality of $\mathrm{Aut}_\xi^0 (X/\mathrm{Alb}) \subset \mathrm{Aut}^0 (X/\mathrm{Alb})$ among reductive subgroups in $\mathrm{Aut}^0 (X/\mathrm{Alb})$ for $X$ admitting a $\mu^\lambda_\xi$-cscK metric for some $\lambda \le 0$ (see Corollary \ref{K-optimal vectors are finite}). 
This is indeed not the case when $\lambda \gg 0$ as we see in the example of $\mathbb{C}P^1$. 
Lahdili also considers a weighted Futaki invariant for smooth test configurations, which should have an advantage towards an algebraic formulation of $\mu$K-stability (or weighted K-stability) for general test configurations (cf. section \ref{section: K-stability}). 
The materials in the section \ref{section: connecting path}, \ref{section: volume} and \ref{section: perturbation} of this article have different original flavors from these overlaps. 

\subsubsection*{Organization}

In section \ref{section: Donaldson-Fujiki}, we explain a motivative interpretation of $\mu$-scalar curvature as a moment map. 
We observe in section \ref{section: connecting path} how $\mu$-scalar curvature is related to extremal metric, assuming some results in section \ref{section: volume}. 
We prove Theorem A in section \ref{section: reductive} and check Theorem B (1) in section \ref{section: Futaki} using a formula obtained in section \ref{section: reductive}. 
Theorem B (2) is verified in section \ref{section: volume}. 
We also prove Theorem D in this section, combining with the observation in section \ref{section: connecting path}. 
Theorem C is demonstrated in section \ref{section: muK-energy}. 
We also present some naive stability notion which should fit into our `$\mu$'-formulation and would be refined in future studies (cf. \cite{Lah, Ino2}). 
In section 5, we prove Theorem E, Theorem B (3) and Theorem D (3). 
In section 6, we firstly observe there are non-trivial $\mu^\lambda$-cscK metrics on $\mathbb{C}P^1$ for $\lambda \gg 0$, and then we construct explicit examples of $\mu$-cscK metrics on ruled surfaces, using Calabi ansatz method. 

\subsubsection*{Acknowledgement}

I am grateful to my advisor Shigeharu Takayama for his constant support. 
I am also grateful to Akito Futaki for helpful comments. 
I thank Abdellah Lahdili and Ruadha\'i Dervan for informing me about relation with Lahdili's work on weighted scalar curvature, which is a far generalization of $\mu$-scalar curvature in view of moment map picture. 
I am thankful to Vestislav Apostolov, Julien Keller, Mehdi Lejmi for kindly suggesting me to see examples on ruled surfaces using Calabi ansatz method. 
Finally, I am grateful to Abdellah Lahdili for many helpful discussions. 
This work is supported by JSPS KAKENHI Grant Number 18J22154 and the Program for Leading Graduate Schools, MEXT, Japan.

\section{Motivative observation on $\mu$-cscK metrics}

\subsection{$\mu$-scalar curvature and Donaldson-Fujiki picture}
\label{section: Donaldson-Fujiki}

In this section, we fix a symplectic structure $\omega$ on a smooth manifold $M$ and a smooth vector field $\xi$ preserving $\omega$ and vary complex structures $J$ and the parameter $\lambda \in \mathbb{R}$. 

In this setup, we can also consider all variation of K\"ahler metrics on a fixed complex manifold $X = (M, J)$ and in a fixed K\"ahler class as follows. 
Let $\omega'$ be another K\"ahler metric on $X$ in the K\"ahler class $[\omega]$. 
As $t \omega' + (1-t) \omega$ is nondegenerate for all $t \in [0,1]$, we can apply Moser's theorem to obtain a diffeomorphism $\phi$ of $M$ so that $\phi^* \omega' = \omega$. 
We obtain another $\omega$-compatible complex structure $\phi^* J$ on $M$. 
Conversely, if we have an $\omega$-compatible complex structure $J'$ which is biholomorphic to $J$ via some diffeomorphism $\phi$ of $M$, i.e. $J' = \phi^* J$, then we obtain another K\"ahler form $\omega' := (\phi^{-1})^* \omega$ on $X = (M, J)$. 
This gives a natural identification of the space of K\"ahler metrics in a fixed K\"ahler class $[\omega]$ on a complex manifold $X = (M, J)$ with the quotient space 
\[ \{ J' \in \mathcal{J} (M, \omega) ~|~ (M, J') \text{ is biholomorphic to } (M, J) \} / \mathrm{Symp} (M, \omega). \]
The leaf $\{ J' \in \mathcal{J} (M, \omega) ~|~ (M, J') \text{ is biholomorphic to } (M, J) \}$ can be psychologically regarded as `the orbit of $J$ by the complexified action of $\mathcal{J} (M, \omega) \curvearrowleft \mathrm{Symp} (M, \omega)$', which enables us to interpret YTD-type conjecture as an infinite dimensional analogy of the finite dimensional Kempf-Ness theorem. (cf. \cite{Don}, \cite{Sze}. )

\subsubsection{Moment map}

Let $(M, \omega)$ be a closed $C^\infty$-symplectic manifold. 
A smooth vector field $\xi$ on $M$ preserves $\omega$ if and only if $i_\xi \omega$ is closed, as $L_\xi \omega = d i_\xi \omega$. 
Put $\mathfrak{ham} (M, \omega)$ as 
\begin{equation}
\mathfrak{ham} (M, \omega) := \{ \xi \in \mathfrak{X} (M) ~|~ i_\xi \omega \text{ is exact } \},  
\end{equation}
which is a Lie subalgebra of the Lie algebra $\mathfrak{X} (M)$ of smooth vector fields. 

A smooth right action $M \curvearrowleft T$ by a Lie group $T$ is called \textit{Hamiltonian} if the linearization $\mathfrak{t} \to \mathfrak{X} (M)$ factors through $\mathfrak{ham} (M, \omega)$. 
In this case, we have an equivariant smooth map $\mu: M \to \mathfrak{t}^*$ satisfying $-d \mu_\xi = i_\xi \omega$ called a \textit{moment map}, where we consider the coadjoint action on $\mathfrak{t}^*$ and $\mu_\xi$ is a real valued function on $M$ defined by $\mu_\xi (x) := \langle \mu (x), \xi \rangle$. 
For two moment maps $\mu, \mu'$ with respect to the same action and the same symplectic form $\omega$, we know that $\mu - \mu'$ is constant as $d (\mu - \mu') = 0$ and moreover $\mu - \mu' \in (\mathfrak{t}^*)^T = \{ \nu \in \mathfrak{t}^* ~|~ \nu. t = \nu \text{ for every } t \in T \}$ by the equivariance of the maps. 
In other words, moment maps are unique modulo $(\mathfrak{t}^*)^T$. 
We will mainly consider an action by a closed real torus $T \cong (U (1))^k$ and have $(\mathfrak{t}^*)^T = \mathfrak{t}^*$ in this case. 

There is an associated element $[\omega + \mu] \in H^2_T (M; \mathbb{R})$ of the equivariant de Rham cohomology. 
For another $T$-invariant symplectic form $\omega'$ and a moment map $\mu'$, we have $[\omega + \mu] = [\omega' + \mu']$ if and only if there exists a $T$-invariant 1-form $\phi$ such that $\omega = \omega' + d \phi$ and $\mu_\xi = \mu'_\xi + i_\xi \phi$. 
In particular, we have $[\omega + \mu] = [\omega + \mu']$ if and only if $\mu = \mu'$ as $\xi$ has a zero. 
The push-forward measure $\mu_* (\omega^n/n!)$ on $\mathfrak{t}^*$ is called the \textit{Duistermaat--Heckman measure}, which defines the same measure independent of the choice of $\omega + \mu$ in the same equivariant cohomology class \cite{GGK}.

\subsubsection{Vector fileds}

For an $\omega$-compatible almost complex structure $J$, we put 
\begin{align}
\xi^J := J \xi + \sqrt{-1} \xi, \quad \bar{\xi}^J := J \xi - \sqrt{-1} \xi 
\end{align}
and $\theta_\xi := -2 \mu_\xi$. 
Then we have 
\begin{align}
\sqrt{-1} \bar{\partial} \theta_\xi 
= i_{\xi^J} \omega, 
\quad 
\sqrt{-1} \partial \theta_\xi 
= - i_{\bar{\xi}^J} \omega, 
\end{align}
where $\bar{\partial} := (d + \sqrt{-1} Jd)/2$ and $\partial := (d - \sqrt{-1} Jd)/2$.
In other words, we have $\xi^J = g^{p \bar{q}} \theta_{\bar{q}} \partial_p$ and $\bar{\xi}^J = g^{p \bar{q}} \theta_p \bar{\partial}_q$ in the usual K\"ahlerian notation. 
We also have $J \xi = - \nabla_{g_J} \mu_\xi$ and thus 
\begin{align} 
\xi^J \theta_\xi 
&= -2 (J\xi) \mu_\xi = 2 |\xi|^2_{g_J} = |\xi^J|^2_{g_J} 
\\
&= |\dbar \theta_\xi|^2_{g_J} = \mathrm{tr}_{g_J} (\sqrt{-1} \partial \theta_\xi \wedge \dbar \theta_\xi). 
\end{align}

\subsubsection{$\mu$-scalar curvature}

Let $(M, \omega)$ be a closed $C^\infty$-symplectic manifold with a Hamiltonian action by a closed real torus $T$ and $\mu: M \to \mathfrak{t}^*$ be a moment map. 
For an $\omega$-compatible almost complex structure $J$, we denote by $s (J)$ the hermitian scalar curvature defined by Donaldson \cite{Don}, which coincides with the usual K\"ahler (the half of the Riemannian) scalar curvature $s_{\text{K\"a}} (g_J) = \frac{1}{2} s_{\text{Rm}} (g_J)$ for integrable $J$. 
Note that this $s (J)$ differs from the half of the Riemannian scalar curvature $\frac{1}{2} s_{\text{Rm}} (g_J)$ for non-integrable $J$ in general. 

The \textit{$\mu$-scalar curvature} $s_\xi (g_J)$ of a metric $g_J (\cdot, \cdot) = \omega (\cdot, J \cdot)$ (associated to a $T$-invariant $\omega$-compatible almost complex structure $J$ on $M$) with respect to a vector $\xi \in \mathfrak{t}$ is defined as follows: 
\begin{equation}
\label{mu scalar}
s_\xi (g_J) := (s (J) - \Delta_{g_J} \mu_\xi) + (-\Delta_{g_J} \mu_\xi + 2\xi^J \mu_\xi), 
\end{equation}
where $\Delta_{g_J}$ denotes the usual Riemannian Laplacian $\Delta_{g_J} = d^* d$ with respect to $g_J$
Since two moment maps with respect to the same symplectic form only differ by a constant, (\ref{mu scalar}) is independent of the choice of the moment map $\mu$. 
When $\xi = 0$ and $J$ is integrable, the $\mu$-scalar curvature is of course nothing but the usual K\"ahlerian scalar curvature. 

As $\Delta_{g_J}$ is the twice of $\dbar/\partial$-Laplacians $\bar{\Box} = \Box = - g^{i \bar{j}} \partial_i \bar{\partial}_j$ when $J$ is integrable, we can express (\ref{mu scalar}) as 
\begin{align} 
\notag
s_\xi (g_J) 
&= (s_{\text{K\"a}} (g_J) + \bar{\Box} \theta_\xi) + (\bar{\Box} \theta_\xi - \xi^J \theta_\xi) 
\\ \notag
&= (\bar{\Box} \log \det g - \xi^J \log \det g + \sum_{i=1}^n \partial_i \xi^i) + (\bar{\Box} \theta_\xi - \xi^J \theta_\xi)
\\ 
&=(\bar{\Box} - \xi^J) \log (e^{\theta_\xi} \det g) + \sum_{i=1}^n \partial_i \xi^i, 
\end{align}
using $\theta_\xi = - 2 \mu_\xi$. 
Note that $\xi^J \log \det g - \sum \partial_i \xi^i = -\bar{\Box} \theta_\xi$ is a globally-defined function while $\xi^J \log \det g$ is just locally-defined on a holomorphic chart. 

Put 
\begin{align}
\bar{s}_\xi (J) 
&:= \int_X s_\xi (g_J) e^{-2\mu_\xi} \omega^n \Big{/} \int_X e^{-2 \mu_\xi} \omega^n 
\\ \notag
&= \int_X (s (g_J) - \Delta_{g_J} \mu_\xi) e^{-2 \mu_\xi} \omega^n \Big{/} \int_X e^{-2 \mu_\xi} \omega^n. 
\end{align}
A similar calculation as in the proof of Proposition 3.1 in \cite{Ino} shows
\begin{align*}
\frac{d}{dt} \bar{s}_\xi (J_t) 
&= \Big{(} \frac{1}{4} \frac{d}{dt} (4 s (J_t), e^{-2 \mu_\xi}) + \int_X 2 |\xi|^2_{g_J} \Big{)} \Big{/} \int_X e^{-2 \mu_\xi} \omega^n
\\
&= \Big{(} \frac{1}{4} (L_{-2 e^{-2 \mu_\xi} \xi} J_t, J_t \dot{J}_t) + \int_X 2 \omega (\xi, \dot{J}_t \xi) e^{-2 \mu_\xi} \omega^n \Big{)} \Big{/} \int_X e^{-2 \mu_\xi} \omega^n
\\
&= \int_X \Big{(} (-J d\mu_\xi \otimes \xi + d \mu_\xi \otimes J \xi, J_t \dot{J}_t) + 2 \omega (\xi, \dot{J}_t \xi) \Big{)} e^{-2 \mu_\xi} \omega^n \Big{/} \int_X e^{-2 \mu_\xi} \omega^n
\\
&= 0. 
\end{align*}
So $\bar{s}_\xi (J)$ is a constant independent of $J$ compatible with $\omega$. 
For an integrable complex structure $J$, we can compute it as 
\begin{align*} 
\bar{s}_\xi 
&= \Big{(} \int_X n e^{-2 \mu_\xi} \Ric (\omega) \wedge \omega^{n-1} + \int_X \bar{\Box} (-2 \mu_\xi) e^{-2 \mu_\xi} \omega^n \Big{)} \Big{/} \int_X e^{-2 \mu_\xi} \omega^n
\\ 
&= \int_X (\Ric (\omega) + \bar{\Box} \theta_\xi) ~e^{\omega + \theta_\xi} \Big{/} \int_X e^{\omega + \theta_\xi}
\\ 
& = 2 \pi (c_1 (X, \xi) \cdot e^{c_1 (L, \xi)})/e^{c_1 (L, \xi)}, 
\end{align*}
where the last expression depends only on the equivariant Chern classes. 


For each $\lambda \in \mathbb{R}$, we define the \textit{$\mu^\lambda$-scalar curvature} $s^\lambda_\xi (g_J)$ of a metric $g_J$ by 
\begin{equation}
s^\lambda_\xi (g_J) = s_\xi (g_J) + 2 \lambda \mu_\xi. 
\end{equation}
We put 
\begin{equation}
\bar{\mu}_\xi := \int_M \mu_\xi e^{-2\mu_\xi} \omega^n \Big{/} \int_M e^{-2\mu_\xi} \omega^n
\end{equation}
and 
\begin{align}
\hat{\mu}_\xi 
&:= \mu_\xi - \bar{\mu}_\xi, 
\\
\bar{s}^\lambda_\xi 
&:= \bar{s}_\xi + 2 \lambda \bar{\mu}_\xi, 
\\
\hat{s}^\lambda_\xi (g_J)
&:= s^\lambda_\xi (g_J) - \bar{s}^\lambda_\xi. 
\end{align}
Then the constant $\bar{s}^\lambda_\xi$ depends only on the equivariant Chern classes $c_1^T (X), c_1^T (L)$ and $s^\lambda_\xi (g_J)$ is constant iff $\hat{s}^\lambda_\xi (g_J) = 0$. 

\subsubsection{Relation with K\"ahler-Ricci soliton}
\label{KR soliton}

There are two fundamental examples of constant $\mu$-scalar curvature K\"ahler metric: 
\begin{itemize}
\item A constant scalar curvature K\"ahler metric is also a constant $\mu$-scalar curvature K\"ahler metric with respect to $\xi = 0$ and any $\lambda \in \mathbb{R}$. 

\item A K\"ahler-Ricci soliton $g_J$ with respect to $\xi$, i.e. $\Ric (g_J) - L_{\xi^J} g_J = \lambda g_J$, is a constant $\mu$-scalar curvature K\"ahler metric with respect to $\xi$ and $\lambda$. 
\end{itemize}

The second claim follows from a standard calculation in \cite{TZ} (cf. \cite{Ino}). 
For the readers' convenience, we exhibit the proof here. 
Remember that K\"ahler-Ricci soliton with nontrivial $\xi \neq 0$ could exist only when $\lambda > 0$ and $[\lambda \omega] \in 2 \pi c_1 (X)$. 
In particular, $X$ is a Fano manifold in this case. 
Take a Ricci potential $h$ of $\omega$, i.e. $\Ric (\omega) - \lambda \omega = \sqddbar h$, and consider a moment map $\mu$ with respect to $\omega$ normalized as 
\begin{equation} 
\label{Fano moment}
\int_X \mu e^h \omega^n = 0. 
\end{equation} 
Taking the Lie derivative $L_{\xi^J}$ of $\Ric (\omega) - \lambda \omega = \sqddbar h$, we have 
\[ \sqddbar (\bar{\Box} \theta_\xi -\lambda \theta_\xi) = \sqddbar ((\dbar^\sharp h) \theta_\xi), \]
where we used that $\xi^J$ is holomorphic and $(\dbar^\sharp h) \theta_\xi = \xi^J h$. 
Note that the operator $\bar{\Box} - \dbar^\sharp h$ is formally self-adjoint with respect to the weighted measure $e^h \omega^n$, therefore $(\bar{\Box} - \dbar^\sharp h) f = \varphi$ has a solution $f$ (unique up to constant) iff $\int_X \varphi e^h \omega^n = 0$. 
So under the normalization (\ref{Fano moment}), we obtain 
\begin{equation} 
\label{Fano moment 2}
\bar{\Box} \theta_\xi - \xi^J h - \lambda \theta_\xi = 0. 
\end{equation}
Then we can express $\bar{s}_\xi$ as 
\begin{align}
\notag 
\bar{s}_\xi 
&= \int_X (\bar{\Box} (- h + \theta_\xi) + \lambda n) e^{\theta_\xi} \omega^n \Big{/} \int_X e^{\theta_\xi} \omega^n
\\ \notag
&= \lambda n + \int_X (-\bar{\Box} h + \xi^J h + \lambda \theta_\xi) e^{\theta_\xi} \omega^n \Big{/} \int_X e^{\theta_\xi} \omega^n
\\ \label{Fano s bar}
&= \lambda n + \lambda \int_X \theta_\xi e^{\theta_\xi} \omega^n \Big{/} \int_X e^{\theta_\xi} \omega^n, 
\end{align}
where we again used that $\bar{\Box} - \xi^J$ is formally self-adjoint with respect to the weighted measure $e^{\theta_\xi} \omega^n$. 
Now suppose $\omega$ is a K\"ahler-Ricci soliton, then taking the trace of $\Ric (g_J) - L_{\xi'} g_J = \lambda g_J$, we obtain 
\[ s (g_J) + \bar{\Box} \theta_\xi = \lambda n. \]
As $h$ is equal to $\theta_\xi$ up to constant, we have $\bar{\theta}_\xi = \int_X \theta_\xi e^{\theta_\xi} \omega^n = 0$ under the normalization (\ref{Fano moment}) and $\bar{\Box} \theta_\xi - \xi^J \theta_\xi - \lambda \theta_\xi = 0$. 
Therefore, we conclude
\[ s_\xi (g_J) - \lambda \theta_\xi = (s (g_J) + \bar{\Box} \theta_\xi) + (\bar{\Box} \theta_\xi - \xi^J \theta_\xi) - \lambda \theta_\xi = \bar{s}_\xi. \]

The normalization (\ref{Fano moment}) of the moment map $\mu$ is equivalent to $[\omega + \mu] = c_1^T (X)$ where $c_1^T (X)$ denotes the equivariant Chern class of the anticanonical bundle $-K_X$, which can be represented by the equivariant closed form $\Ric (\omega) + \bar{\Box} \theta$ in the equivariant deRham cohomology. 

\subsubsection{Donaldson-Fujiki picture for $\mu$-scalar curvature}

Now we explain the moment map picture for $\mu$-scalar curvature. 
Let $(M, \omega)$ be a real $2n$-dimensional $C^\infty$-symplectic manifold. 
Denote by $\mathcal{J}_\xi (M, \omega)$ the space of all $\xi$-invariant almost complex structures compatible with $\omega$, which admits the structure of an infinite dimensional Fr\'echet manifold and is path-connected. 
We have the following symplectic structure $\Omega_\xi$ on $\mathcal{J}_\xi (M, \omega)$: 
\begin{equation}
\Omega_\xi (A, B) := \int_M \mathrm{Tr} (JA B) e^{-2 \mu_\xi} \omega^n
\end{equation}
for each $A, B \in T_J \mathcal{J}_\xi (M, \omega) \subset \mathrm{End} TM$. 

For simplicity, we assume the first Betti number of $M$ is zero. 
In this case, we can identify the Lie algebra $\mathfrak{symp}_\xi (M, \omega)$ of the Fr\'echet Lie group $\mathrm{Symp}_\xi (M, \omega)$ of symplectic diffeomorphisms preserving $\xi$ with the space $C^\infty_\xi (M)/\mathbb{R}$ of real-valued $\xi$-invariant $C^\infty$-functions on $M$ modulo constant. 
We identify a $2n$-form $\varphi$ on $M$ satisfying $\int_M \varphi = 0$ and $L_\xi \varphi = 0$ with the following element of the dual of $\mathfrak{symp}_\xi (M, \omega)$: $f \mapsto \int_M f \varphi$. 

Now define a smooth map $\mathcal{S}^\lambda_\xi: \mathcal{J}_\xi (M, \omega) \to \mathfrak{symp}_\xi (M, \omega)^*$ of Fr\'echet manifolds by 
\begin{equation} 
\mathcal{S}^\lambda_\xi (J) :=  4 \hat{s}^\lambda_\xi (g_J) e^{-2\mu_\xi} \omega^n. 
\end{equation}
Then we have the following. 

\begin{prop}[\cite{Ino}]
The map $\mathcal{S}^\lambda_\xi :\mathcal{J}_\xi (M, \omega) \to \mathfrak{symp}_\xi (M, \omega)^*$ is a moment map with respect to the symplectic structure $\Omega_\xi$ and the action of $\mathrm{Symp}_\xi (M, \omega)$ on $\mathcal{J}_\xi (M, \omega)$. 
Namely, $\mathcal{S}^\lambda_\xi$ is a $\mathrm{Symp}_\xi (M, \omega)$-equivariant smooth map satisfying
\begin{equation} 
- \frac{d}{dt}\Big{|}_{t=0} \langle \mathcal{S}^\lambda_\xi (J_t) , f \rangle = \Omega_\xi (L_{X_f} J_0, \dot{J}_0) 
\end{equation}
for every smooth curve $J_t \in \mathcal{J}_\xi (M, \omega)$ and $f \in C^\infty_\xi (M)$, where $X_f$ is the Hamiltonian vector field of $f$: $df = -i_{X_f} \omega$. 
\end{prop}

Note that moment maps with respect to the symplectic structure $\Omega_\xi$ is unique up to $\mathrm{Symp}_\xi (M, \omega)$-invariant elements of $\mathfrak{symp}_\xi (M, \omega)$. 
In particular, the map $J \mapsto (\hat{s}_\xi (g_J) + \mu_\zeta - \int_M \mu_\zeta e^{-2 \mu_\xi} \omega^n / \int_M e^{-2\mu_\xi} \omega^n) e^{-2\mu_\xi} \omega^n$ also gives a moment map for any $\zeta$ tangent to the action of the closed torus generated by $\xi$. 
In this article, we restrict our interest to the proportional one, i.e. $\zeta = -2\lambda \xi$ for some $\lambda \in \mathbb{R}$. 

The following invariant gives a constraint on $\lambda$ for each fixed $\xi$ and conversely a constraint on $\xi$ for each fixed $\lambda$ for the non-emptiness of the moduli space $(\mathcal{S}^\lambda_\xi)^{-1} (0)/\mathrm{Symp}_\xi (M, \omega)$. 
We will study these constraints in the next section and section \ref{section: volume}, respectively. 

\begin{cor}[$\mu$-Futaki invariant]
Let $\mathfrak{t}$ be the Lie algebra of the closed torus generated by $\xi$. 
The following linear map $\Fut^\lambda_\xi: \mathfrak{t} \to \mathbb{R}$, 
\begin{align*}
\Fut_\xi^\lambda (\zeta) 
&:= \int_M \hat{s}^\lambda_\xi (g_J) (-2 \mu_\zeta) e^{- 2 \mu_\xi} \omega^n \Big{/} \int_M e^{- 2 \mu_\xi} \omega^n
\end{align*}
is independent of the choice of $J \in \mathcal{J}_\xi (M, \omega)$ and the moment map $\mu$ (as we divide it by $\int_M e^{-2 \mu_\xi} \omega^n$). 
\end{cor}

If we fix a complex structure $J$, it is independent of the choice of the K\"ahler metric $\omega'$ in the K\"ahler class $[\omega]$ (by Moser's theorem). 
So in particular, $\Fut^\lambda_\xi$ can be regarded as an invariant of the quadruple $(X, [\omega], \xi, \lambda)$ where $X = (M, J)$ is a complex manifold. 
As observed in \cite{Wang}, the moment map picture further expects that $\Fut^\lambda_\xi$ extends to $\mathfrak{h}_{0, \xi} (X)$. 
In section \ref{section: Futaki}, we further show that $\Fut^\lambda_\xi$ extends to $\mathfrak{h}_0 (X)$, which is larger than $\mathfrak{h}_{0, \xi} (X)$. 
Such an extension is out of expectations coming from the moment map picture. 

Note that the above corollary also shows that this complex invariant $\Fut^\lambda_\xi$ (restricted to $\mathfrak{t}$) is also a $T$-equivariant deformation invariant. 

\subsubsection{Weighted cscK metrics and $\mu$-cscK metrics}

For a smooth positive function $v$ on $P$, Lahdili \cite{Lah} defines the \textit{weighted scalar curvature} $s_v (\omega)$ by 
\begin{equation} 
s_v (\omega) := s (\omega) \cdot (v \circ \mu^\omega) + \Delta_\omega (v \circ \mu^\omega) - \frac{1}{2} \sum_{1 \le i, j \le k} (J\xi_i) \mu^\omega_{\xi_j} \cdot (\frac{\partial^2 v}{\partial x^i \partial x^j} \circ \mu^\omega). 
\end{equation}
As observed in \cite{Lah}, weighted scalar curvature has a moment map picture similar to that for $\mu$-scalar curvature in the previous section. 

When $v$ is of the form $v (x) = \tilde{v} (\langle x, \xi \rangle)$ with some smooth positive function $\tilde{v}$ on $\mathbb{R}$ and $\xi \in \mathfrak{t}$, we can simplify it as 
\begin{align*} 
s_v (\omega) 
&= s (\omega) \cdot (\tilde{v} \circ \mu^\omega_\xi) + \big{(} \Delta_\omega \mu^\omega_\xi \cdot (\tilde{v}' \circ \mu^\omega_\xi) - (\nabla \mu^\omega_\xi, \nabla \mu^\omega_\xi) \cdot (\tilde{v}'' \circ \mu^\omega_\xi) \big{)} - \frac{1}{2} (J\xi) \mu^\omega_{\xi} \cdot (\tilde{v}'' \circ \mu^\omega_\xi) 
\\
&= s (\omega) \cdot (\tilde{v} \circ \mu^\omega_\xi) + \Delta_\omega \mu^\omega_\xi \cdot (\tilde{v}' \circ \mu^\omega_\xi) + \frac{1}{2} (J\xi) \mu^\omega_{\xi} \cdot (\tilde{v}'' \circ \mu^\omega_\xi) 
\end{align*}
Substituting $v (x) = e^{\langle x, -2 \xi \rangle }$ yields our $\mu$-scalar curvature $s_\xi (\omega)$: 
\[ s_v (\omega)  = \big{(} (s (\omega) + \bar{\Box} \theta_\xi) + (\bar{\Box} \theta_\xi - (J\xi) \theta_\xi) \big{)} e^{\theta_\xi} =: s_\xi (\omega) e^{\theta_\xi}. \]
So $\mu^0_\xi$-cscK metrics are equivalent to weighted cscK metrics with the weight $v (x) = e^{\langle x, -2 \xi \rangle}$. 
For general $\lambda \in \mathbb{R}$, $\mu^\lambda_\xi$-cscK metrics are regarded as a special case of weighted extremal metrics. 

\subsection{From $\mu$-cscK metrics to extremal metric: the limit $\lambda \searrow - \infty$}
\label{section: connecting path}

In this section, we fix a complex structure $J$ on $M$ and a K\"ahler class $[\omega]$. 
We observe some intriguing features of $\mu^\lambda$-cscK, assuming some results in the rest of this article. 

\subsubsection{Constraint on $\lambda$}

There is an a priori constraint on $\lambda$ for each fixed $\xi \neq 0$ to admit a $\mu$-cscK metric in a fixed K\"ahler class $[\omega]$. 
If there is a $\mu^\lambda_\xi$-cscK metric $\omega'$, i.e. $\hat{s}^\lambda_\xi (\omega') =0$, in the K\"ahler class $[\omega]$, then we must have 
\begin{equation}
\label{Fut lambda} 
0 =  \Fut_\xi^\lambda (\xi) = \Fut_\xi^0 (\xi) - \lambda \left( \int_X \theta_\xi^2 e^{\theta_\xi} \omega^n \Big{/} \int_X e^{\theta_\xi} \omega^n - \Big{(} \int_X \theta_\xi e^{\theta_\xi} \omega^n \Big{/} \int_X e^{\theta_\xi} \omega^n \Big{)}^2 \right). 
\end{equation}

For $\zeta \in \mathfrak{t}$, we put 
\begin{equation} 
\label{nu}
\nu_\xi (\zeta) := \int_X \theta_\zeta^2 e^{\theta_\xi} \omega^n \Big{/} \int_X e^{\theta_\xi} \omega^n - \Big{(} \int_X \theta_\zeta e^{\theta_\xi} \omega^n \Big{/} \int_X e^{\theta_\xi} \omega^n \Big{)}^2. 
\end{equation}
This is invariant when we add a constant $c$ on $\theta_\zeta$, so it must be positive when $\zeta \neq 0$ since it is obviously positive when normalizing $\theta_\zeta$ so that $\int_X \theta_\zeta e^{\theta_\xi} \omega^n = 0$. 
The function $\nu_\xi$ is an invariant of the equivariant deRham class $[\omega + \mu]$ and $\xi$, since it can be expressed as 
\[ \nu_\xi (\zeta) = \frac{\int_P \langle m, -2\zeta \rangle^2 e^{\langle m, -2\xi \rangle} DH (m)}{\int_P e^{\langle m, -2\xi \rangle} DH (m)} - \Big{(} \frac{\int_P \langle m, -2\zeta \rangle e^{\langle m, -2\xi \rangle} DH (m)}{\int_P e^{\langle m, -2\xi \rangle} DH (m)} \Big{)}^2, \] 
using the Duistermaat-Heckman measure $DH = \mu_* \omega^n$, which is an invariant of the equivariant deRham class $[\omega + \mu]$ associated to the moment map. 
Here $P$ denotes the support of the measure $DH$. 

Thus from (\ref{Fut lambda}) we can determine $\lambda$ as 
\begin{equation} 
\lambda = \lambda_\xi := \Fut^0_\xi (\xi)/\nu_\xi (\xi), 
\end{equation}
where the right hand side is an invariant of the triple $(X, [\omega], \xi)$ (also an invariant of the symplectic triple $(M, \omega, \xi)$). 
The sign of $\lambda_\xi$ coincides with that of $\Fut^0_\xi (\xi)$. 

\subsubsection{$\lambda$ as a function on the real blowing-up $\hat{\mathfrak{t}}$}
\label{lambda as a function on t}

While the function $\lambda_\xi$ is well-defined and continuous just on the punctured space $\mathfrak{t} \setminus \{ 0 \}$, the following functional 
\begin{equation}
\xi \mapsto |\xi| \cdot \lambda_\xi 
\end{equation} 
continuously extends to the real blowing-up 
\[ \hat{\mathfrak{t}} := \{ (\xi, \Xi) ~|~ \xi \in \Xi = [0, \infty) \cdot v \subset \mathfrak{t}, ~v \in \mathfrak{t} \setminus \{ 0 \} \} \xrightarrow{\pi} \mathfrak{t}: (\xi, \Xi) \mapsto \xi \] 
of $\mathfrak{t}$ at the origin, where we take the norm on $\mathfrak{t}$ as $|\xi|^2 := \int_X \theta_\xi^2 \omega^n$. 
Indeed, as $|\xi|$ tends to $0$, the function $|\xi|^{-2} \nu_\xi (\xi)$ on $\mathfrak{t} \setminus \{ 0 \}$ approaches to a positive continuous function $\hat{\nu} (0, \Xi) = 1/\int_X \omega^n - (\int_X \theta_\Xi \omega^n/\int_X \omega^n)^2$ on the boundary sphere $\pi^{-1} (0)$, where we put $\theta_\Xi := \theta_v$ for a unique vector $v \in \Xi$ with $|v| = 1$ and similarly $|\xi|^{-1} \cdot \Fut^0_\xi (\xi)$ approaches to a continuous function $\widehat{\Fut} (0, \Xi) = \Fut (v) = \int_X (s - \bar{s}) \theta_\Xi \omega^n /\int_X \omega^n$ on $\pi^{-1} (0)$. 
Here the positivity of $\hat{\nu}$ again follows by the Cauchy-Schwartz inequality. 

We will see in section \ref{section: volume} that $\lambda_\xi$, i.e. $\Fut^0_\xi (\xi)$, is always positive sufficiently away from the origin. 
Assuming this, it follows that any sequence $\xi_i \in \mathfrak{t}$ with $\lambda_{\xi_i} \to - \infty$ must converge to the origin $0 \in \mathfrak{t}$. 
Moreover, as the function $|\xi| \cdot \lambda_\xi$ is bounded near the origin, we have a uniform bound $|\lambda_i \xi_i| \le C$, so that there is a subsequence such that $\lambda_i \xi_i$ converges to some vector $\check{\xi} \in \mathfrak{t}$. 
Now suppose $\Fut^{\lambda_i}_{\xi_i} = 0$ for every $i$. 
Since we can compute all $\Fut^{\lambda_i}_{\xi_i}$ by a fixed $T$-invariant K\"ahler metric $\omega$, the limit of this functional is given as 
\[ \check{\Fut}^0_{\check{\xi}} (\zeta) := \int_X \Big{(} (s (\omega) - \underline{s}) - (\theta_{\check{\xi}} - \underline{\theta}_{\check{\xi}}) \Big{)} \theta_\zeta \omega^n \Big{/} \int_X \omega^n, \]
where we put $\underline{\theta}_{\check{\xi}} := \int_X \theta_{\check{\xi}} \omega^n /\int_X \omega^n$. 
We must have $\check{\Fut}^0_{\check{\xi}} \equiv 0$ for the limit vector $\check{\xi}$. 

Such a vector $\check{\xi}$ is uniquely characterized as the critical point of the following strictly convex functional on $\mathfrak{t}$: 
\begin{equation}
C (\xi) := \int_X \Big{(} (s (\omega) - \underline{s}) - (\theta_\xi - \underline{\theta}_\xi) \Big{)}^2 \omega^n \Big{/} \int_X \omega^n - \int_X (s (\omega) - \underline{s})^2 \omega^n \Big{/} \int_X \omega^n, 
\end{equation}
whose derivative at $\xi$ is $2 \check{\Fut}^0_\xi$. 
(We add the second term so that the functional is independent of the choice $\omega \in [\omega]$. )
The minimizer of this functional is called \textit{the extremal vector}. 
We denote it by $\xi_{\mathrm{ext}}$. 
From the above observation, we obtain $\check{\xi} = \xi_{\mathrm{ext}}$ for the limit vector $\check{\xi}$, independent of the choice of the subsequence of $\{ i \}$. 
It follows that the original sequence $\lambda_i \xi_i$ also converges to $\xi_{\mathrm{ext}}$.

\subsubsection{Extremal metric in the limit of $\lambda \to - \infty$}

Suppose there is a sequence of $\mu^{\lambda_i}_{\xi_i}$-cscK metrics $\omega_i$ in the fixed K\"ahler class $[\omega]$ with $\lambda_i \to - \infty$: 
\[ (s (\omega_i) + \bar{\Box}_{\omega_i} \theta_{\xi_i} (\omega_i)) + (\bar{\Box}_{\omega_i} \theta_{\xi_i} (\omega_i) - \xi_i^J \theta_{\xi_i} (\omega_i)) - \lambda_i \theta_{\xi_i} (\omega_i) = \bar{s}_{\xi_i}^{\lambda_i}, \]
where $\theta_{\xi_i} (\omega_i)$ denotes the $\dbar$-Hamiltonian potential with respect to $\omega_i$ in the same equivariant class. 
Fix a reference metric $\omega$ and take a K\"ahler potential $\phi_i$ of $\omega_i$ so that $\max \phi_i = 0$. 

Suppose we have a uniform $C^{3, \alpha}$-bound of $\phi_i$ and a uniform bound $C \omega \le \omega_i$, then the limit of the metrics gives a metric $\omega_{-\infty} \in [\omega]$ after taking a subsequence. 
Remember that the vectors $\xi_i$ must converge to $0$ and the sequence $\lambda_i \xi_i$ converges to the extremal vector $\xi_{\mathrm{ext}}$ (by the observation in the last subsection). 
It follows that $\theta_{\xi_i} (\omega_i) = \theta_{\xi_i} (\omega) - \xi_i^J \phi_i$ converges to $0$ in $C^{2, \alpha}$ and the limit metric $\omega_{-\infty}$ must satisfy the following equation 
\[ s (\omega_{-\infty}) - \theta_{\xi_{\mathrm{ext}}} (\omega_{-\infty}) = \mathrm{const}, \]
which is nothing but the equation of extremal metric. 

Conversely, we will see in section 5 the following: 
\begin{itemize}
\item If there exists an extremal metric, there also exists $\mu^\lambda$-cscK metrics in the same K\"ahler class for $\lambda$ sufficiently small or large. 

\item If there is a $\mu^\lambda$-cscK metric for $\lambda \le 0$, then we can find a $\mu^{\lambda'}$-cscK metric in the same K\"ahler class for small perturbations $\lambda' \in (\lambda - \epsilon, \lambda + \epsilon)$. 
\end{itemize}
Thus the problem of connecting $\mu^0$-cscK metric/K\"ahler--Ricci soliton and extremal metric (when both of them exist) reduces to the problem on the a priori estimate. 

Though we firstly introduced the parameter $\lambda$ so that we can include K\"ahler-Ricci soliton in our study on $\mu$-cscK metric, the above observation now tells us that the parameter $\lambda$ can be regarded as a continuity path connecting $\mu^0$-cscK metric/K\"ahler--Ricci soliton and extremal metric. 

\section{$\mu$-Futaki invariant, $\mu$-volume functional and automorphism group}

\subsection{$\mu$-Lichnerowicz operator and reductiveness}
\label{section: reductive}

In this section, we fix a complex structure $J$ on $M$, a K\"ahler metric $\omega$ and a function $\theta$ on $M$. 
We prove the reductiveness of the automorphism group of a K\"ahler manifold admitting $\mu$-cscK. 
This result is a first step to construct a good moduli space of the complex structures of K\"ahler manifolds admitting $\mu$-cscK metrics, in order to apply GIT locally. 
We firstly begin with basic calculations for the readers' and the author's convenience. 

\subsubsection{Warming up for calculations}

Let $(X, \omega)$ be a K\"ahler manifold, $\theta$ be a smooth real-valued function on $X$ and $(E, h)$ be a hermitian (not necessarily holomorphic, so far) vector bundle on $X$. 
Define an $L^2$-norm $\langle \cdot, \cdot \rangle_\theta$ by
\[ \langle \alpha, \beta \rangle_\theta := \int_X h (\alpha, \beta) ~e^\theta \omega^n \]
for smooth sections $\alpha, \beta \in \Omega^0 (E)$. 
For a differential operator $D: \Omega^0 (E) \to \Omega^0 (F)$ from $E$ to $F$, denote by $D^{\theta *}: \Omega^0 (F) \to \Omega^0 (E)$ the formal left adjoint of $D$ with respect to such pairing, i.e. 
\[ \langle D^{\theta *} \alpha, \beta \rangle_{E, \theta} = \langle \alpha, D \beta \rangle_{F, \theta} \]
for all sections $\alpha \in \Omega^0 (F), \beta \in \Omega^0 (E)$. 
As usual, we denote by $\Lambda: \Omega^{p,q} (E) \to \Omega^{p-1, q-1} (E)$ the adjoint operator of $\omega \wedge$: 
\[ h^{p-1,q-1} (\Lambda (\alpha), \beta) = h^{p,q} (\alpha, \omega \wedge \beta), \]
where $h^{p,q}$ is the induced hermitian metric on $\Lambda^{p,q} \otimes E$ defined as 
\[ (u_{1 \dotsb p} \wedge u_{\bar{1} \dotsb \bar{q}} \otimes \sigma, v_{1 \dotsb p} \wedge v_{\bar{1} \dotsb \bar{q}} \otimes \tau) := 
 h (\sigma, \tau) \cdot \det \begin{pmatrix} g (u_i, \overline{v_j}) & 0 \\ 0 & g (u_{\bar{k}}, \overline{v_{\bar{l}}}) \end{pmatrix}^{i, j=1, \ldots, p}_{k, l= 1, \ldots, q} \]
for $u_{1 \dotsb p} = u_1 \wedge \dotsb \wedge u_p, v_{1 \dotsb p} = v_1 \wedge \dotsb \wedge v_p \in \Lambda^{p, 0}_x X$, $u_{\bar{1} \dotsb \bar{q}} = u_{\bar{1}} \wedge \dotsb \wedge u_{\bar{q}}, v_{\bar{1} \dotsb \bar{q}} = v_{\bar{1}} \wedge \dotsb \wedge v_{\bar{q}} \in \Lambda^{0,q}_x X$ and $\sigma, \tau \in E_x$. 

\begin{eg}
For $\alpha = \alpha_{i \bar{j}} \otimes dz^i \wedge d\bar{z}^j \in \Omega^{1,1} (E)$, we have 
\[ \Lambda (\alpha) = - \sqrt{-1} g^{i \bar{j}} \alpha_{i \bar{j}}. \]
For $\gamma = \gamma_{i \bar{j} \bar{k}} \otimes dz^i \wedge d\bar{z}^j \wedge d\bar{z}^k \in \Omega^{1,2} (E)$, we have 
\[ \Lambda (\gamma) = - \sqrt{-1} g^{i \bar{j}} (\gamma_{i \bar{j} \bar{k}} - \gamma_{i \bar{k} \bar{j}}) d \bar{z}^k. \]
\end{eg}

For a hermitian connection $\nabla$ on $(E, h)$, the following local expressions yield global operators. 
\begin{align} 
\nabla' 
&:= \sum_{i=1}^n dz^i \wedge \nabla^{\wedge}_{\partial_i} : \Omega^{p,q} (E) \to \Omega^{p+1, q} (E), 
\\
\nabla'' 
&:= \sum_{i=1}^n d\bar{z}^i \wedge \nabla^{\wedge}_{\bar{\partial}_i} :\Omega^{p,q} (E) \to \Omega^{p,q+1} (E), 
\end{align}
where 
\[ \nabla^{\wedge}: \Omega^0 (\Lambda^{p,q} \otimes E) \to \Omega^1 (\Lambda^{p,q} \otimes E) \]
is the induced connection on $\Lambda^{p,q} \otimes E$. 
These operators $\nabla', \nabla''$ are the first order differential operators from $\Lambda^{p,q} \otimes E$ to $\Lambda^{p+1, q} \otimes E$ and $\Lambda^{p,q} \otimes E$ to $\Lambda^{p, q+1} \otimes E$, respectively, and $\nabla' + \nabla''$ is the exterior covariant derivative of $\nabla$. 

Put $\theta_p := \partial \theta/\partial z^p, \theta_{\bar{q}} := \partial \theta/\partial \bar{z}^q$ on a holomorphic chart of $X$ and denote by $\xi', \xi''$ the following global vector fields associated to $\theta$
\[ \xi' := \partial^\sharp \theta = g^{p \bar{q}} \theta_{\bar{q}} \partial_p, \quad \xi'' := \bar{\partial}^\sharp \theta = g^{p \bar{q}} \theta_p \bar{\partial}_q. \]
Then the formal adjoints ${\nabla'}^{\theta*}, {\nabla''}^{\theta*}$ of $\nabla', \nabla''$ with respect to the pairing $\langle \cdot, \cdot \rangle_\theta$ can be written as 
\begin{align}
\label{weighted adjoint 1}
{\nabla'}^{\theta*} 
&= {\nabla'}^* - i_{\xi'} = \sqrt{-1} \Big{(} \Lambda \nabla'' - \nabla'' \Lambda \Big{)} - i_{\xi'}, 
\\ \label{weighted adjoint 2}
{\nabla''}^{\theta*}
&= {\nabla''}^* - i_{\xi''} = - \sqrt{-1} \Big{(} \Lambda \nabla' - \nabla' \Lambda \Big{)} - i_{\xi''}. 
\end{align}
Indeed, using $\langle \alpha, \beta \rangle_\theta = \langle \alpha, \beta e^\theta \rangle$, we compute 
\begin{align*}
\langle {\nabla'}^{\theta*} \alpha, \beta \rangle_\theta 
&= \langle \alpha, \nabla' \beta \rangle_\theta 
\\ 
&= \langle \alpha, \nabla' (e^\theta \beta) \rangle - \langle \alpha, \partial \theta \wedge \beta \rangle_\theta 
\\
&= \langle {\nabla'}^* \alpha, \beta \rangle_\theta - \langle i_{\xi'} \alpha, \beta \rangle_\theta. 
\end{align*}

\subsubsection{Weitzenb\"ock formula}

Let  $(L, e^{-\phi})$ be a holomorphic hermitian line bundle on $X$, where we denote the hermitian connection by a local expression $e^{-\phi}$. 
We denote by $\nabla$ the Chern connection on $L$. 
The Chern curvature is given by $\partial \dbar \phi$. 
Put 
\begin{align*} 
\nabla^\sharp 
&:= \nabla_{T^{1,0} X \otimes L}' \circ \sharp: \Omega^{0,1} (L) \to \Omega^{1,0} (T^{1,0} X \otimes L), 
\\
\nabla^{\sharp \sharp} 
&:= \nabla_{T^{1,0} X \otimes L}'' \circ \sharp: \Omega^{0,1} (L) \to \Omega^{0,1} (T^{1,0} X \otimes L), 
\end{align*}
where $\sharp: \Omega^{0,1} (L) \to \Omega^0 (T^{1,0} X \otimes L)$ is given by $\sharp (\alpha_{\bar{j}} d \bar{z}^j) = g^{i \bar{j}} \alpha_{\bar{j}} \partial_i$. 
Consider the following four variants of weighted Laplacian acting on $\Omega^{0,1} (L)$: 
\begin{align*} 
\Box^\theta
&:= {\nabla'}^{\theta *} \nabla' + \nabla'~ {\nabla'}^{\theta*} = {\nabla'}^{\theta*} \nabla', 
\\
\bar{\Box}^\theta 
&:= {\nabla''}^{\theta*} \nabla'' + \nabla'' ~ {\nabla''}^{\theta*}, 
\\
\Box^{\theta}_\#
&:= {\nabla^\sharp}^{\theta*} \nabla^\sharp = \flat ({\nabla'}_{TX \otimes L}^{\theta*} \nabla'_{TX \otimes L}) \sharp, 
\\
\bar{\Box}^{\theta}_\# 
&:= {\nabla^{\sharp \sharp}}^{\theta*} \nabla^{\sharp \sharp} = \flat ({\nabla''}_{TX \otimes L}^{\theta*} \nabla''_{TX \otimes L}) \sharp, 
\end{align*}
where $\flat: \Omega^0 (T^{1,0} X \otimes L) \to \Omega^{0,1} (L)$ is given by $\flat (\eta^i \partial_i) = g_{i \bar{j}} \eta^i d \bar{z}^j$. 

\begin{lem}[Weighted Laplacians]
The above weighted Laplacians can be expressed by the usual Laplacians as follows. 
\begin{align}
\Box^\theta 
&= \Box - \nabla_{\xi'}^\wedge, 
\\
\bar{\Box}^\theta
&= \bar{\Box} - \nabla_{\xi''}^\wedge - g^{i \bar{j}} \theta_{i \bar{k}} d\bar{z}^k \otimes \bar{\partial}_j, 
\\
\Box^{\theta}_\# 
&= \Box_\# - \nabla_{\xi'}^\wedge, 
\\
\bar{\Box}^{\theta}_\#
&= \bar{\Box}_\# - \nabla_{\xi''}^\wedge, 
\end{align}
where $\theta_{i \bar{k}} = \partial^2 \theta/\partial z^i \partial \bar{z}^k$ and $g^{i \bar{j}} \theta_{i \bar{k}} d\bar{z}^k \otimes \bar{\partial}_j \in \mathrm{End} (T^{0,1} X)$ is identified with the operator acting on $\Omega^{0, 1} (L)$. 
\end{lem}

\begin{proof}
Let $\alpha_{\bar{k}} \otimes d \bar{z}^k$ be an element of $\Omega^{0,1} (L)$ expressed by local sections $\alpha_{\bar{k}}$ of $L$. 
Then from (\ref{weighted adjoint 1}), we have
\begin{align*}
\Box^\theta (\alpha_{\bar{k}} \otimes d \bar{z}^k) 
&= {\nabla'}^{\theta*} \nabla' (\alpha_{\bar{k}} \otimes d \bar{z}^k)
\\ 
&= \Big{(} {\nabla'}^* \nabla' - i_{\xi'} \nabla' \Big{)} (\alpha_{\bar{k}} \otimes d \bar{z}^k)
\\
&= (\Box - \nabla_{\xi'}^\wedge) (\alpha_{\bar{k}} \otimes d \bar{z}^k). 
\end{align*}
We can do similarly as for $\Box^\theta_\#$ and $\bar{\Box}^\theta_\#$. 

As for $\bar{\Box}^\theta$, we calculate as follows. 
\begin{align*}
\bar{\Box}^\theta (\alpha_{\bar{k}} \otimes d \bar{z}^k)
&= ({\nabla''}^{\theta*} \nabla'' + \nabla'' {\nabla''}^{\theta*}) (\alpha_{\bar{k}} \otimes d \bar{z}^k) 
\\ 
&= \Big{(} ({\nabla''}^* \nabla'' + \nabla'' {\nabla''}^*) - (i_{\xi''} \nabla'' + \nabla'' i_{\xi''}) \Big{)} (\alpha_{\bar{k}} \otimes d \bar{z}^k)
\\ 
&= \bar{\Box} (\alpha_{\bar{k}} \otimes d \bar{z}^k) - \nabla^\wedge_{\xi''} (\alpha_{\bar{k}} \otimes d \bar{z}^k) + d \bar{z}^q \wedge i_{\xi''} \nabla_{\bar{\partial}_q}^\wedge (\alpha_{\bar{k}} \otimes d \bar{z}^k) - \bar{\partial}_L (g^{i \bar{j}} \theta_i \alpha_{\bar{j}}) 
\\
&= (\bar{\Box} - \nabla^\wedge_{\xi''}) (\alpha_{\bar{k}} \otimes d \bar{z}^k) - g^{i \bar{j}} \theta_{i \bar{k}} \alpha_{\bar{j}} d\bar{z}^k, 
\end{align*}
where we transform $d\bar{z}^q \wedge i_{\xi''} \nabla^\wedge_{\bar{\partial}_q} (\alpha_{\bar{k}} \otimes d \bar{z}^k)$ as 
\begin{align*} 
d\bar{z}^q \wedge i_{\xi''} \nabla^\wedge_{\bar{\partial}_q} (\alpha_{\bar{k}} \otimes d \bar{z}^k) 
&= d\bar{z}^q \wedge (\xi^{\bar{k}} \bar{\partial}_{L, \bar{q}} \alpha_{\bar{k}} + \alpha_{\bar{k}} (- \xi^{\bar{j}} \Gamma^{\bar{k}}_{\bar{q} \bar{j}}))
\\
&= g^{l \bar{k}} \theta_l \bar{\partial}_L \alpha_{\bar{k}} - \alpha_{\bar{k}} \theta_l g^{l \bar{j}} g^{p \bar{k}} g_{p \bar{j}, \bar{q}} d\bar{z}^q
\\
&= \theta_l \bar{\partial}_L (g^{l \bar{k}} \alpha_{\bar{k}}). 
\end{align*}
\end{proof}

\begin{cor}[Weitzenb\"ock formula]
Write $\Ric (\omega) = \sqdbard \log \det (g_{p \bar{q}})$ as $\sqrt{-1} R_{i \bar{j}} dz^i \wedge d \bar{z}^j$ and put $\xi := (\xi' - \xi'')/2 \sqrt{-1}$. 
Then we have the following. 
\begin{align}
\Box^\theta - \bar{\Box}^\theta
&= \Lambda (\sqrt{-1} \partial \dbar \phi) - 2 \sqrt{-1} \nabla_\xi^\wedge + g^{i \bar{j}} \theta_{i \bar{k}} d\bar{z}^k \otimes \bar{\partial}_j, 
\\
\Box^{\theta}_\# - \bar{\Box}^{\theta}_\#
&= \Lambda (\sqrt{-1} \partial \dbar \phi) + g^{i \bar{j}} R_{i \bar{k}} d\bar{z}^k \otimes \bar{\partial}_j - 2 \sqrt{-1} \nabla_\xi^\wedge, 
\\ 
\Box^{\theta}_\# - \Box^\theta
&= 0, 
\\
\bar{\Box}^{\theta}_\# - \bar{\Box}^\theta
&= - g^{i\bar{j}} (R_{i \bar{k}} - \theta_{i \bar{k}}) d \bar{z}^k \otimes \bar{\partial}_j. 
\end{align}
\end{cor}

\begin{proof}
The first two equalities follow from the above lemma combined with the usual Kodaira-Nakano formula
\begin{align*} 
\Box - \bar{\Box} 
&= \Lambda (\sqrt{-1} \partial \dbar \phi), 
\\
\Box_\# - \bar{\Box}_\# 
&= \flat (g^{i \bar{j}} {R_p^{~k}}_{i \bar{j}} dz^p \otimes \partial_k + g^{i \bar{j}} \phi_{i \bar{j}} dz^p \otimes \partial_p) \sharp 
\\ \notag
&= g^{p \bar{q}} R_{p \bar{l}} d \bar{z}^l \otimes \bar{\partial}_q + g^{i \bar{j}} \phi_{i \bar{j}} d\bar{z}^q \otimes \bar{\partial}_q. 
\end{align*}

Put $\alpha_{\bar{k}, p} := \nabla_p \alpha_{\bar{k}}$ and $\alpha_{\bar{k}, \bar{q}} := \nabla_{\bar{q}} \alpha_{\bar{k}}$. 
Then using (\ref{weighted adjoint 1}) and (\ref{weighted adjoint 2}), we obtain 
\begin{align} 
\nabla^\sharp (\alpha_{\bar{k}} d\bar{z}^k) 
&= g^{l \bar{k}} \alpha_{\bar{k}, p} d z^p \otimes \partial_l, 
\\
{\nabla^\sharp}^{\theta*} (\beta^l_p dz^p \otimes \partial_l) 
&= - g_{l \bar{j}} g^{p \bar{q}} (\beta^l_{p, \bar{q}} + \beta^l_p \theta_{\bar{q}}) d \bar{z}^j, 
\\
\nabla^{\sharp\sharp} (\alpha_{\bar{k}} d \bar{z}^k) 
&= (g^{l \bar{k}} \alpha_{\bar{k}})_{\bar{q}} d \bar{z}^q \otimes \partial_l, 
\\
{\nabla^{\sharp\sharp}}^{\theta*} (\beta^l_{\bar{q}} d \bar{z}^q \otimes \partial_l) 
&= {\nabla^{\sharp\sharp}}^* (\beta^l_{\bar{q}} d \bar{z}^q \otimes \partial_l) - g_{l \bar{\jmath}} g^{p \bar{q}} \theta_p \beta^l_{\bar{q}} d\bar{z}^j. 
\end{align}
\end{proof}

\subsubsection{$\mu$-Lichnerowicz operator and Reductiveness}

\begin{prop}[$\mu$-Lichnerowicz operator]
Put $\mathcal{D} := \nabla^{\sharp\sharp} \dbar: C^\infty_{\mathbb{C}} (X) \to \Omega^{0,1} (T^{1,0} X)$. 
Suppose $\xi' = \partial^\sharp \theta$ is a holomorphic vector field, then 
\begin{align}
\notag 
(\mathcal{D}^{\theta*} \mathcal{D}) f
&= ({\dbar}^{\theta*} \bar{\Box}^{\theta}_\# \dbar) f
\\ \notag
&= ({\dbar}^{\theta*} \bar{\Box}^\theta \dbar) f - ({\dbar}^{\theta*} (g^{i \bar{j}} (R_{i \bar{k}} - \theta_{i \bar{k}}) d \bar{z}^k \otimes \bar{\partial}_j ) \dbar) f
\\ \label{L}
&= (\bar{\Box} - \xi'')^2 f + (\Ric (\omega) - L_{\xi'} \omega, \sqddbar f) 
\\ \notag
& \qquad \qquad \qquad \quad + (\dbar^\sharp s_\xi (\omega)) (f), 
\end{align}
where $s_\xi (\omega) = (s (\omega) + \bar{\Box} \theta) +(\bar{\Box} \theta - \xi' \theta)$. 
\end{prop}

\begin{proof}
It suffices to show the third equality. 
As $\dbar \dbar = 0$, we have ${\dbar}^{\theta*} \bar{\Box}^\theta \dbar = ({\dbar}^{\theta*} \dbar) ({\dbar}^{\theta*} \dbar) = (\bar{\Box} -\xi'') (\bar{\Box} - \xi'')$. 
The second term in the second formula can be simplified as 
\begin{align}
\label{L0}
-({\dbar}^{\theta*} (g^{i \bar{j}} (R_{i \bar{k}} - \theta_{i \bar{k}}) d \bar{z}^k \otimes \bar{\partial}_j ) \dbar) f
&= (\sqrt{-1} (\Lambda \partial) + i_{\xi''}) (g^{i \bar{j}} (R_{i \bar{k}} - \theta_{i \bar{k}}) f_{\bar{j}} d\bar{z}^k). 
\end{align}

As for $\sqrt{-1} (\Lambda \partial) (g^{i \bar{j}} (R_{i \bar{k}} - \theta_{i \bar{k}}) f_{\bar{j}} d\bar{z}^k)$, 
\begin{align}
\notag
\sqrt{-1} (\Lambda \partial) (g^{i \bar{j}} (R_{i \bar{k}} - \theta_{i \bar{k}}) f_{\bar{j}} d\bar{z}^k)
&= \sqrt{-1} (-\sqrt{-1} g^{l \bar{k}}) (g^{i \bar{j}} (R_{i \bar{k}} - \theta_{i \bar{k}}) f_{\bar{j}})_l
\\ \notag
&= g^{l \bar{k}} g^{i \bar{j}} (R_{i \bar{k}} - \theta_{i \bar{k}}) f_{l \bar{j}} + g^{l \bar{k}} (g^{i \bar{j}} (R_{i \bar{k}} - \theta_{i \bar{k}}))_l f_{\bar{j}}
\\ \label{L1}
&= (\mathrm{Ric} (\omega) - L_{\xi'} \omega, \sqddbar f)
\\ \notag
& \qquad + g^{l \bar{k}} {g^{i \bar{j}}}_{, l} (R_{i \bar{k}} -\theta_{i \bar{k}}) f_{\bar{j}} + g^{l \bar{k}} g^{i \bar{j}} (R_{i\bar{k}, l} - \theta_{i \bar{k}, l}) f_{\bar{j}}, 
\end{align}
where $R_{i \bar{k}, l} = \partial R_{i \bar{k}}/\partial z^l$ and $\theta_{i \bar{k}, l} = \partial^3 \theta/\partial z^i \partial \bar{z}^k \partial z^l$. 
As $R_{i\bar{k}, l} - \theta_{i \bar{k}, l} = R_{l \bar{k}, i} - \theta_{l \bar{k}, i}$, the last term of (\ref{L1}) is equal to 
\begin{equation}
\label{L2} 
g^{i \bar{j}} (g^{l\bar{k}} (R_{l \bar{k}} - \theta_{l \bar{k}}))_i f_{\bar{j}} - g^{i \bar{j}} {g^{l \bar{k}}}_{, i} (R_{l \bar{k}} - \theta_{l \bar{k}}) f_{\bar{j}}. 
\end{equation}
As $g^{l \bar{k}} {g^{i \bar{j}}}_{,l} = - g^{l \bar{k}} g^{i \bar{q}} g_{p \bar{q}, l} g^{p \bar{j}} = - g^{l \bar{k}} g^{i \bar{q}} g_{l \bar{q}, p} g^{p \bar{j}} = {g^{i \bar{k}}}_{,p} g^{p \bar{j}}$, the second term of (\ref{L1}) is distinguished by the second term of (\ref{L2}). 
So we obtain 
\[ \sqrt{-1} (\Lambda \partial) (g^{i \bar{j}} (R_{i \bar{k}} - \theta_{i \bar{k}}) f_{\bar{j}} d\bar{z}^k) = (\mathrm{Ric} (\omega) - L_{\xi'} \omega, \sqddbar f) + (\dbar^\sharp (s+ \bar{\Box} \theta)) f. \]

The rest term in (\ref{L0}) is 
\begin{align*} 
i_{\xi''} (g^{i \bar{j}} (R_{i \bar{k}} - \theta_{i \bar{k}}) f_{\bar{j}} d\bar{z}^k) 
&= \xi^{\bar{k}} g^{i \bar{j}} (R_{i \bar{k}} - \theta_{i \bar{k}}) f_{\bar{j}} 
\\
&= g^{l \bar{k}} g^{i \bar{j}} \theta_l (R_{i \bar{k}} - \theta_{i \bar{k}}) f_{\bar{j}} 
\end{align*}
and the following calculations show (\ref{L}). 
As $\xi' = g^{q \bar{p}} \theta_{\bar{p}}$ is holomorphic, we have $(g^{p \bar{q}} \theta_p)_i = \overline{(g^{\bar{p} q} \theta_{\bar{p}})_{\bar{i}}} = \overline{\dbar_i \xi^q} = 0$. 
It follows that 
\begin{align*}
(\bar{\Box} \theta)_i 
= - (g^{p \bar{q}} \theta_{p \bar{q}})_i 
&= - (g^{p\bar{q}} \theta_p)_{i \bar{q}} + ({g^{p \bar{q}}}_{, \bar{q}} \theta_p)_i
\\ 
&= (- g^{p \bar{k}} g^{l \bar{q}} g_{l \bar{k}, \bar{q}} \theta_p)_i
\\ 
&= - (g^{l \bar{q}} g_{l \bar{q}, \bar{k}})_i g^{p \bar{k}} \theta_p - g^{l \bar{q}} g_{l \bar{q}, \bar{k}} (g^{p \bar{k}} \theta_p)_i
\\
&= g^{p \bar{k}} \theta_p R_{i \bar{k}}
\end{align*}
and 
\[ (\xi' \theta)_i 
= (g^{l \bar{k}} \theta_{\bar{k}} \theta_l)_i 
= (g^{l \bar{k}} \theta_l)_i \theta_{\bar{k}} + g^{l \bar{k}} \theta_l \theta_{i\bar{k}} = g^{l\bar{k}} \theta_l \theta_{i \bar{k}}. \]
\end{proof}

\begin{cor}[Reductiveness]
\label{reductiveness}
Suppose there exists a $\mu^\lambda_\xi$-cscK metric $\omega$ on $X$, then the identity component $\mathrm{Aut}^0_{\xi} (X/\mathrm{Alb})$ of the subgroup of the reduced automorphism group $\mathrm{Aut} (X/ \mathrm{Alb})$ preserving $\xi$ is the complexification of the group ${^\nabla \mathrm{Isom}}^0_\xi (X, \omega)$ of the Hamiltonian isometries of the $\mu$-cscK metric $\omega$ preserving $\xi$, especially, it is reductive. 
\end{cor}

\begin{proof}
If $\omega$ is a $\mu$-cscK, then the operator $\mathcal{D}^{\theta_*} \mathcal{D}$ restricted to $C^\infty_{\xi} (X, \mathbb{C}) = \{ f \in C^\infty (X, \mathbb{C}) ~|~ \xi f = 0 \}$ is a real operator. 
It follows that 
\[ \{ f \in C^\infty_{\xi} (X, \mathbb{C}) ~|~ \mathcal{D} f = 0 \} = \{ g + \sqrt{-1} h ~|~ \mathcal{D} g = \mathcal{D} h = 0, ~g, h \in C^\infty_\xi (X, \mathbb{R}) \}, \] 
which are respectively isomorphic to $\mathfrak{aut}_{\xi} (X, [\omega])$ and ${^\nabla \mathfrak{isom}}_{\xi} (X, g) \oplus \sqrt{-1} {^\nabla \mathfrak{isom}}_{\xi} (X, g)$ as we have $\mathcal{D} = \nabla^{\sharp\sharp} \dbar = \dbar_{TX} \partial^\sharp$. 
\end{proof}

\subsection{$\mu$-Futaki invariant}
\label{section: Futaki}

In this section, we fix a complex structure $J$ on $M$, a K\"ahler class $[\omega]$, the properly $\dbar$-Hamiltonian vector field $\xi$ and the parameter $\lambda \in \mathbb{R}$. 

Let $\xi$ be a properly $\dbar$-Hamiltonian vector field on a K\"ahler manifold $X$. 
Taking a $\xi$-invariant K\"ahler metric $\omega \in [\omega]$, we define a $\mathbb{C}$-linear functional $\Fut^\lambda_\xi: \mathfrak{h}_0 (X) \to \mathbb{C}$ by 
\begin{equation}
\Fut^{\lambda}_\xi (\zeta) := \int_X \hat{s}^\lambda_\xi (\omega) \theta_\zeta ~e^{\theta_\xi} \omega^n \Big{/} \int_X e^{\theta_\xi} \omega^n. 
\end{equation}
Remember that 
\begin{align*}
\hat{s}^\lambda_\xi (\omega)
&= (s (\omega) + \bar{\Box} \theta_\xi) + (\bar{\Box} \theta_\xi - \xi^J \theta_\xi) - \lambda \theta_\xi - \bar{s}^\lambda_\xi, 
\\ 
\bar{s}^\lambda_\xi 
&= \int_X (s + \bar{\Box} \theta_\xi - \lambda \theta_\xi) e^{\theta_\xi} \omega^n \Big{/} \int_X e^{\theta_\xi} \omega^n. 
\end{align*}

The following proposition proves Theorem B (1). 

\begin{prop}
\label{Futaki}
The linear functional $\Fut^\lambda_\xi$ is independent of the choice of the $\xi$-invariant K\"ahler metric $\omega$ in the fixed K\"ahler class $[\omega]$ and of the normalization of the moment map $\theta$ (independent of the equivariant cohomology class $[\omega + \theta]$). 
\end{prop}

\begin{proof}
Take two $\xi$-invariant K\"ahler forms $\omega, \omega' \in [\omega]$ and take a smooth function $\phi$ so that $\omega' = \omega + \sqddbar \phi$. 
Put $\omega_t := \omega + t \sqddbar \phi$. 
Then the moment map $\mu^t$ with respect to $\omega_t$ with $\omega_t + \mu^t \in [\omega + \mu]$ is given by $\mu^t_\xi = \mu_\xi - t \xi^J \phi/2$. 
We put $\theta^t_\zeta := \theta_\zeta + t \zeta^J \phi$ for $\zeta \in \mathfrak{h}_0 (X)$, which satisfies $\dbar \theta^t_\zeta = i_{\zeta^J} \omega_t$ and is complex-valued in general and becomes real-valued when $\zeta \in \mathfrak{h}_0 (X, \omega)$. 
As we already know that $\int_X e^{\theta_\xi^t} \omega^n$ is invariant, it is sufficient to see 
\[ \frac{d}{dt} \int_X \hat{s}^\lambda_\xi (g_t) \theta^t_\zeta ~e^{\theta^t_\xi} \omega_t^n = 0 \]
for every $t \in [0,1]$. 
Firstly, we compute 
\begin{align*}
\frac{d}{dt} \hat{s}^\lambda_\xi (\omega_t)
&= \frac{d}{dt} \Big{(} ( g_t^{i \bar{j}} R^t_{i \bar{j}} - g^{i \bar{j}}_t \theta^t_{\xi, i \bar{j}}) + (- g^{i \bar{j}}_t \theta^t_{\xi, i \bar{j}} - \xi^J \theta_{t, \xi}) - \lambda \theta_{t, \xi} \Big{)}
\\ 
&= - g^{i \bar{q}}_t \dot{g}^t_{p \bar{q}} g^{p \bar{j}}_t R^t_{i \bar{j}} - g^{i \bar{j}}_t (g^{k \bar{l}} \dot{g}^t_{k \bar{l}})_{i \bar{j}} + g^{i \bar{q}}_t \dot{g}^t_{p \bar{q}} g_t^{p \bar{j}} \theta^t_{\xi, i \bar{j}} - g_t^{i \bar{j}} (\xi^J \phi)_{i \bar{j}}
\\ 
& \qquad + g^{i \bar{q}}_t \dot{g}^t_{p \bar{q}} g^{p \bar{j}} \theta^t_{\xi, i \bar{j}} - g^{i \bar{j}} (\xi^J \phi)_{i \bar{j}} - \xi^J \xi^J \phi - \lambda \xi^J \phi
\\
&= - \bar{\Box}_t \bar{\Box}_t \phi - (\Ric (\omega_t) - L_{\xi^J} \omega_t , \sqrt{-1} \partial \dbar \phi) + \bar{\Box}_t \xi^J \phi 
\\
& \qquad + (L_{\xi^J} \omega_t , \sqrt{-1} \partial \dbar \phi) + (\bar{\Box}_t - \xi^J )(\xi^J \phi) - \lambda \xi^J \phi
\\
&= - ((\bar{\Box}_t - \bar{\xi}^J)^2 \phi + (\Ric (\omega_t) - L_{\xi^J} \omega_t , \sqrt{-1} \partial \dbar \phi)) - \bar{\xi}^J \bar{\Box}_t \phi - \bar{\Box}_t \bar{\xi}^J \phi + \bar{\xi}^J \bar{\xi}^J \phi 
\\
& \qquad + \bar{\Box}_t \xi^J \phi + (L_{\xi^J} \omega_t , \sqrt{-1} \partial \dbar \phi) + (\bar{\Box}_t - \xi^J )(\xi^J \phi) - \lambda \xi^J \phi
\\ 
&= - \mathcal{D}_t^{\theta*} \mathcal{D}_t \phi + (\dbar^\sharp s_\xi (g_t)) (\phi) - \lambda \xi^J \phi 
\\ 
&\qquad \qquad \qquad - \bar{\xi}^J \bar{\Box}_t \phi + \bar{\Box}_t \xi^J \phi + (L_{\xi^J} \omega_t , \sqrt{-1} \partial \dbar \phi) 
\\
&= - \mathcal{D}_t^{\theta*} \mathcal{D}_t \phi + (\dbar^\sharp \hat{s}^\lambda_\xi (g_t)) (\phi) = - \overline{\mathcal{D}_t^{\theta*} \mathcal{D}_t} \phi + (\partial^\sharp \hat{s}^\lambda_\xi (g_t)) (\phi), 
\end{align*}
where we used the $\xi$-invariance of metrics for $\bar{\xi}^J \phi = \xi^J \phi$ etc. and compute the last line by 
\begin{align*} 
(L_{\xi^J} \omega_t, \sqddbar \phi) 
&= g_t^{i \bar{l}} g_t^{k \bar{j}} \theta^t_{\xi, i \bar{j}} \phi_{k \bar{l}} 
\\ 
&= g^{i \bar{l}}_t (\xi^k \phi_{k \bar{l}})_i - g^{i \bar{l}}_t g^{k \bar{j}}_{t, i} \theta^t_{\xi, \bar{j}} \phi_{k \bar{l}} - g^{i \bar{l}}_t \xi^k \phi_{k \bar{l} i}
\\
&= g^{i \bar{l}}_t ((\xi^k \phi_k)_{\bar{l}})_i - g^{k \bar{l}}_{t, p} g^{p \bar{j}}_t \theta^t_{\xi, \bar{j}} \phi_{k \bar{l}} - \xi^k (g^{i \bar{l}}_t \phi_{i \bar{l}})_k + g^{k \bar{j}}_t \theta_{\xi, \bar{j}}^t g_{t, k}^{i \bar{l}} \phi_{i \bar{l}}
\\ 
&= - \bar{\Box} \xi^J \phi + \xi^J \bar{\Box} \phi. 
\end{align*}
It follows that 
\begin{align*}
\frac{d}{dt} \int_X  \hat{s}^\lambda_\xi (g_t) \theta^t_{\zeta} ~e^{\theta^t_{\xi}} \omega_t^n
&= \int_X - \overline{\mathcal{D}_t^{\theta*} \mathcal{D}_t} \phi ~ \theta^t_{\zeta} e^{\theta^t_{\xi}} \omega_t^n 
+ \int_X (\partial^\sharp \hat{s}^\lambda_\xi (g_t)) (\phi) ~\theta^t_{\zeta} e^{\theta^t_{\xi}} \omega_t^n 
\\
& \qquad + \int_X \hat{s}^\lambda_\xi (g_t) \zeta^J \phi ~ e^{\theta^t_{\xi}} \omega_t^n - \int_X \hat{s}^\lambda_\xi (g_t) \theta^t_{\zeta} (\Box_t -\xi^J) (\phi) e^{\theta^t_{\xi}} \omega_t^n
\\
& =- \int_X  \phi ~ \mathcal{D}_t^{\theta*} \mathcal{D}_t \theta^t_{\zeta} e^{\theta^t_{\xi}} \omega_t^n 
+ \int_X (\partial^\sharp \hat{s}^\lambda_\xi (g_t)) (\phi) ~\theta^t_{\zeta} e^{\theta^t_{\xi}} \omega_t^n 
\\
& \qquad + \int_X \hat{s}^\lambda_\xi (g_t) \zeta^J \phi ~ e^{\theta^t_{\xi}} \omega_t^n - \int_X \partial^\sharp (\hat{s}^\lambda_\xi (g_t) \theta^t_{\zeta}) (\phi) e^{\theta^t_{\xi}} \omega_t^n 
\\
&= - \int_X \phi ~\mathcal{D}_t^{\theta*} \mathcal{D}_t  \theta^t_{\zeta} e^{\theta^t_{\xi}} \omega_t^n 
\end{align*}
and the last term vanishes as $\zeta \in \mathfrak{h}_0 (X)$. 
\end{proof}

By the definition of $\Fut^\lambda_\xi$, if there is a $\mu^\lambda_\xi$-cscK metric in the K\"ahler class $[\omega]$, $\Fut^\lambda_\xi$ must vanish. 

We put 
\begin{equation} 
\tilFut^\lambda_\xi (\zeta) := \int_X \hat{s}^\lambda_\xi (\omega) \theta_\zeta ~e^{\theta_\xi} \omega^n = \Fut^\lambda_\xi (\zeta) \int_X e^{\theta_\xi} \omega^n. 
\end{equation}
In contrast to $\Fut^\lambda_\xi$, $\tilFut^\lambda_\xi$ depends on the choice of the moment map $\theta$ while it is independent of the choice of the K\"ahler metric in the fixed K\"ahler class $[\omega]$. 

When $X$ is a Fano manifold and $[\omega] = 2 \pi c_1 (X)$, $\Fut_\xi^1$ reduces to the following well-known form: 
\[ \Fut_\xi^1 (\zeta) = - \int_X \zeta^J (h - \theta_\xi^v) e^{\theta_\xi^v} \omega^n = -\int_X \theta_\zeta e^{\theta_\xi} \omega^n \Big{/} \int_X e^{\theta_\xi} \omega^n, \]
where $h$ is a Ricci potential and $\theta^v_\xi$ denotes the normalization of $\theta_\xi$ satisfying $\int_X e^{\theta^v_\xi} \omega^n = 1$. 
This invariant was investigated in \cite{TZ}. 

\subsection{$\mu$-volume functional}
\label{section: volume}

Here we introduce a generalization of a functional considered in \cite{TZ}. 

Let $X$ be a compact K\"ahler manifold with a Hamiltonian holomorphic action of a compact Lie group $K$ and $\omega$ be a $K$-invariant K\"ahler form on $X$. 
Define the \textit{$\mu$-volume functional} $\mathrm{Vol}^\lambda$ with respect to $\omega$ on $\mathfrak{k}$ by 
\begin{equation}
\mathrm{Vol}^\lambda (\xi) := e^{\bar{s}_\xi^\lambda} \Big{(} \int_X e^{\theta_\xi} \omega^n \Big{)}^\lambda 
\end{equation}
using a real-valued Hamiltonian potential $\theta_\xi: \sqrt{-1} \dbar \theta_\xi = i_{\xi^J} \omega$. 
We can easily check that $\mathrm{Vol}^\lambda (\xi)$ is independent of the choice of the Hamiltonian potential. 
Remember again that the constant $\bar{s}^\lambda_\xi$ is given by
\[ \bar{s}^\lambda_\xi = \int_X (s + \bar{\Box} \theta_\xi -\lambda \theta_\xi) e^{\theta_\xi} \omega^n \Big{/} \int_X e^{\theta_\xi} \omega^n. \]
As $\bar{s}^\lambda_\xi$ and $\int_X e^{\theta_\xi} \omega^n$ is independent of the choice of the $\xi$-invariant K\"ahler metric, the $\mu$-volume functional $\mathrm{Vol}^\lambda$ is also independent of the choice of the $K$-invariant K\"ahler metric in the fixed K\"ahler class $[\omega]$. 


When $X$ is a Fano manifold and the K\"ahler class $[\omega]$ is equal to $2 \pi c_1 (X)$, we have $\bar{s}_\xi^1 = n$ by (\ref{Fano s bar}) under the normalization (\ref{Fano moment}) and thus obtain $\mathrm{Vol}^1 = e^n \int_X e^{\theta_\xi} \omega^n$, which is equivalent to the volume functional considered in \cite{TZ}. 
We can easily see the properness and the convexity of $\mathrm{Vol}^1$ in this case and thus obtain a unique critical point $\xi$ of $\mathrm{Vol}^1$, which is equivalent to $\Fut_\xi \equiv 0$. 

In \cite{Ino}, the author used this result in order to formulate an appropriate moduli problem for Fano manifolds admitting K\"ahler-Ricci solitons, which is equivalent to detect a sensible moduli stack, and to construct the moduli space of them. 
It is important that we have such a result for all Fano manifolds, not only for Fano manifolds admitting K\"ahler-Ricci solitons, as we must include `K-semistable' manifolds in the member of the moduli stack in order to ensure the openness of the interested families in general families, which corresponds to the Artinness of the moduli stack. 

\subsubsection{Variational formulas}

\begin{prop}
The derivative $d_\xi \mathrm{Vol}^\lambda$ of $\mathrm{Vol}^\lambda$ at $\xi \in \mathfrak{k}$ is given by 
\[
(d_{\xi} \mathrm{Vol}^\lambda) (\zeta) = \mathrm{Vol}^\lambda (\xi) \cdot \Fut^{\lambda}_\xi (\zeta). 
\]
\end{prop}

\begin{proof}
We calculate the derivative of $\log \mathrm{Vol}^\lambda (\xi) = \bar{s}^\lambda_\xi + \lambda \log \int_X e^{\theta_\xi} \omega^n$. 
We have the following basic calculations: 
\begin{align} 
\frac{d}{dt}\Big{|}_{t=0} \int_X (s  + \bar{\Box} \theta_{\xi + t \zeta} 
& - \lambda \theta_{\xi + t \zeta}) e^{\theta_{\xi + t \zeta}} \omega^n
\\ \notag
&= \int_X \Big{(} (s + \bar{\Box} \theta_\xi - \lambda) + (\bar{\Box} \theta_\xi - \xi^J \theta_\xi - \lambda \theta_\xi) \Big{)} \theta_\zeta e^{\theta_\xi} \omega^n
\\ \notag
&= \tilFut^{\lambda}_\xi (\zeta) + (\bar{s}^\lambda_\xi - \lambda) \int_X \theta_\zeta e^{\theta_\xi} \omega^n, 
\\
\frac{d}{dt}\Big{|}_{t=0} \int_X e^{\theta_{\xi+ t \zeta}} \omega^n 
&= \int_X \theta_\zeta e^{\theta_\xi} \omega^n. 
\end{align}
It follows that 
\[ \frac{d}{dt}\Big{|}_{t=0} \bar{s}^\lambda_{\xi + t \zeta} = \Fut^{\lambda}_\xi (\zeta) - \lambda \int_X \theta_\zeta e^{\theta_\xi} \omega^n \Big{/} \int_X e^{\theta_\xi} \omega^n. \]
So we obtain 
\begin{align*} 
\frac{d}{dt}\Big{|}_{t=0} \log \mathrm{Vol}^\lambda (\xi + t \zeta) 
&= \frac{d}{dt}\Big{|}_{t=0} \Big{(} \bar{s}^\lambda_{\xi + t\zeta} + \lambda \log \int_X e^{\theta_{\xi + t \zeta}} \omega^n \Big{)} 
\\ 
&= \Fut^\lambda_\xi (\zeta). 
\end{align*}
\end{proof}

\begin{rem}
The log of the $\mu$-volume functional is given by 
\[ \log \mathrm{Vol}^\lambda = \int_X (s + \bar{\Box} \theta^v_\xi - \lambda \theta^v_\xi) e^{\theta^v_\xi} \omega^n, \]
where we put $\theta_\xi^v := \theta_\xi - \log \int_X e^{\theta_\xi} \omega^n$ so that $\int_X e^{\theta_\xi^v} \omega^n = 1$. 
As we have $\int_X \bar{\Box} \theta^v_\xi e^{\theta^v_\xi} \omega^n = \int_X |\dbar \theta_\xi^v|^2 e^{\theta_\xi^v} \omega^n$, this functional has the same expression with the Perelman's $W$-functional: 
\[ W (\omega, f, \lambda^{-1}) := \int_X (2 \lambda^{-1} (s + |\dbar f|^2) + f ) e^{-f} \Big{(} \frac{\lambda \omega}{4 \pi} \Big{)}^n. \]
While we usually consider the $W$-functional for positive $\lambda > 0$ (and for general smooth function $f$ with the normalization $\int_X e^{-f} ( \frac{\lambda \omega}{4 \pi} )^n = 1$), we are mainly interested in $\lambda \le 0$ for our $\mathrm{Vol}^\lambda$ in the context of $\mu$-cscK. 

\end{rem}

Next, we exhibit the second variational formula of $\mathrm{Vol}^\lambda$. 
Define a smooth map $\mathrm{DVol}^\lambda: \mathfrak{k} \to \mathfrak{k}^*$ by 
\begin{equation}
\mathrm{DVol}^\lambda (\xi) = d_\xi \mathrm{Vol}^\lambda = \mathrm{Vol}^\lambda (\xi) \cdot \Fut^\lambda_\xi. 
\end{equation}

\begin{prop}
The derivative $d_\xi \mathrm{DVol}^\lambda: \mathfrak{k} \to \mathfrak{k}^*$ of $\mathrm{DVol}^\lambda$ at $\xi \in \mathfrak{k}$ is given by 
\begin{align*}
\langle d_\xi \mathrm{DVol}^\lambda (\zeta) ,\bullet \rangle
&= \mathrm{Vol}^\lambda (\xi)^2 \cdot \Fut^\lambda_\xi (\zeta) \cdot \Fut^\lambda_\xi (\bullet) 
\\
&\qquad - \mathrm{Vol}^\lambda (\xi) \cdot \frac{\int_X \theta_\zeta e^{\theta_\xi} \omega^n}{\int_X e^{\theta_\xi} \omega^n} \Fut^\lambda_\xi (\bullet)
- \mathrm{Vol}^\lambda (\xi) \cdot \frac{\int_X \theta_\bullet e^{\theta_\xi} \omega^n}{\int_X e^{\theta_\xi} \omega^n} \Fut^\lambda_\xi (\zeta)
\\ \notag
&\qquad \quad + \mathrm{Vol}^\lambda (\xi) \cdot \left( \int_X e^{\theta_\xi} \omega^n \right)^{-1} \int_X (\hat{s}^\lambda_\xi \theta_\zeta \theta_{\bullet} + 2 \zeta^J \theta_{\bullet} ) e^{\theta_\xi} \omega^n
\\
& \qquad \qquad - \lambda \mathrm{Vol}^\lambda (\xi) \cdot \left( \int_X e^{\theta_\xi} \omega^n \right)^{-1} \left( \int_X \theta_\zeta \theta_{\bullet} e^{\theta_\xi} \omega^n 
- \frac{\int_X \theta_\zeta e^{\theta_\xi} \omega^n}{\int_X e^{\theta_\xi} \omega^n} \int_X \theta_{\bullet} e^{\theta_\xi} \omega^n \right). 
\end{align*}
\end{prop}

\begin{proof}
Using the first variational formula, we have 
\begin{align*}
\langle d_\xi \mathrm{DVol}^\lambda (\zeta) , \bullet \rangle 
&= \mathrm{Vol}^\lambda (\xi)^2 \cdot \Fut^\lambda_\xi (\zeta) \cdot \Fut^\lambda_\xi (\bullet) 
- \mathrm{Vol}^\lambda (\xi) \cdot \frac{\int_X \theta_\zeta e^{\theta_\xi} \omega^n}{\int_X e^{\theta_\xi} \omega^n} \Fut^\lambda_\xi (\bullet)
\\
&\qquad + \mathrm{Vol}^\lambda (\xi) \left( \int_X e^{\theta_\xi} \omega^n \right)^{-1} \frac{d}{dt} \Big{|}_{t=0} \tilFut^\lambda_{\xi+ t\zeta} (\bullet). 
\end{align*}
The claim follows by the following computation. 
\begin{align*}
\frac{d}{dt} \Big{|}_{t=0} \tilFut^\lambda_{\xi+ t\zeta} (\bullet)
&= \frac{d}{dt} \Big{|}_{t=0} \int_X (s + 2 \bar{\Box} \theta_{\xi + t\zeta} - (\xi+ t\zeta)^J \theta_{\xi + t\zeta} - \lambda \theta_{\xi + t \zeta} - \bar{s}^\lambda_{\xi + t \zeta}) \theta_\bullet e^{\theta_{\xi + t \zeta}} \omega^n
\\
&= \int_X (2 \bar{\Box} \theta_\zeta - 2 \xi^J \theta_\zeta - \lambda \theta_\zeta) \theta_\bullet e^{\theta_\xi} \omega^n  - \frac{\tilFut^\lambda_\xi (\zeta) - \lambda \int_X \theta_\zeta e^{\theta_\xi} \omega^n}{\int_X e^{\theta_\xi} \omega^n} \int_X \theta_\bullet e^{\theta_\xi} \omega^n
\\
& \qquad + \int_X \hat{s}^\lambda_\xi \theta_\bullet \theta_\zeta e^{\theta_\xi} \omega^n
\\
&= 2 \int_X \zeta^J \theta_\bullet e^{\theta_\xi} \omega^n - \frac{\int_X \theta_\bullet e^{\theta_\xi} \omega^n}{\int_X e^{\theta_\xi} \omega^n} \tilFut^\lambda_\xi (\zeta) + \int_X \hat{s}^\lambda_\xi \theta_\bullet \theta_\zeta e^{\theta_\xi} \omega^n
\\
& \qquad - \lambda \left( \int_X \theta_\zeta \theta_{\bullet} e^{\theta_\xi} \omega^n 
- \frac{\int_X \theta_\zeta e^{\theta_\xi} \omega^n}{\int_X e^{\theta_\xi} \omega^n} \int_X \theta_{\bullet} e^{\theta_\xi} \omega^n \right). 
\end{align*}
\end{proof}


\subsubsection{Corollaries of the second variational formula}

Using the second variational formula, we obtain a criterion for $\xi$ to be a local minimizer. 

\begin{cor}
\label{corollary of the second variational formula}
Let $\omega$ be a $\mu^\lambda_\xi$-cscK metric on $X$. 
If $\xi$ is a local minimizer of $\mathrm{Vol}^\lambda$, then 
\begin{equation}
\lambda \le \frac{2 \int_X |\zeta^J|_g^2 e^{\theta_\xi} \omega^n / \int_X e^{\theta_\xi} \omega^n}{\nu_\xi (\zeta)}
\end{equation}
for every $\zeta \in \mathfrak{k} \setminus \{ 0 \}$ (with $|\zeta| = 1$). 

Conversely, if we have 
\begin{equation}
\lambda < \frac{2 \int_X |\zeta^J|_g^2 e^{\theta_\xi} \omega^n / \int_X e^{\theta_\xi} \omega^n}{\nu_\xi (\zeta)}
\end{equation}
for every $\zeta \in \mathfrak{k} \setminus \{ 0 \}$ (with $|\zeta| = 1$), then $\xi \in \mathfrak{k}$ is an isolated local minimizer of the functional $\mathrm{Vol}^\lambda$. 
(Note that $2 \int_X |\zeta^J|^2_g e^{\theta_\xi} \omega^n/\int_X e^{\theta_\xi} \omega^n$ depends on the $\mu$-cscK metric $\omega$ and so on $\lambda$. )
\textbf{Especially, $\xi$ is an isolated local minimizer when $\lambda \le 0$. }

Moreover, let $\lambda_1$ be the first eigenvalue of the weighted $\bar{\partial}$-Laplacian $\bar{\Box}_g - \xi^J$ (restricted to the space of $\xi$-invariant real functions) with respect to the $\mu$-cscK metric $\omega$ and suppose $\lambda < 2 \lambda_1$, then $\xi \in \mathfrak{k}$ is an isolated local minimizer. 
\end{cor}

\begin{proof}
If $\xi$ is a local minimizer, then we should have $(d/dt)^2|_{t=0} \mathrm{Vol}^\lambda (\xi + t \zeta) = d_\xi D\mathrm{Vol}^\lambda (\zeta) (\zeta) \ge 0$ for every $\zeta \neq 0$. 
On the other hand, if we have $(d/dt)^2|_{t=0} \mathrm{Vol}^\lambda (\xi + t \zeta) = d_\xi D\mathrm{Vol}^\lambda (\zeta) (\zeta) > 0$ for every $\zeta \neq 0$, then $\xi$ is an isolated minimizer. 
Then the first two claims follow by the second variational formula of $\mathrm{Vol}^\lambda$. 
The last statement follows by the Poincare's inequality. 
\end{proof}



Note that the origin $0 \in \mathfrak{k}$ is a critical point of $\mathrm{Vol}^\lambda$ if and only if the usual Futaki invariant $\Fut$ vanishes, which is independent of $\lambda$. 
So we also obtain the following corollary, which will give a non-uniqueness of critical points in the next subsection. 

\begin{cor}
\label{Fut zero}
Suppose $\Fut \equiv 0$. 
Then the origin $0 \in \mathfrak{k}$ is an isolated local minimizer of $\mathrm{Vol}^\lambda$ if 
\[ \lambda < \frac{\int_X ((s- \bar{s}) \theta_\zeta^2 + 2 |\zeta^J|^2) \omega^n / \int_X \omega^n}{\nu_0 (\zeta)} \]
and $0 \in \mathfrak{k}$ is a local minimizer only when
\[ \lambda \le \frac{\int_X ((s- \bar{s}) \theta_\zeta^2 + 2 |\zeta^J|^2) \omega^n / \int_X \omega^n}{\nu_0 (\zeta)} \]
for every $\zeta \in \mathfrak{k} \setminus \{ 0 \}$, where the right hand side is independent of the choices of the K\"ahler metric in the fixed K\"ahler class and the moment map. 
\end{cor}

\begin{proof}
Note $\hat{s}^\lambda_0 = s - \underline{s}$. 
The claim follows by the second variational formula. 
We can express $\int_X ((s- \underline{s}) \theta_\zeta^2 + 2 |\zeta^J|^2) \omega^n$ by the integral of equivariant closed forms as 
\[ \int_X ((s- \underline{s}) \theta_\zeta^2 + 2 |\zeta^J|^2) \omega^n = \frac{2}{n+1} \Big{(} \int (\Ric (\omega) + \bar{\Box} \theta_\zeta) (\omega+ \theta_\zeta)^{n+1} - \frac{\underline{s}}{n+2} \int_X (\omega + \theta_\zeta)^{n+2} \Big{)}, \]
which proves the independence from $\omega$. 
As for the independence of the normalization of the moment map, it follows from $\Fut = \int_X (s-\underline{s}) \theta \omega^n \equiv 0$. 
\end{proof}

For a K\"ahler manifold $(X, [\omega])$ with $\Fut \equiv 0$, we put 
\begin{align}
\lambda_\infty (X, [\omega]) 
&:= \sup \{ \lambda \in \mathbb{R} ~|~ \mathrm{Vol}^\lambda \text{ is locally minimized at the origin } \} 
\\ \notag
&= \min_{|\zeta| = 1} \frac{\int_X ((s- \underline{s}) \theta_\zeta^2 + 2 |\zeta^J|^2) \omega^n}{\nu_0 (\zeta)}. 
\end{align}
Note that for every $\lambda \le \lambda_\infty (X, [\omega])$, $\mathrm{Vol}^\lambda_{[\omega]}$ is locally minimized at the origin. 
By the Poincare's inequality, we have the following lower bound: 
\begin{equation} 
\label{lambda estimate}
\sup_{\omega \in [\omega]} (\min_X s (\omega) + 2\lambda_1 (\omega)) - \underline{s} \le \lambda_\infty (X, [\omega]), 
\end{equation}
where $\lambda_1 (\omega)$ denotes the first eigenvalue of $\bar{\Box}_\omega$. 

Now suppose $X$ is a Fano manifold and there is a K\"ahler metric $\omega$ in a fixed K\"ahler class on $X$ with a lower bound on the Ricci curvature $\Ric (\omega) \ge \delta \omega$ for $\delta > 0$. 
Then by Lichnerowicz-Obata's theorem, we obtain a lower bound on the first eigenvalue $\lambda_1 (\omega) \ge \frac{n}{2n-1} \delta$ (note that $\bar{\Box} = \frac{1}{2} \Delta$). 
On the other hand, we have $s (\omega) \ge \delta n$. 
It follows that if the Futaki invariant of $X$ vanishes, then by (\ref{lambda estimate}) we obtain a lower bound $\lambda_\infty (X, [\omega]) \ge \delta (n + \frac{2n}{2n-1}) - \underline{s}$. 
So in particular, in this case, the origin $0 \in \mathfrak{k}$ is an isolated local minimizer of $\mathrm{Vol}^\lambda$ for all $\lambda \le 0$ if $\delta > \frac{2n-1}{(2n + 1)n} \underline{s}$. 

As for $[\omega] = 2 \pi c_1 (X)$, we can explicitly compute as 
\begin{equation}
\lambda_\infty (X, 2 \pi c_1 (X)) = 2. 
\end{equation}
This follows by the equality of equivariant classes $[\Ric + \bar{\Box} \theta_\zeta] = [\omega + \theta_\zeta]$ and the formula in the proof of the above corollary. 
We can also deduce this by using $s- \underline{s} = - \bar{\Box} h$ and $\bar{\Box} \theta_\zeta - \zeta^J h - \theta_\zeta = 0$ as in (\ref{Fano moment 2}).

\begin{quest}
Is $\lambda_\infty (X, [\omega])$ positive for every K\"ahler manifold $X$ and K\"ahler class $[\omega]$ with vanishing Futaki invariant? 
\end{quest}

\subsubsection{Properness of $\mathrm{Vol}^\lambda$}

Now we show that $\mathrm{Vol}^\lambda$ is proper for general $X$, not necessarily a Fano manifold, and thus always have a critical point. 

\begin{lem}
Let $M$ be a closed manifold and $f$ be a Morse-Bott function. 
Normalize $f$ so that $\max f = 0$ by adding a constant and suppose $f^{-1} (0)$ is connected of codimension $k$. 
Then for any smooth measure $dm$, the parametrized measure $t^{k/2} e^{tf} dm$ converges to a non-zero finite measure supported on $f^{-1} (0)$ as $t$ tends to $+ \infty$. 

Moreover, the parametrized measure $(-1)^p t^{k/2 + p} f^p e^{tf} dm$ converges to a non-zero finite measure supported on $f^{-1} (0)$ for every non-negative integer $p$. 
\end{lem}

\begin{proof}
On any compact set $K \subset M \setminus f^{-1} (0)$, the parametrized measure $t^{k/2} e^{tf} dm$ converges to zero in the order $o (t^{k/t} e^{-\epsilon t})$ as $f$ is smaller than some $-\epsilon < 0$ on $K$. 

For a point $p$ of $f^{-1} (0)$, we can take a local coordinate of $p$ so that $f (x)$ can be written as $ = - (x_1^2 + \dotsb + x_k^2)$. 
Then we can write $t^{k/2} e^{t f} dm$ as 
\[ t^{k/2} e^{t f} dm = t^{k/2} e^{- t (x_1^2 + \dotsb + x_k^2)} m (x) dx_1 \dotsb dx_n \]
for a positive function $m (x)$ on this coordinate. 
It suffices to prove that the parametrized measure $t^{k/2} e^{- t (x_1^2 + \dotsb x_k^2)} dx_1 \dotsb dx_n$ converges to a non-zero finite measure supported on $\{ x_1 = \dotsb = x_k = 0 \}$. 
As we only need to check the convergence of the integration of all the test functions of boxes, the claim follows by the Gaussian integral. 

As for $(-1)^p t^{k/2 + p} f^p e^{tf} dm = t^{k/2 + p}  (x_1^2 + \dotsb + x_k^2)^p e^{- t (x_1^2 + \dotsb + x_k^2)} m (x) dx_1 \dotsb dx_n$, we have 
\[ \int x_1^{2p_1} \dotsb x_k^{2p_k} e^{-t (x_1^2 + \dotsb + x_k^2)} dx_1 \dotsb dx_k = \prod_{i=1}^k \int x_i^{2p_i} e^{-t x_i^2} dx_i \]
for $p_i$ with $p_1 + \dotsb + p_k = p$. 
Integrating by parts, we obtain 
\begin{align*} 
\int_0^a x_i^{2p_i} e^{-t x_i^2} dx_i 
&= - \frac{1}{2t} a^{2p_i -1} e^{-t a^2} + \frac{2p_i -1}{2t} \int_0^a x_i^{2p_i -2} e^{-t x_i^2} dx_i
\\
&= \dotsb = o (e^{-t a^2}) + C t^{-p_i} \int_0^a e^{-t x_i^2} dx_i = o (e^{-ta^2}) + C t^{-p_i} t^{-1/2}. 
\end{align*}
This proves the claim. 
\end{proof}

\begin{prop}
\label{properness of mu-volume}
Let $X$ be a compact K\"ahler manifold and $K$ be a compact Lie group acting on $X$. 
The limit $\lim_{t \to \infty} t^{-1} \log \mathrm{Vol}^\lambda (t \xi)$ exists. 
It is moreover independent of $\lambda \in \mathbb{R}$ and is strictly positive for each $\xi \in \mathfrak{k} \setminus \{ 0 \}$. 
In particular, $\mathrm{Vol}^\lambda$ is proper on $\mathfrak{k}$ for each $\lambda \in \mathbb{R}$. 
\end{prop}

\begin{proof}
Recall that the Hamiltonian potential $\theta_\xi$ is a Morse-Bott function with only even indices and coindices. 
In particular, $\theta_\xi^{-1} (c)$ is a connected submanifold for every $c \in \mathbb{R}$ (cf. \cite{MS}). 
As $\mathrm{Vol}^\lambda$ is independent of the normalization of $\theta_\xi$, we can suppose $\max \theta_\xi = 0$. 
Note that we have $\theta_{t \xi} = t \theta_\xi$ for $t \ge 0$ with respect to this normalization, while it is not linear on $\xi$. 
Let $2k$ be the real codimension of $\Sigma := \theta_\xi^{-1} (0)$. 

We can write the log of the $\mu$-volume functional as
\[ \log \mathrm{Vol}^\lambda (\xi) = \bar{s}^0_\xi + \lambda \log \int_X e^{\theta_\xi} \omega^n - \lambda \int_X \theta_\xi e^{\theta_\xi} \omega^n \Big{/} \int_X e^{\theta_\xi} \omega^n. \]

As for the first term, we can write as 
\[ \bar{s}^0_{t \xi}/t = \int_X (s(x) /t) e^{t \theta_\xi} \omega^n \Big{/} \int_X e^{t \theta_\xi} \omega^n + \int_X \bar{\Box} \theta_\xi e^{t \theta_\xi} \omega^n \Big{/} \int_X e^{t \theta_\xi} \omega^n. \]
Since $\max_{x \in X} |s (x)/t|$ goes to $0$ and $e^{t \theta_\xi} \omega^n / \int_X e^{t\theta_\xi} \omega^n$ is a probability measure for any $t$, the first term converges to zero as $t$ tends to infinity. 
Thanks to the above lemma, the second term converges to the integration of $\bar{\Box} \theta_\xi$ with respect to a non-zero finite measure supported on $\Sigma$. 
Since $\theta_\xi$ is a Morse--Bott function, the Hessian at critical points are non-degenerate to the normal direction, so that we obtain a strict positivity of $\bar{\Box} \theta_\xi$ on $\Sigma$. 
It follows that $\bar{s}^0_{t \xi}/t$ converges to a positive constant $\lim_{t \to \infty} \int_X \bar{\Box} \theta_\xi t^k e^{t \theta_\xi} \omega^n /\int_X t^k e^{t \theta_\xi} \omega^n = \int_\Sigma \bar{\Box} \theta_\xi dm_\infty /\int_X dm_\infty$. 

It suffices to show the rest terms converge to zero as $t$ tends to infinity. 
Again by the above lemma, $t^k \int_X e^{t \theta_\xi} \omega^n$ converges to a positive constant, so that we have 
\[ t^{-1} \log \int_X e^{t \theta_\xi} \omega^n = O (t^{-1} \log t) \to 0 \]
as $t \to \infty$. 
Similarly, we have 
\[ t^{-1} \int_X \theta_{t\xi} e^{\theta_{t \xi}} \omega^n \Big{/} \int_X e^{\theta_{t \xi}} \omega^n = t^{-1} \int_X t^{k+1} \theta_\xi e^{t \theta_\xi} \omega^n \Big{/} \int_X t^k e^{t \theta_\xi} \omega^n = O (t^{-1}) \to 0 \]
as $t \to \infty$. 
\end{proof}

We obtain Theorem B (2). 

\begin{cor}
\label{Futaki vanish}
There exists a vector $\xi \in \mathfrak{k}$ for which the $\mu$-Futaki invariant $\Fut_\xi$ restricted to $\mathfrak{k}^c$ vanishes. 
\end{cor}

\begin{rem}
From Corollary \ref{Fut zero} in the last subsection, we conclude that critical points of $\mathrm{Vol}^\lambda$ are \textbf{not} unique for a K\"ahler class $[\omega]$ with vanishing Futaki invariant $\Fut = 0$ and sufficiently large $\lambda$. 
\end{rem}

\begin{prop}
For each $\lambda \in \mathbb{R}$ and $\xi \in \mathfrak{k} \setminus \{ 0 \}$, the limit of $t^{-1} \Fut_{t \xi}^\lambda (t \xi) = \Fut_{t \xi}^\lambda (\xi)$ as $t \to \infty$ exists and is strictly positive. 
In particular, the functional $\xi \mapsto \Fut^\lambda_\xi (\xi)$ is proper on $\mathfrak{k}$ for each $\lambda \in \mathbb{R}$. 
\end{prop}

\begin{proof}
Remember that $\Fut_{t\xi}^\lambda ({t\xi}) = \int_X \hat{s}^\lambda_{t\xi} \theta_{t \xi} e^{\theta_{t \xi}} \omega^n /\int_X e^{\theta_{t \xi}} \omega^n$. 
As $t^k e^{\theta_{t \xi}} \omega^n$ converges to a positive measure, it suffices to prove that $t^{k-1} \int_X \hat{s}^\lambda_{t\xi} \theta_{t \xi} e^{\theta_{t \xi}} \omega^n$ converges to a positive constant for any $\xi \in \mathfrak{k}$ ($k$ depends on $\xi$). 
Similarly as before, we can suppose $\max \theta_\xi = 0$. 
Put $\Sigma := \theta_\xi^{-1} (0)$. 
We can compute as 
\begin{align*} 
t^{k-1} \int_X \hat{s}^\lambda_{t\xi} \theta_{t \xi} e^{\theta_{t \xi}} \omega^n 
&= t^{k+1} \int_X (s/t + \bar{\Box} \theta_\xi) \theta_\xi e^{t \theta_\xi} \omega^n + t^k \int_X \bar{\Box} \theta_\xi e^{t \theta_\xi} \omega^n 
\\
& \qquad - \lambda t^{k+1} \int_X \theta_\xi^2 e^{t \theta_\xi} \omega^n - t^{k+1} (\bar{s}_{t \xi}/t - \lambda \bar{\theta}_{t\xi}/t) \int_X \theta_\xi e^{t \theta_\xi} \omega^n. 
\end{align*}
By the above lemma, the third term and $\bar{\theta}_{t\xi}/t = \int_X \theta_\xi e^{t \theta_\xi} \omega^n/\int_X e^{t \theta_\xi} \omega^n$ converges to zero, so that the limit can be computed as the limit of 
\[ \int_X (s/t + \bar{\Box} \theta_\xi - \bar{s}_{t \xi}/t) t^{k+1} \theta_\xi e^{t \theta_\xi} \omega^n + \int_X \bar{\Box} \theta_\xi t^k e^{t \theta_\xi} \omega^n. \]

Let $d\varpi$ denote the probability measure on $\Sigma$ given as the limit of the measures $e^{t \theta_\xi} \omega^n/\int_X e^{t \theta_\xi} \omega^n = t^k e^{t \theta_\xi} \omega^n/\int_X t^k e^{t \theta_\xi} \omega^n$. 
Then the integrand $s/t + \bar{\Box} \theta_\xi - \bar{s}_{t \xi}/t$ of the first term uniformly converges to $\bar{\Box} \theta_\xi - \int_{\Sigma} \bar{\Box} \theta_\xi d\varpi$. 
Again thanks to the above lemma, we have non-zero finite measures $dm_\infty' = \lim_{t \to \infty} (-1) t^{k+1} \theta_\xi e^{t \theta_\xi} \omega^n$ and $dm_\infty = \lim_{t \to \infty} t^k e^{t \theta_\xi} \omega^n$ supported on $\Sigma$. (We have $d \varpi = dm_\infty/\int_\Sigma dm_\infty$. ) 
It follows that the limit is given by 
\[ - \int_{\Sigma} \Big{(} \bar{\Box} \theta_\xi - \int_{\Sigma} \bar{\Box} \theta_\xi d\varpi \Big{)} dm'_\infty + \int_{\Sigma} \bar{\Box} \theta_\xi dm_\infty. \] 
Since we have $\sqrt{-1} \dbar (\bar{\Box} \theta_\xi) = i_{\xi^J} \Ric (\omega)$, $\bar{\Box} \theta_\xi$ is constant along each connected critical manifold. 
It follows that $\bar{\Box} \theta_\xi$ is constant along $\Sigma$ (thanks to the connectedness of $\Sigma$, as we noted in the proof of the last proposition) and so the integrand of the first term is identically zero. 
So the limit is $\int_\Sigma \bar{\Box} \theta_\xi dm_\infty$, which is strictly positive as $\theta_\xi$ has a non-degenerate Hessian to the normal direction. 
\end{proof}

Now we obtain the following expected result, which shows that critical points of $\mathrm{Vol}^\lambda$ must converge to the origin as $\lambda$ tends to $- \infty$ as we observed in subsection \ref{lambda as a function on t}. 

\begin{cor}
\label{compact}
The set $\{ \xi \in \mathfrak{k} ~|~ \lambda_\xi \le 0 \}$ is compact for the functional $\lambda_\xi$ considered in section \ref{section: connecting path}. 
As a consequence, $\{ \xi \in \mathfrak{k} ~|~ \mathrm{Fut}^\lambda_\xi \equiv 0 \text{ for some } \lambda \le 0 \}$ is compact. 
\end{cor}

\begin{proof}
It follows from 
\[ \{ \xi \in \mathfrak{k} ~|~ \mathrm{Fut}^\lambda_\xi \equiv 0 \text{ for some } \lambda \le 0 \} \subset \{ \xi \in \mathfrak{k} ~|~ \lambda_\xi \le 0 \} = \{ \xi \in \mathfrak{k} ~|~ \Fut^0_\xi (\xi) \le 0 \}. \]
\end{proof}

The following is a partial evidence for the uniqueness of the candidates of $\xi$ for $\mu^\lambda$-cscK metrics. 

\begin{cor}
\label{K-optimal vectors are finite}
For each $\lambda \le 0$, the set $\{ \xi ~|~ \exists \omega \in [\omega] \text{ is a $\mu^\lambda_\xi$-cscK metric } \}$ is finite and is in the center of $\mathfrak{k}$. 
In particular, $\mathrm{Aut}_\xi^0 (X/\mathrm{Alb}) \subset \mathrm{Aut}^0 (X/\mathrm{Alb})$ is a maximal reductive subgroup if there exists a $\mu^\lambda_\xi$-cscK for some $\lambda \le 0$. 
\end{cor}

\begin{proof}
The set $\kappa$ of isolated local minimizers of $\mathrm{Vol}^\lambda$ with the non-degenerate Hessians is a zero dimensional compact submanifold of $\mathfrak{k}$ and thus consists of finitely many points. 
As we saw in the last subsection, a vector $\xi$ of a $\mu^\lambda_\xi$-cscK metric must be an element of $\kappa$ when $\lambda \le 0$. 
This proves the first claim. 

For each $g \in K$, we have $\Fut^\lambda_{g_* \xi} (\zeta) = \Fut^\lambda_\xi (g^{-1}_* \zeta)$. 
It follows that $K$ fixes the set $\kappa$ and thus $\kappa$ must be in the centralizer of $\mathfrak{k}$. 
We can see the maximal reductiveness of $\mathrm{Aut}^0_\xi (X/\mathrm{Alb})$ from Corollary \ref{reductiveness} and by taking a maximal compact subgroup $K$. 
We already know that the properly $\dbar$-Hamiltonian vector $\xi$ must be tangent to the centralizer of a maximal compact subgroup $K$. 
(It is essential that $\mu$-Futaki invariant is defined on $\mathfrak{h}_0 (X)$ rather than on $\mathfrak{h}_{0, \xi} (X)$ and vanishes on $\mathfrak{h}_0 (X)$ rather than on the complexification ${\mathfrak{k}'}^c \subset \mathfrak{h}_0 (X)$ of the Lie algebra $\mathfrak{k}'$ of the isometry group of the $\mu^\lambda_\xi$-cscK metric. 
It is a priori not evident that we can find a $K$-invariant $\mu^\lambda_\xi$-cscK for a maximal compact subgroup $K \subset \mathrm{Aut}^0 (X/\mathrm{Alb})$, however, the claim indeed holds from this corollary and Corollary \ref{reductiveness} as for $\lambda \le 0$. )
Therefore, the subgroup $\mathrm{Aut}_\xi^0 (X/\mathrm{Alb})$ contains the complexification of $K$, which is a maximal reductive subgroup of $\mathrm{Aut}^0 (X/\mathrm{Alb})$. 
\end{proof}

\section{$\mu$K-energy and $\mu$K-stability}

\subsection{$\mu$K-energy functional}
\label{section: muK-energy}

We introduce $\mu$K-energy functional and observe some fundamental properties of it. 

\subsubsection{Space of K\"ahler metrics and geodesics}

Let $\omega$ be a K\"ahler metric on a K\"ahler manifold $X$ and $\xi$ be a properly $\dbar$-Hamiltonian vector field preserving $\omega$. 
We denote by $\mathcal{H}_{\omega, \xi}$ the space of $\xi$-invariant smooth K\"ahler potentials with respect to $\omega$ and $\ddot{\mathcal{H}}_{\omega, \xi}$ the space of $\xi$-invariant K\"ahler metrics in the fixed cohomology class $[\omega]$. 
Namely, we put 
\begin{align}
\mathcal{H}_{\omega, \xi} 
&:= \{ \phi \in C^\infty_\xi (X; \mathbb{R}) ~|~ \omega + \sqddbar \phi > 0 \}, 
\\ 
\ddot{\mathcal{H}}_{\omega, \xi} 
&:= \{ \omega_\phi \in [\omega] ~|~ \omega_\phi = \omega + \sqddbar \phi > 0, ~\xi \phi = 0 \}. 
\end{align}

We consider the following Riemannian metric on $\ddot{\mathcal{H}}_{\omega, \xi}$: 
\begin{equation}
(\psi_1, \psi_2)_\xi = \int_X \psi_1 \psi_2 ~e^{\theta_\xi (\phi)} \omega_\phi^n, 
\end{equation}
where we identify the tangent space $T_{\omega_\phi} \ddot{\mathcal{H}}_{\omega, \xi}$ with $\{ \psi \in C_\xi^\infty (X) ~|~ \int_X \psi ~e^{\theta_\xi (\phi)} \omega_\phi^n = 0 \}$. 
This pairing is real-valued as $\omega_\phi$ is $\xi$-invariant. 

A path in $\ddot{\mathcal{H}}_{\omega, \xi}$ corresponds to a path of $\xi$-invariant functions $\phi_t$ normalized as $\int_X \dot{\phi}_t e^{\theta_\xi (\phi_t)} \omega_{\phi_t}^n = 0$. 
The energy of a finite path $\{ \phi_t \}_{t \in [a,b]}$ with respect to the Riemannian metric $(\cdot, \cdot)_\xi$ is given by
\[ E (\phi_t) = \int_a^b \int_X |\dot{\phi}_t|^2 e^{\theta_\xi (\phi_t)} \omega_{\phi_t}^n. \]
A \textit{geodesic} is by definition a critical point of the energy functional on the space of paths with fixed initial and terminal points. 
Computing the first derivative of the energy functional shows that geodesic paths precisely correspond to paths satisfying the following equation 
\begin{equation}
\nabla_X (\ddot{\phi}_t - |\dbar \dot{\phi}_t|^2_{g_{\phi_t}}) = 0 
\end{equation}
under the normalization $\int_X \dot{\phi}_t e^{\theta_\xi (\phi_t)} \omega_{\phi_t}^n = 0$. 
As the equation does not change by adding a function depending only on $t$, we can find a geodesic $\phi_t$ by solving the equation
\[ \ddot{\varphi}_t - |\dbar \dot{\varphi}_t|^2_{g_{\varphi_t}} = 0 \]
and putting $\phi_t := \varphi_t - \int_0^t dt \int_X \dot{\varphi}_t e^{\theta_\xi (\varphi_t)} \omega^n_{\varphi_t}$. 
Note that the geodesic equation itself does not depend on $\xi$, however, the normalization of paths does depend on $\xi$. 

\subsubsection{$\mu$K-energy}

Define the \textit{$\mu$K-energy} $\M^\lambda_\xi$ on the space $\mathcal{H}_{\omega, \xi}$ of smooth K\"ahler potentials by 
\begin{equation}
\M_\xi^\lambda (\phi) := - \int_0^1 dt \int_X \hat{s}_\xi^\lambda (g_{\phi_t}) \dot{\phi}_t ~e^{\theta_\xi (\phi_t)} \omega_{\phi_t}^n, 
\end{equation}
where $\phi_t$ is a path connecting $0$ and $\phi$, i.e. $\phi_0 = 0$ and $\phi_1 = \phi$. 
It is independent of the choice of the smooth path $\phi_t$ connecting $0$ and $\phi$. 
Indeed, let $\phi_{t, 0}$ and $\phi_{t, 1}$ be two paths connecting $0$ and $\phi$ and take an interpolating path $\phi_{t, s}$ of paths, then we can calculate as 
\begin{align*}
\frac{d}{ds} \int_0^1 dt 
&\int_X \hat{s}^\lambda_\xi (g_{t, s}) \frac{d\phi_{t, s}}{dt} e^{\theta_\xi^{t, s}} \omega^n_{t, s}
\\
&= \int_0^1 dt \int_X \Big{(} (- \mathcal{D}^{\theta *}_{t, s} \mathcal{D}_{t, s} \frac{d\phi_{t, s}}{ds} + (\dbar^\sharp \bar{s}^\lambda_\xi (g_{t, s})) \frac{d\phi_{t, s}}{ds} ) \frac{d\phi_{t, s}}{dt}
\\
&\qquad + \hat{s}^\lambda_\xi (g_{t, s}) \frac{d^2 \phi_{t, s}}{ds dt} - \hat{s}^\lambda_\xi (g_{t, s}) \frac{d\phi_{t, s}}{dt} (\bar{\Box}_{g_{t, s}} -\xi^J) \frac{d\phi_{t, s}}{ds} \Big{)} e^{\theta_\xi^{t, s}} \omega^n_{t, s}
\\
&= \int_0^1 dt \int_X \Big{(} - (\mathcal{D}_{t, s} \frac{d\phi_{t, s}}{ds} , \mathcal{D}_{t, s} \frac{d\phi_{t, s}}{dt}) + \hat{s}^\lambda_\xi (g_{t, s}) (\frac{d^2 \phi_{t, s}}{dt ds} - (\dbar \frac{d\phi_{t, s}}{dt}, \dbar \frac{d\phi_{t, s}}{ds} )) \Big{)} e^{\theta_\xi^{t, s}} \omega^n_{t, s}
\\
&= \int_0^1 dt \frac{d}{dt} \int_X \hat{s}^\lambda_\xi (g_{t, s}) \frac{d\phi_{t, s}}{ds} e^{\theta_\xi^{t, s}} \omega_{t, s}^n 
\\
&= \int_X \hat{s}^\lambda_\xi (g_{1, s}) \frac{d \phi_{1, s}}{ds} e^{\theta_\xi^{1, s}} \omega_{1, s}^n - \int_X \hat{s}^\lambda_\xi (g_{0, s}) \frac{d\phi_{0, s}}{ds} e^{\theta_\xi^{0, s}} \omega_{0, s}^n
\\
&= 0. 
\end{align*}
Here the third equality follows by the symmetry of the second expression with respect to $s$ and $t$ and the last equality follows just by $(d/ds) \phi_{1, s} = (d/ds) \phi_{0, s} = 0$. 

The $\mu$K-energy $\M^\lambda_\xi$ descends to the space of K\"ahler metrics $\ddot{\mathcal{H}}_{\omega, \xi}$ and the critical points of $\M^\lambda_\xi$ precisely correspond to $\mu^\lambda_\xi$-cscK metrics. 

In the proof of the finite dimensional Kempf-Ness theorem for a moment map $\mu: X \to \mathfrak{k}^*$, we make use of the convexity of the Kempf-Ness functional $\mathfrak{k}/\mathfrak{k}_x \to \mathbb{R}$ to prove that $\mu^{-1} (0) \cap x. K^c = x. K$, which is analytically analogous to the uniqueness of ($\mu$-)cscK in a given K\"ahler class and geometrically corresponds to the injectivity of the map to the GIT quotient $\mu^{-1} (0)/K \to X^{ss} \sslash K^c$. 
In order to study the uniqueness of $\mu$-cscK in the same spirit of the Kempf-Ness theorem, we should have the following result. 

\begin{prop}[Convexity along smooth geodesics]
\label{convexity}
The $\mu$K-energy $\M^\lambda_\xi$ is convex along smooth geodesics. 
\end{prop}

\begin{proof}
For a smooth path $\phi_t$ in $\ddot{\mathcal{H}}_{\omega, \xi}$, we compute 
\begin{align*}
\frac{d^2}{dt^2} \M_\xi^\lambda (\phi_t)
&= - \frac{d}{dt} \int_X \hat{s}_\xi^\lambda (g_{\phi_t}) \dot{\phi}_t e^{\theta_\xi (\phi_t)} \omega_{\phi_t}^n
\\
&= - \int_X \Big{(} - \mathcal{D}^{\theta*}_t \mathcal{D}_t \dot{\phi}_t + (\partial^\sharp \hat{s}^\lambda_\xi (g_{\phi_t})) (\dot{\phi}_t) \Big{)} \dot{\phi}_t e^{\theta_\xi (\phi_t)} \omega_{\phi_t}^n 
\\
& \qquad - \int_X \hat{s}^\lambda_\xi (g_{\phi_t}) \ddot{\phi}_t e^{\theta_\xi (\phi_t)} \omega_{\phi_t}^n + \int_X \hat{s}^\lambda_\xi (g_{\phi_t}) \dot{\phi}_t (\bar{\Box}_t - \xi') \dot{\phi}_t e^{\theta_\xi (\phi_t)} \omega_{\phi_t}^n
\\
&= \int_X |\mathcal{D}_t \dot{\phi}_t|^2_{g_t} e^{\theta_\xi (\phi_t)} \omega_{\phi_t}^n - \int_X \hat{s}^\lambda_\xi (g_{\phi_t}) (\ddot{\phi}_t - |\dbar \dot{\phi}_t|^2_{g_{\phi_t}}) e^{\theta_\xi (\phi_t)} \omega_{\phi_t}^n. 
\end{align*}
It follows that for a smooth geodesic $\phi_t$, we have
\[ \frac{d^2}{dt^2} \M_\xi^\lambda (\phi_t) = \int_X |\mathcal{D}_t \dot{\phi}_t|^2_{g_t} e^{\theta_\xi (\phi_t)} \omega_{\phi_t}^n \ge 0. \]
\end{proof}

\subsubsection{Extension to $C^{1,1}$-potentials}

We show that $\M^\lambda_\xi$ can be extended to the space 
\[ \overline{\mathcal{H}_{\omega, \xi}^{1,1}} := \{ \phi \in C_\xi^{1,1} (X) ~|~ \omega + \sqddbar \phi \ge 0 \} \]
of $C^{1,1}$-smooth sub-K\"ahler potentials, which generalizes the result of \cite{Chen2} known as Chen-Tian's formula. 

It is known by \cite{Chen} that for any two smooth K\"ahler metrics there always exists a unique connecting $C^{1,1}$-smooth geodesic in $\overline{\mathcal{H}^{1,1}_{\omega, \xi}}$, where one interprets the geodesic equation as a solution of a Monge-Amp\`ere equation on the complex manifold $X \times \{ a \le |z| \le b \}$ with boundary. 
Using the $C^{1,1}$-extension of the usual K-energy, Berman and Berndtsson \cite{BB} proves the uniqueness of cscK and extremal metrics. 

\begin{prop}
\label{expression}
The $\mu$K-energy $\M^\lambda_\xi$ can be expressed as follows. 
\begin{align}
\notag
\M_\xi^\lambda (\phi)
&= \int_X \log \frac{\omega_{\phi}^n}{\omega^n} e^{\theta_\xi (\phi)} \omega^n_{\phi} - n ! \int_0^1 dt \int_X \dot{\phi}_t (\Ric (\omega) + \bar{\Box}_g \theta_\xi) e^{\omega_{\phi_t} + \theta_\xi (\phi_t)} 
\\ \notag
&\qquad + \bar{s}_\xi \int_0^1 dt \int_X \dot{\phi}_t e^{\theta_\xi (\phi_t)} \omega^n_{\phi_t} + \lambda \int_0^1 dt \Big{(} \int_X \theta_\xi (\phi_t) \dot{\phi}_t e^{\theta_\xi (\phi_t)} \omega^n_{\phi_t} - \bar{\theta}_\xi \int_X \dot{\phi}_t e^{\theta_\xi (\phi_t)} \omega_{\phi_t}^n \Big{)}. 
\\ \label{Mabuchi extension}
&= \int_X \log \frac{e^{\theta_\xi (\phi)} \omega_{\phi}^n}{e^{\theta_\xi} \omega^n} e^{\theta_\xi (\phi)} \omega^n_{\phi} - n ! \int_0^1 dt \int_X \dot{\phi}_t \big{(} (\Ric (\omega) - \sqrt{-1} \partial \bar{\partial} \theta_\xi) + (\bar{\Box}_g \theta_\xi - \xi^J \theta_\xi ) \big{)} e^{\omega_{\phi_t} + \theta_\xi (\phi_t)} 
\\ \notag
&\qquad + \bar{s}_\xi \int_0^1 dt \int_X \dot{\phi}_t e^{\theta_\xi (\phi_t)} \omega^n_{\phi_t} + \lambda \int_0^1 dt \Big{(} \int_X \theta_\xi (\phi_t) \dot{\phi}_t e^{\theta_\xi (\phi_t)} \omega^n_{\phi_t} - \bar{\theta}_\xi \int_X \dot{\phi}_t e^{\theta_\xi (\phi_t)} \omega_{\phi_t}^n \Big{)},
\end{align}
where we put $\bar{\theta}_\xi := \int_X \theta_\xi e^{\theta_\xi} \omega^n /\int_X e^{\theta_\xi} \omega^n$ (independent of $\phi_t$). 

\end{prop}

\begin{proof}
Recall the definition 
\[ \hat{s}_\xi^\lambda (g_{\phi_t}) = (s (g_{\phi_t}) + \bar{\Box}_{\phi_t} \theta_\xi (\phi_t)) + (\bar{\Box}_{\phi_t} \theta_\xi (\phi_t) - \xi^J \theta_\xi (\phi_t)) - \lambda \theta_\xi (\phi_t) - (\bar{s}_\xi - \lambda \bar{\theta}_\xi). \]
Firstly, we transform $s (g_{\phi_t})$ as follows: 
\begin{align*} 
s (g_{\phi_t}) 
&= \mathrm{tr}_{g_{\phi_t}} (\sqdbard \log \det \omega_{\phi_t}) 
\\ 
&= \bar{\Box}_{\phi_t} \log \frac{\omega_{\phi_t}^n}{\omega^n} + \mathrm{tr}_{g_{\phi_t}} (\sqdbard \log \det \omega) 
\\
&= (\bar{\Box}_{\phi_t} - \xi^J) \log \frac{\omega_{\phi_t}^n}{\omega^n} + \xi^J \log \frac{\omega_{\phi_t}^n}{\omega^n} + \mathrm{tr}_{g_{\phi_t}} (\Ric (\omega)). 
\end{align*}
For the second term, we have 
\[ \xi^J \log \frac{\omega_{\phi_t}^n}{\omega^n} = \frac{\omega^n}{\omega_{\phi_t}^n} \xi^J \Big{(} \frac{\omega_{\phi_t}^n}{\omega^n} \Big{)} = \frac{L_{\xi^J} \Big{(} \frac{\omega_{\phi_t}^n}{\omega^n} \cdot \omega^n \Big{)}}{\omega_{\phi_t}^n} - \frac{\omega_{\phi_t}^n}{\omega^n} \frac{L_{\xi^J} \omega^n}{\omega_{\phi_t}^n} = - \bar{\Box}_{g_{\phi_t}} \theta_\xi (\phi_t) + \bar{\Box}_g \theta_\xi. \]
The integration of the first term yields the following entropy term 
\begin{align*}
\int_X \Big{(} (\bar{\Box}_{\phi_t} - \xi^J) \log \frac{\omega_{\phi_t}^n}{\omega^n} \Big{)} \dot{\phi}_t ~ e^{\theta_\xi (\phi_t)} \omega^n_{\phi_t} 
&= \int_X \log \frac{\omega_{\phi_t}^n}{\omega^n} \Big{(} (\bar{\Box}_{\phi_t} - \xi^J) \dot{\phi}_t \Big{)} ~ e^{\theta_\xi (\phi_t)} \omega^n_{\phi_t} 
\\
&= - \frac{d}{dt} \Big{(} \int_X \log \frac{\omega_{\phi_t}^n}{\omega^n} e^{\theta_\xi (\phi_t)} \omega_{\phi_t}^n \Big{)} - \int_X \bar{\Box}_{\phi_t} \dot{\phi}_t e^{\theta_\xi (\phi_t)} \omega_{\phi_t}^n
\\ 
&= - \frac{d}{dt} \Big{(} \int_X \log \frac{\omega_{\phi_t}^n}{\omega^n} e^{\theta_\xi (\phi_t)} \omega_{\phi_t}^n \Big{)} - \int_X \xi^J \dot{\phi}_t e^{\theta_\xi (\phi_t)} \omega_{\phi_t}^n. 
\end{align*}
The second term of the last expression removes the following second term of the minus of the $\mu$K-energy 
\begin{align*} 
\int_0^1 dt \int_X 
& (\bar{\Box}_{\phi_t} \theta_\xi (\phi_t) - \xi^J \theta_\xi (\phi_t)) \dot{\phi}_t ~ e^{\theta_\xi (\phi_t)} \omega_{\phi_t}^n
\\ 
& = \int_0^1 dt \int_X  (\bar{\partial} \theta_\xi (\phi_t), \bar{\partial} \dot{\phi}_t)_{g_{\phi_t}} ~ e^{\theta_\xi (\phi_t)} \omega_{\phi_t}^n
\\
&= \int_0^1 dt \int_X  \xi^J \dot{\phi}_t ~ e^{\theta_\xi (\phi_t)} \omega_{\phi_t}^n. 
\end{align*}
Thus we have the following expression of the minus of the $\mu$K-energy: 
\begin{align*} 
- \M_\xi^\lambda (\phi) = - \int_X
& \log \frac{\omega_{\phi}^n}{\omega^n} e^{\theta_\xi (\phi)} \omega_{\phi}^n + \int_0^1 dt \int_X \Big{(} - \bar{\Box}_{g_{\phi_t}} \theta_\xi (\phi_t) + \bar{\Box}_g \theta_\xi + \mathrm{tr}_{g_{\phi_t}} (\Ric (\omega)) \Big{)} \dot{\phi}_t ~e^{\theta_\xi (\phi_t)} \omega_{\phi_t}^n 
\\
&+ \int_0^1 dt \int_X \bar{\Box}_{g_{\phi_t}} \theta_\xi (\phi_t) \dot{\phi}_t e^{\theta_\xi (\phi_t)} \omega^n_{\phi_t} - \bar{s}_\xi \int_0^1 dt \int_X \dot{\phi}_t ~e^{\theta_\xi (\phi_t)} \omega_{\phi_t}^n 
\\ 
&\quad - \lambda \int_0^1 dt \Big{(} \int_X \theta_\xi (\phi_t) \dot{\phi}_t e^{\theta_\xi (\phi_t)} \omega^n_{\phi_t} - \bar{\theta}_\xi \int_X \dot{\phi}_t e^{\theta_\xi (\phi_t)} \omega_{\phi_t}^n \Big{)} 
\end{align*}
and obtain the first expression of the $\mu$K-energy by $\mathrm{tr}_{g_{\phi_t}} (\Ric (\omega)) \omega_{\phi_t}^n = n \Ric (\omega) \wedge \omega_{\phi_t}^{n-1}$. 
The second expression follows by 
\begin{align*} 
\int_X \xi^J \phi e^{\theta_\xi (\phi)} \omega_{\phi}^n 
&= \int_0^1 dt \frac{d}{dt} \int_X \xi^J \phi_t e^{\theta_\xi (\phi_t)} \omega_{\phi_t}^n 
\\
&= \int_0^1 dt \int_X \xi^J \dot{\phi}_t e^{\theta_\xi (\phi_t)} \omega_{\phi_t}^n - \int_0^1 dt \int_X \xi^J \phi_t (\bar{\Box}_{\phi_t} - \xi^J) \dot{\phi}_t e^{\theta_\xi (\phi_t)} \omega_{\phi_t}^n
\\
&= \int_0^1 dt \int_X (\bar{\Box}_{\phi_t} - \xi^J) (\theta_\xi (\phi_t) - \xi^J \phi_t) \dot{\phi}_t e^{\theta_\xi (\phi_t)} \omega_{\phi_t}^n
\\
&= n! \int_0^1 dt \int_X (-\sqrt{-1} \partial \bar{\partial} \theta_\xi - \xi^J \theta_\xi) e^{\omega_{\phi_t} + \theta_\xi (\phi_t)}, 
\end{align*} 
where we applied 
\[ \int_X \xi^J \varphi e^{\theta_\xi (\phi_t)} \omega_{\phi_t}^n  = \int_X (\bar{\partial} \theta_\xi, \bar{\partial} \varphi) e^{\theta_\xi (\phi_t)} \omega_{\phi_t}^n = \int_X (\bar{\Box}_{\phi_t} \theta_\xi (\phi_t) - \xi^J \theta_\xi (\phi_t)) \varphi e^{\theta_\xi (\phi_t)} \omega_{\phi_t}^n \]
and $\bar{\Box}_{\phi_t} \theta_\xi \omega_{\phi_t}^n = \mathrm{tr}_{g_{\phi_t}} (-\sqrt{-1} \partial \bar{\partial} \theta_\xi) \omega_{\phi_t}^n = -n \sqrt{-1} \partial \bar{\partial} \theta_\xi \wedge \omega_{\phi_t}^{n-1}$. 

\end{proof}

The first term in the second expression (\ref{Mabuchi extension}) of $\M^\lambda_\xi$ is known as the entropy 
\[ H_\mu (\nu) = \int_X \frac{d \nu}{d \mu} \log \Big{(} \frac{d \nu}{d \mu} \Big{)} d\mu \]
for the probability measures $\nu = \frac{1}{V_\xi} e^{\theta_\xi (\phi)} \omega_\phi^n, \mu = \frac{1}{V_\xi} e^{\theta_\xi} \omega^n$. 
Here, for general probability measures, $d \nu/d\mu$ denotes the Radon-Nykodim derivative, which is a measurable function, and the value of the function $(d\nu/d\mu) \log (d\nu/d\mu)$ is defined to be zero on which $d\nu/d\mu$ is zero. 
The total mass $V_\xi = \int_X e^{\theta_\xi} \omega^n = \int_X e^{\theta_\xi (\phi)} \omega_\phi^n$ is independent of the choice of $\phi$ as the Duistermaat-Heckman measure is an invariant of $[\omega + \mu]$. 
Applying the Jensen's inequality with respect to the convex function $\phi (t) = t \log t$ on $[0, \infty)$, we get 
\[ \int_X \frac{d \nu}{d \mu} \log \Big{(} \frac{d \nu}{d \mu} \Big{)} d\mu \ge \phi \Big{(} \int_X \frac{d\nu}{d\mu} d\mu \Big{)} = \phi (1) = 0. \]

For any $C^{1,1}$-smooth path of $C^{1,1}$-smooth sub-K\"ahler potentials $\phi_t$, the current $\omega_{\phi_t}$ is just a differential form with $L^\infty$-coefficient and $\theta_\xi (\phi_t)$ and $\dot{\phi}_t$ are Lipschitz functions on $X$. 
As for $\Ric (\omega), \sqrt{-1} \partial \bar{\partial} \theta_\xi, \bar{\Box}_g \theta_\xi$ and $\xi^J \theta_\xi$, they are constructed from the initial smooth metric $\omega$, so are smooth. 
Thus we obtain the following corollary. 

\begin{cor}[Extension to the space of $C^{1,1}$-smooth sub-K\"ahler potentials]
\label{extension}
The $\mu$K-energy $\M^\lambda_\xi$ can be uniquely extended to the space $\overline{\mathcal{H}^{1,1}_{\omega, \xi}}$ of $C^{1,1}$-smooth sub-K\"ahler potentials so that $\M^\lambda_\xi - \mathcal{H}_\xi$ is continuous, where $\mathcal{H}_\xi$ is the lower-semi continuous function
\[ \mathcal{H}_\xi (\phi) = \int_X \log \frac{e^{\theta_\xi (\phi)} \omega^n_\phi}{e^{\theta_\xi} \omega^n} e^{\theta_\xi (\phi)} \omega^n_\phi \]
on $\overline{\mathcal{H}^{1,1}_{\omega, \xi}}$. 
\end{cor}


\subsection{A prelude to $\mu$K-stability}
\label{section: K-stability}

In this section, we discuss on `$\mu$K-stability' which should fit into the existence problem on $\mu$-cscK metrics. 

For a geodesic ray $\phi: [0, \infty) \to \overline{\mathcal{H}^{1,1}_{\omega, \xi}}$, we put 
\[ M_\xi^{\lambda, \mathrm{NA}} (\phi) := \liminf_{t \to \infty} \frac{\M^\lambda_\xi (\phi_t)}{t}, \]
which might take the value $\infty$ for a general geodesic. 

For a vector $\zeta \in \mathfrak{h}_0 (X)$, the following ray $\phi_t$ gives a smooth geodesic ray: 
\[ \sqddbar \phi_t = f_t^* \omega - \omega, \quad \int_X \dot{\phi}_t e^{\theta_\xi (\phi_t)} (f_t^* \omega)^n = 0, \]
where $f_t$ is the one parameter subgroup generated by the vector field $J\zeta$. 
As for this geodesic ray $\phi$, we can easily see that $M^{\lambda, NA}_\xi (\phi)$ exists along this ray and is nothing but the $\mu$-Fuatki invariant $- \Fut^\lambda_\xi (\zeta)$. 

If the $\mu$K-energy is bounded from below, then we must have $M_\xi^{\lambda, NA} (\phi) \ge 0$. 
The most naive and pretty analytic formulation of $\mu$K-stability is that we call a quadruple $(X, [\omega], \xi, \lambda)$ \textit{$\mu$K-semistable} (with respect to geodesics) if we have $M_\xi^{\lambda, NA} (\phi) \ge 0$ for all geodesics $\phi$ and call it \textit{$\mu$K-polystable} (with respect to geodesics) if it is $\mu$K-semistable and we have $M_\xi^{\lambda, NA} (\phi) = 0$ iff $\phi$ is a geodesic given by a vector $\zeta \in \mathfrak{h}_0 (X)$. 
Then we conjecture there exists a $\mu^\lambda_\xi$-cscK metric in the K\"ahler class $[\omega]$ if and only if the quadruple $(X, [\omega], \xi, \lambda)$ is $\mu$K-polystable. (cf. \cite[Theorem 7]{Lah})

Of course, it is desirable that we can reformulate this quite naive $\mu$K-stability notion to fit into a more algebraic formalism. 
Namely, we should 
\begin{itemize}
\item express $M_\xi^{\lambda, NA} (\phi)$ for a geodesic $\phi$ associated to a test configuration by an equivariant intersection formula using the equivariant polarization $\mathcal{L}$ and the equivariant relative canonical sheaf $\omega_{\bar{\mathcal{X}}/\mathbb{P}^1}$ of the compactified test configuration.  (cf. \cite{Lah} and \cite{Ino2})

\item detect the candidate vector $\xi$ for the solution of the $\mu$-cscK equation uniquely in a torus action, in order to formulate a sensible notion of families of $\mu$K-semistable $T$-varieties enjoying the separation property. 
\end{itemize}

The detection of the candidate (called \textit{K-optimal} in \cite{Ino}) vector $\xi$ follows from, for instance, the uniqueness of local minimizers of $\mathrm{Vol}^\lambda$ when $\lambda \le 0$. 
If this is the case, we can formulate the $\mu^\lambda$K-stability for a $T$-equivariant polarized manifold $(X, [\omega])$ by using the local minimizer $\xi$ of $\mathrm{Vol}^\lambda$. 
It is interesting to ask if there is a wall-crossing phenomena, namely, if the $\mu^\lambda$K-stability of $(X, [\omega])$ with a torus action jumps at some $\lambda \le 0$. 
We will see in the next section the behavior of the existence of the $\mu^\lambda$-cscK metric when perturbing $\lambda$. 

\section{Perturbation and propagation}
\label{section: perturbation}

\subsection{Perturbation of K\"ahler class and $\lambda$}

\subsubsection{Regularity}

We firstly check an elliptic regularity for constant $\mu$-scalar curvature K\"ahler metric. 
Remember that the $\mu$-scalar curvature of a K\"ahler metric $\omega_\phi = \omega + \sqddbar \phi$ can be written as 
\[ s_\xi (\omega_\phi) = (\bar{\Box}_\phi - \xi^J) \Big{(} \log (e^{\theta_\xi (\phi)} \det (g_{k \bar{l}} + \phi_{k \bar{l}})) \Big{)} + \sum_{i=1}^n \partial_i \xi^i. \]
Using this, the equation of constant $\mu$-scalar curvature
\[ s_\xi (\omega + \sqddbar \phi) - \lambda \theta_\xi (\phi) = \bar{s}_\xi - \lambda \bar{\theta}_\xi \]
reduces to the following coupled equation
\begin{equation}
\label{couple}
\begin{cases}
F = \log \frac{e^{\theta_\xi (\phi)}\det (g_{k \bar{l}} + \phi_{k \bar{l}})}{e^{\theta_\xi} \det g_{k \bar{l}}}
\\[5pt]
(\bar{\Box}_\phi - \xi^J) F = \bar{s}_\xi - \lambda \bar{\theta}_\xi + \lambda \theta_\xi (\phi) - (\bar{\Box}_\phi - \xi^J) \log (e^{\theta_\xi} \det g) + \sum_{i=1}^n \partial_i \xi^i. 
\end{cases}
\end{equation}

Take a $C^\infty$-smooth initial K\"ahler metric $\omega$ and a $C^{2, \alpha}$-smooth function $\phi$ so that $\omega_\phi = \omega + \sqddbar \phi$ is a $C^{0, \alpha}$-smooth K\"ahler metric. 
Then $\theta_\xi (\phi) = \theta_\xi - \xi^J \phi$ is a $C^{1, \alpha}$-smooth function and the equation (\ref{couple}) makes sense for a $C^2$-smooth function $F$. 

Suppose $F \in C^2$ satisfies the equation (\ref{couple}). 
By differentiating the first equation in (\ref{couple}), we obtain a local equation 
\begin{equation} 
\bar{\Box}_\phi (\partial_i \phi) = \partial_i F - \partial_i (\xi^J \phi) - g^{k \bar{l}}_{\phi} (\partial_i g_{k \bar{l}}) + g^{k \bar{l}} (\partial_i g_{k \bar{l}}). 
\end{equation}
Since the right hand side of this equation is $C^{0, \alpha}$-smooth and the elliptic operator $\bar{\Box}_\phi$ has $C^{0, \alpha}$-coefficients, the elliptic regularity shows that $\partial_i \phi$ should be $C^{2, \alpha}$-smooth. 
By taking all the derivative $\partial_i$, we obtain the $C^{3,\alpha}$-smoothness of $\phi$. 
Then the right hand side of the second equation in (\ref{couple}) becomes $C^{1,\alpha}$-smooth and the elliptic operator $\bar{\Box}_\phi - \xi^J$ has $C^{1, \alpha}$-coefficients, so again the elliptic regularity shows that $F$ is actually $C^{3,\alpha}$-smooth. 
Now the bootstrapping argument shows that the function $\phi$ and $F$ must be $C^\infty$-smooth functions. 

\subsubsection{Perturbation}
\label{section: perturbation}

Let $\omega$ be a $\mu^\lambda_\xi$-cscK metric on a compact K\"ahler manifold $X$. 
By Corollary \ref{reductiveness}, the metric $\omega$ is preserved by some maximal closed torus $T \subset \mathrm{Aut}^0_\xi (X/\mathrm{Alb})$ containing the torus generated by $\xi$. 
The centralizer of $T$ in $\mathrm{Aut}^0_\xi (X/\mathrm{Alb})$ is the complexified algebraic torus $T^c$. 
We denote by $\mathcal{H}^{1,1} (X, \mathbb{R})$ the space of harmonic real $(1,1)$-form with respect to $\Delta_g = d^* d + d d^*$ associated to the metric $g = \omega J$, i.e. $\mathcal{H}^{1, 1} (X, \mathbb{R}) = \{ \alpha \in \Omega^{1,1} (X, \mathbb{R}) ~|~ \Delta_g \alpha = 0 \}$, which is isomorphic to $H^{1,1} (X, \mathbb{R})$ by the projection. 
The action of the maximal torus $T$ on $\mathcal{H}^{1,1} (X, \mathbb{R})$ is trivial as the action extends to the action on $H^2 (X, \mathbb{R})$, which is trivial as it preserves the integral lattice and $T$ is connected, so that each $\alpha \in \mathcal{H}^{1,1}$ is $T$-invariant. 

Let $\mathcal{U} \subset \mathcal{H}^{1,1} (X, \mathbb{R}) \times C^{k+4, \alpha} (X, \mathbb{R})^T$ be an open neighbourhood of the origin on which we have $\omega + \alpha + \sqddbar \phi > 0$. 
For $(\alpha, \phi) \in \mathcal{U}$, we denote by $g_{\alpha, \phi}$ the K\"ahler metric associated to the K\"ahler form $\omega_{\alpha, \phi} := \omega + \alpha + \sqddbar \phi$ and by $\theta_\eta^{\alpha, \phi}$ the real-valued function satisfying $\sqrt{-1} \dbar \theta_\eta^{\alpha, \phi} = i_{\eta^J} \omega_{\alpha, \phi}$ and $\int_X \theta^{\alpha, \phi}_\eta e^{\theta^{\alpha, \phi}_\xi} \omega_{\alpha, \phi}^n = 0$. 
(This normalization is well-defined since for any constant $c \in \mathbb{R}$ we have $\int_X \theta^{\alpha, \phi}_\eta e^{\theta^{\alpha, \phi}_\xi + c} \omega_{\alpha, \phi}^n = 0$ iff $\int_X \theta^{\alpha, \phi}_\eta e^{\theta^{\alpha, \phi}_\xi} \omega_{\alpha, \phi}^n = 0$. )
The function $\theta_\eta^{\alpha, \phi}$ linearly depends on $\eta$, so that $\theta^{\alpha, \phi}$ is a moment map with respect to $\omega_{\alpha, \phi}$. 
Now consider a smooth map $\mathscr{S}_\xi^\lambda: \mathbb{R} \times \mathfrak{t} \times \mathcal{U} \to C^{k, \alpha} (X, \mathbb{R})^T$ defined by 
\begin{align}
\mathscr{S}_\xi^\lambda (\delta, \zeta, \alpha, \phi) 
&= s^{\lambda+\delta}_{\xi + \zeta} (\omega + \alpha + \sqddbar \phi) 
\\ \notag
&= (s (g_{\alpha, \phi}) + \bar{\Box}_{g_{\alpha, \phi}} \theta_{\xi+ \zeta}^{\alpha, \phi}) + (\bar{\Box}_{g_{\alpha, \phi}} \theta_{\xi+ \zeta}^{\alpha, \phi} - (\xi+ \zeta)^J \theta_{\xi + \zeta}^{\alpha, \phi}) - (\lambda+\delta) \theta_{\xi + \zeta}^{\alpha, \phi}. 
\end{align} 

The linearization of this smooth map $\mathscr{S}^\lambda_\xi$ at $(0, 0, 0, 0) \in \mathfrak{t} \times \mathcal{U}$ is given by 
\begin{align}
(0, \zeta, 0, 0)
& \mapsto 2 (\bar{\Box} - \xi^J) \theta_\zeta - \lambda \theta_\zeta, 
\\ 
(0, 0, 0, \phi)
& \mapsto - \mathcal{D}_\omega^{\theta*} \mathcal{D}_\omega \phi + (\dbar_\omega^\sharp s^\lambda_\xi (\omega)) (\phi)
\end{align}
with respect to a general $T$-invariant initial metric $\omega$, which is not necessarily a $\mu^\lambda_\xi$-cscK metric. 
We do not need the derivative to the directions $(\delta, 0, 0, 0)$ and $(0, 0, \alpha, 0)$. 

Now we show the following Theorem E. 

\begin{thm}
\label{perturbation}
Let $\omega$ be a $\mu$-cscK metric on a compact K\"ahler manifold $X$ with respect to $\xi$ and $\lambda \in \mathbb{R}$. 
Suppose we have $\lambda < 2 \lambda_1$ for the first eigenvalue $\lambda_1$ of the weighted $\bar{\partial}$-Laplacian $\bar{\Box}_\omega - \xi^J$ restricted to the space $C^\infty (X)^T$, where $T$ is a maximal torus contained in ${^\nabla \mathrm{Isom}}_\xi^0 (X, \omega)$. 
Then there exists a neighbourhood $U$ of $[\omega]$ in the K\"ahler cone and a positive constant $\epsilon > 0$ such that for each $\tilde{\lambda} \in (\lambda - \epsilon, \lambda + \epsilon)$, every K\"ahler class $[\tilde{\omega}]$ in $U$ admits a $\mu$-cscK metric $\tilde{\omega}_{\tilde{\lambda}}$ with respect to some vector $\tilde{\xi}_{\tilde{\lambda}} \in \mathfrak{t}$ and the given $\tilde{\lambda}$. 
The vector $\tilde{\xi}_{\tilde{\lambda}}$ is in the center of a maximal compact of $\mathrm{Aut}^0 (X/\mathrm{Alb})$ when $\tilde{\lambda} \le 0$. 
\end{thm}

\begin{proof}
Let $\bar{\mathscr{S}}^\lambda_\xi: \mathbb{R} \times \mathfrak{t} \times \mathcal{U} \to C^{k, \alpha} (X, \mathbb{R})^T/\mathbb{R}$ be the projection of $\mathscr{S}^\lambda_\xi$. 
By the implicit function theorem, it suffices to show that the derivative operator $d_0 \bar{\mathscr{S}}^\lambda_\xi: \mathbb{R} \times \mathfrak{t} \times \mathcal{H}^{1,1} (X, \mathbb{R}) \times C^{k+4, \alpha} (X, \mathbb{R}) \to C^{k, \alpha} (X, \mathbb{R})/\mathbb{R}$ is Fredholm and surjective when restricted to $\{ 0 \} \times \mathfrak{t} \times \{ 0 \} \times C^{k+4, \alpha} (X, \mathbb{R})$. 

As $\omega$ is a $\mu^\lambda_\xi$-cscK metric, we have $d_0 \mathscr{S}^\lambda_\xi (0, 0, 0, \phi) = - \mathcal{D}^{\theta *} \mathcal{D} \phi$. 
Since $\mathcal{D}^{\theta *} \mathcal{D}$ is an elliptic operator and $\mathbb{R} \times \mathfrak{t} \times \mathcal{H}^{1,1} (X, \mathbb{R})$ is finite dimensional, both $d_0 \mathscr{S}^\lambda_\xi$ and $d_0 \bar{\mathscr{S}}^\lambda_\xi$ are Fredholm operators. 

The cokernel (the $L^2 (e^{\theta_\xi} \omega^n)$-orthogonal complement of the image) of the operator $- \mathcal{D}^{\theta *} \mathcal{D}$ is given by 
\begin{align*} 
\{ \psi \in C^{k, \alpha} (X, \mathbb{R})^T 
&~|~ \int_X (\mathcal{D}^{\theta*} \mathcal{D} \phi) \psi ~e^{\theta_\xi} \omega^n = 0 \text{ for all } \phi \in C^{k+4, \alpha} (X, \mathbb{R})^T \} 
\\
&= \{ \psi \in C^{k, \alpha} (X, \mathbb{R})^T ~|~ \mathcal{D} \psi = 0 \} = \mathbb{R} \oplus \mathfrak{t}, 
\end{align*}
where the last equality holds as $T$ is maximal. 
For each non-zero element $\theta_\zeta \in \mathfrak{t}$, which is normalized as $\int_X \theta_\zeta e^{\theta_\xi} \omega^n = 0$, we have 
\[ \int_X (d_0 \mathscr{S}^\lambda_\xi (0, \zeta, 0, 0)) \theta_\zeta e^{\theta_\xi} \omega^n = \int_X (2 |\dbar \theta_\zeta|^2 - \lambda \theta_\zeta^2) e^{\theta_\xi} \omega^n > 0 \]
by our assumption $\lambda < 2 \lambda_1$ and the Poincare inequality. 
Therefore the image $d_0 \mathscr{S}^\lambda_\xi (0, \zeta, 0, 0)$ is non-constant and the composition $D = p \circ d_0 \mathscr{S}^\lambda_\xi|_{\{ 0 \} \times \mathfrak{t} \times \{ 0 \} \times \{ 0 \}}: \mathfrak{t} \to \mathbb{R} \oplus \mathfrak{t}$ with the $L^2 (e^{\theta_\xi} \omega^n)$-orthogonal projection $p: C^{k, \alpha} (X, \mathbb{R})^T \to \mathbb{R} \oplus \mathfrak{t}$ is injective. 
It follows that $\mathbb{R} \oplus \mathrm{Im} D = \mathbb{R} \oplus \mathfrak{t}$ and so $d_0 \bar{\mathscr{S}}^\lambda_\xi$ is surjective when restricted to $\{ 0 \} \times \mathfrak{t} \times \{ 0 \} \times C^{k+4, \alpha} (X, \mathbb{R})$. 
\end{proof}

The perturbed vector $\tilde{\xi}$ is a local minimizer of $\mathrm{Vol}^{\tilde{\lambda}}_{[\tilde{\omega}]}$ by Corollary \ref{corollary of the second variational formula} in the above theorem. 

\begin{rem}
As a cscK metric $\omega$ is a $\mu^\lambda_0$-cscK metric for every $\lambda \in \mathbb{R}$, we in particular obtain a $\mu^\lambda_\xi$-cscK metric for every $\lambda \le 0$ and in every K\"ahler class $[\tilde{\omega}]$ in a neighbourhood $U_\lambda$ of $[\omega]$. 

It is proved in \cite{LS} that there is also a neighbourhood $U_{-\infty}$ of $[\omega]$ such that $[\tilde{\omega}]$ admits an extremal metric. 
Note that a $\mu^\lambda_\xi$-cscK metric (or an extremal metric) is not a cscK metric iff $[\tilde{\omega}]$ has non-trivial Futaki invariant $\Fut_{[\tilde{\omega}]} \neq 0$. 

In the next section, we show that we can take such a neighbourhood $U_{-\infty}$ so that $U_{-\infty} \subset U_\lambda$ for every $\lambda \le 0$. 
\end{rem}


\subsection{Propagation from infinity}

\subsubsection{$\mu$-volume functional and M\"obius bundles}

Consider a funcitonal $W (\xi, \lambda) = W^\lambda (\xi) = \log (\mathrm{Vol}^\lambda (\xi)/(\int \omega^n)^\lambda) - \bar{s}$ on $\mathfrak{k} \times \mathbb{R}$. 
When $\kappa \to 0$, we have the limit of $\kappa^{-1} W (\kappa \eta, \kappa^{-1})$ as follows: 
\begin{align*} 
\kappa^{-1} W (\kappa \eta, \kappa^{-1}) 
&= \kappa^{-1} \left( \int_X (s + \bar{\Box} \theta_{\kappa \eta}) e^{\theta_{\kappa \eta}} \omega^n \Big{/} \int_X e^{\theta_{\kappa \eta}} \omega^n - \underline{s} \right)
\\
&\qquad - \kappa^{-2} \left( \int_X \theta_{\kappa \eta} e^{\theta_{\kappa \eta}} \omega^n \Big{/} \int_X e^{\theta_{\kappa \eta}} \omega^n - \log \int_X e^{\theta_{\kappa \eta}} \omega^n \Big{/} \int_X \omega^n \right) 
\\
&= \kappa^{-1} \left( \int_X (s + \bar{\Box} \theta_{\kappa \eta}) e^{\theta_{\kappa \eta}} \omega^n \Big{/} \int_X e^{\theta_{\kappa \eta}} \omega^n - \underline{s} \right)
\\
&\qquad - \kappa^{-1} \left( \int_X \theta_{\eta} e^{\theta_{\kappa \eta}} \omega^n \Big{/} \int_X e^{\theta_{\kappa \eta}} \omega^n - \int_X \theta_\eta \omega^n \Big{/} \int_X \omega^n \right)
\\
&\qquad + \kappa^{-2} \left( \log \int_X e^{\theta_{\kappa \eta}} \omega^n \Big{/} \int_X \omega^n - \int_X \theta_{\kappa \eta} \omega^n \Big{/} \int_X \omega^n \right)
\\[3pt]
&\longrightarrow \frac{d}{d\kappa}\bigg{|}_{\kappa=0} \left( \int_X (s+ \bar{\Box} \theta_{\kappa \eta}) e^{\theta_{\kappa \eta}} \omega^n \Big{/} \int_X e^{\theta_{\kappa \eta}} \omega^n - \int_X \theta_\eta e^{\kappa \eta} \omega^n \Big{/} \int_X e^{\theta_{\kappa \eta}} \omega^n \right) 
\\
&\qquad \qquad + \lim_{\kappa \to 0} (2\kappa)^{-1} \left( \int_X \theta_\eta e^{\theta_{\kappa \eta}} \omega^n \Big{/} \int_X e^{\theta_{\kappa \eta}} \omega^n - \int_X \theta_\eta \omega^n \Big{/} \int_X \omega^n \right)
\\
&\qquad = \left( \int_X \omega^n \right)^{-2} \left( \int_X s \theta_\eta \omega^n \cdot \int_X \omega^n - \int_X s \omega^n \cdot \int_X \theta_\eta \omega^n \right)
\\
& \qquad \qquad - \frac{1}{2} \left( \int_X \omega^n \right)^{-2} \left( \int_X \theta_\eta^2 \omega^n \cdot \int_X \omega^n - \Big{(} \int_X \theta_\eta \omega^n \Big{)}^2 \right)
\\
&\qquad = - \frac{1}{2} \int_X \Big{(} (s - \underline{s}) - ( \theta_\eta - \underline{\theta}_\eta) \Big{)}^2 \omega^n \Big{/} \int_X \omega^n + \frac{1}{2} \int_X (s - \underline{s})^2 \omega^n \Big{/} \int_X \omega^n. 
\end{align*}
The limit functional is nothing but $-2 C (\eta)$ in subsection \ref{lambda as a function on t}. 
So we get a well-defined continuous map 
\[ \check{W}: \mathfrak{k} \times \mathbb{R} \to \mathbb{R}: (\eta, \kappa) \mapsto \check{W} (\eta, \kappa) = \check{W}^\kappa (\eta) := \kappa^{-1} W (\kappa \eta, \kappa^{-1}). \]
The limit functional $\check{W}^0 = -2 C$ is proper, concave and its unique critical point gives the extremal vector. 
By a similar calculus, we can easily see that this map is at least $C^2$-smooth.

\begin{prop}
\label{extremal background}
There exists a constant $\lambda_0 \in \mathbb{R}$ such that $\mathrm{Vol}^\lambda$ has a unique critical point for every $\lambda < \lambda_0$. 
\end{prop}

\begin{proof}
The derivative of $\check{W}^\kappa$ at $\eta \in \mathfrak{k}$ is given by $\Fut^{\kappa^{-1}}_{\kappa \eta}$, so the critical points of $\check{W}^\kappa$ for $\kappa \neq 0$ are precisely $\kappa^{-1}$-times that of $\mathrm{Vol}^{\kappa^{-1}}$. 
It suffices to show that there exists some $\kappa_0 < 0$ such that $\check{W}^\kappa$ admits a unique critical point for each $\kappa \in (\kappa_0, 0)$. 
As we already see, the set $K := \{ \lambda \xi \in \mathfrak{k} ~|~ \Fut^\lambda_\xi \equiv 0, \lambda \le 0 \}$ is compact (moreover, $\{ \lambda \xi ~|~ \Fut^\lambda_\xi \equiv 0 \}$ converges to $\xi_{\mathrm{ext}}$ as $\lambda \to - \infty)$. 
Since $\check{W}^0 = -2 C$ is strictly concave, a small $C^2$-perturbation of it is again strictly concave on $K$, so that there exists $\kappa_0 < 0$ such that $\check{W}^\kappa$ has a unique critical point on $K$ for every $\kappa \in (\kappa_0, - \kappa_0)$. 
Thus for $\kappa \in (\kappa_0, 0)$, $\check{W}^\kappa$ has a unique critical point on $\mathfrak{k}$ as there is no critical points outside $K$. 
\end{proof}

\begin{rem}
We saw in the above proof that $\check{W}^\kappa$ is strictly concave around $\xi_{\mathrm{ext}}$. 
On the other hand, we have proven in Proposition \ref{properness of mu-volume} that the slope at infinity $\lim_{t \to \infty} t^{-1} \check{W}^\kappa (t \eta) = \lim_{t \to \infty} \mathrm{sign} \kappa \cdot (t |\kappa|)^{-1} \log \mathrm{Vol}^{\kappa^{-1}} (t |\kappa| (\mathrm{sign} \kappa \cdot \eta))$ exists and its sign is that of $\kappa$ for each $\kappa \neq 0$. 
This in particular implies that for a positive $\kappa$ close to $0$ ($\lambda = \kappa^{-1} \gg 0$), the functional $\check{W}^\kappa$ is a `mexican hat potential' on $\mathfrak{k}$, so that the critical points are not unique for these $\kappa > 0$ ($\lambda \gg 0$). 

\end{rem}

We can understand the relation of $W^\lambda$ and $\check{W}^\kappa$ as local indications of a map between M\"obius bundles. 
Let $V$ be a vector space over $\mathbb{R}$. 
We construct a circle $S^1$ by gluing two copies of $\mathbb{R}$, which we distinguish as $\mathbb{R}_{(0)} = \mathbb{R}$ and $\mathbb{R}_{(\infty)} = \mathbb{R}$, by the diffeomorphism $\mathbb{R}_{(0)} \setminus \{ 0 \} \xrightarrow{\sim} \mathbb{R}_{(\infty)} \setminus \{ 0 \}: \lambda \mapsto \lambda^{-1}$ and denote by $\infty \in S^1$ the point corresponding to $0 \in \mathbb{R}_{(\infty)}$. 
We construct a vector bundle $\text{M\"ob} (V)$ over $S^1$ by patching two copies of the trivial bundle $V \times \mathbb{R}_{(0/\infty)} \to \mathbb{R}_{(0/\infty)}$ over the charts by the isomorphism $V \times (\mathbb{R}_{(0)} \setminus \{ 0 \}) \xrightarrow{\sim} V \times (\mathbb{R}_{(\infty)} \setminus \{ 0 \}): (\xi, \lambda) \mapsto (\lambda \xi, \lambda^{-1})$ of vector bundles. 

We can construct a smooth map $\text{M\"ob} W: \text{M\"ob} (\mathfrak{k}) \to \text{M\"ob} (\mathbb{R})$ over $S^1$, which behaves non-linearly over the fibres, by patching the following two horizontal maps via the vertical gluing maps:  
\[ \begin{tikzcd}
\mathfrak{k} \times \mathbb{R}_{(0)} \ar{r} \ar[dashed]{d}{(\xi, \lambda) \mapsto (\lambda \xi, \lambda^{-1})}
& \mathbb{R} \times \mathbb{R}_{(0)} \ar[dashed]{d}{(\rho, \lambda) \mapsto (\lambda \rho, \lambda^{-1})}
& (\xi, \lambda) \ar[mapsto]{r}
& (W (\xi, \lambda), \lambda)
\\
\mathfrak{k} \times \mathbb{R}_{(\infty)} \ar{r}
& \mathbb{R} \times \mathbb{R}_{(\infty)}
& (\eta, \kappa) \ar[mapsto]{r}
& (\kappa^{-1} W (\kappa \eta, \kappa^{-1}), \kappa) 
\end{tikzcd} \]

The fibrewise derivative $D \text{M\"ob} W: \text{M\"ob} (\mathfrak{k}) \to \mathrm{Hom} (\text{M\"ob} (\mathfrak{k}), \text{M\"ob} (\mathbb{R})) = \mathfrak{k}^\vee \times S^1$ of this map is given by 
\[ \begin{tikzcd}
\mathfrak{k} \times \mathbb{R}_{(0)} \ar{r} \ar[dashed]{d}{(\xi, \lambda) \mapsto (\lambda \xi, \lambda^{-1})}
& \mathfrak{k}^\vee \times \mathbb{R}_{(0)} \ar[dashed]{d}{(\phi, \lambda) \mapsto (\phi, \lambda^{-1})}
& (\xi, \lambda) \ar[mapsto]{r}
& (\Fut^\lambda_\xi, \lambda)
\\
\mathfrak{k} \times \mathbb{R}_{(\infty)} \ar{r}
& \mathfrak{k}^\vee \times \mathbb{R}_{(\infty)}
& (\eta, \kappa) \ar[mapsto]{r}
& (\Fut^{\kappa^{-1}}_{\kappa \eta}, \kappa). 
\end{tikzcd} \]


\subsubsection{From extremal metric to $\mu$-cscK metrics}

Consider the following for $\eta \in \mathfrak{k}$ and $\kappa \in \mathbb{R}$:  
\begin{equation}
\check{s}^\kappa_\eta (\omega) := (s (\omega) + \bar{\Box} \theta_{\kappa \eta}) + (\bar{\Box} \theta_{\kappa \eta} - (\kappa \eta)^J \theta_{\kappa \eta}) - \theta_\eta. 
\end{equation}
When $\kappa = 0$, we have 
\[ \check{s}^0_\eta (\omega) = (s (\omega) - \theta_\eta), \]
so that $\check{s}^0_\eta (\omega)$ is constant if and only if $\omega$ is an extremal metric with respect to the vector field $\eta$. 
On the other hand, when $\kappa \neq 0$, we have $\check{s}^\kappa_\eta = s^{\kappa^{-1}}_{\kappa \eta}$, so that $\check{s}^\kappa_\eta (\omega)$ is constant if and only if $\omega$ is a $\mu^\lambda_\xi$-cscK metric with respect to $\lambda = \kappa^{-1}$ and $\xi = \kappa \eta$. 

Let $\omega$ be an extremal metric on $X$ and $T \subset {^\nabla \mathrm{Isom}}^0_\eta (X, \omega)$ be a maximal torus containing the extremal vector $\eta = \mathrm{Im} \partial^\sharp s (\omega)$. 
Take an open set $\mathcal{U} \subset \mathcal{H}^{1,1} (X, \mathbb{R}) \times C^{k+4, \alpha} (X, \mathbb{R})^T$ as in section \ref{section: perturbation}. 
We define a map $\check{\mathscr{S}}^0_\eta: \mathbb{R} \times \mathfrak{t} \times \mathcal{U} \to C^{k, \alpha} (X, \mathbb{R})^T$ by 
\begin{align*} 
\check{\mathscr{S}}^0_\eta (\kappa, \chi, \alpha, \phi) 
&:= \check{s}^\kappa_{\eta + \chi} (\omega + \alpha + \sqrt{-1} \partial \bar{\partial} \phi)
\\
&= (s (g_{\alpha, \phi}) + \bar{\Box}_{g_{\alpha, \phi}} \theta_{\kappa (\eta + \chi)}^{\alpha, \phi}) + (\bar{\Box}_{g_{\alpha, \phi}} \theta_{\kappa (\eta + \chi)}^{\alpha, \phi} - (\kappa (\eta + \chi))^J \theta_{\kappa (\eta + \chi)}^{\alpha, \phi}) - \theta^{\alpha, \phi}_{\eta +\chi}, 
\end{align*}
where $(\alpha, \phi) \in \mathcal{U} \subset \mathcal{H}^{1,1} (X, \mathbb{R}) \times C^{k+4, \alpha}_\eta (X, \mathbb{R})^T$. 

The linearization of this smooth map $\check{\mathscr{S}}^0_\eta$ is given by 
\begin{align}
(0, \chi, 0, 0)
&\mapsto - \theta_\chi
\\
(0, 0, 0, \phi) 
&\mapsto - \mathcal{D}^* \mathcal{D} \phi + (\bar{\partial}^\sharp \check{s}^0_\eta (\omega)) (\phi) 
\end{align}

By applying the implicit function theorem similarly to the proof of Theorem \ref{perturbation}, we get the following theorem. 

\begin{thm}
\label{extremal propagation}
Let $\omega$ be an extremal metric on a compact K\"ahler manifold $X$ with the extremal vector $\eta$. 
There exists a neighbourhood $U$ of $[\omega]$ in the K\"ahler cone and constants $\lambda_-, \lambda_+ \in \mathbb{R}$ such that for each $\lambda \in (- \infty, \lambda_-) \cup (\lambda_+, \infty)$, every K\"ahler class $[\omega]$ in $U$ admits a $\mu$-cscK metric $\omega_\lambda$ with respect to some vector $\xi_\lambda \in \mathfrak{t}$ and the given $\lambda$. 
The vector $\xi_\lambda$ is uniquely determined when $\lambda \ll 0$. 
\end{thm}

\section{Examples}

Here we observe explicit examples of K\"ahler classes admitting $\mu$-cscK metrics, using the method of Calabi ansatz. 
While we get some expected results for $\lambda \le 0$, we also find some strange phenomenon when $\lambda \gg 0$. 

\subsection{Phase transition of $\mu^\lambda$-cscK metrics on $\mathbb{C}P^1$}

\subsubsection{$\mu$-volume functional of $\mathbb{C}P^1$}
\label{mu-volume functional example}

We firstly compute the $\mu$-volume functional of $\mathbb{C}P^1$. 
Consider a $U (1)$-action on $\mathbb{C}P^1$ given by $(z: w). t = (zt: w)$. 
We denote by $\eta \in \mathfrak{u} (1)$ the positive generator of the $U (1)$-action. 

Let us consider the following variant of $\mu$-volume functional: 
\begin{equation} 
\bm{\mu}^\lambda (-2 \xi) := - \log \frac{\mathrm{Vol}^\lambda (\xi)}{(n! e^n)^\lambda}. 
\end{equation}
The critical points $\bm{\mu}^\lambda$ are precisely $(-2)$-times of the critical points of $\mathrm{Vol}^\lambda$. 
Then since $\theta_\xi = \mu_{-2 \xi}$, the functional $\bm{\mu}^\lambda$ can be expressed as the integration of $U(1)$-equivariant closed forms: 
\[ \bm{\mu}^\lambda = -\frac{\int_X (\Ric_\omega + \bar{\Box} \mu) e^{\omega + \mu}}{\int_X e^{\omega + \mu}} + \lambda \frac{\int_X (\omega + \mu) e^{\omega + \mu}}{\int_X e^{\omega + \mu}} - \lambda \log \int_X e^{\omega + \mu}. \]
When $\omega \in m \pi c_1 (X)$, we can normalize the moment map $\mu$ so that $[\omega + \mu] = c_1^{U (1)} (X) = \frac{m}{2} [\mathrm{Ric}_\omega + \bar{\Box} \mu]$ as equivariant cohomology classes. 
Under this normalization, we have 
\[ \int_X (\Ric_\omega + \bar{\Box} \mu) e^{\omega + \mu} = \frac{2}{m} \int_X (\omega + \mu) e^{\omega + \mu} \]
since the integration of equivariant closed form depends only on its equivariant cohomology class. 
Thus the functional $\bm{\mu}^\lambda$ for $[\omega] = m \pi c_1 (X)$ can be expressed as follows: 
\begin{align*} 
\bm{\mu}^\lambda 
&= (\lambda -\frac{2}{m}) \frac{\int_X (\omega + \mu) e^{\omega + \mu}}{\int_X e^{\omega + \mu}} - \lambda \log \int_X e^{\omega + \mu}
\\
&= (\lambda -\frac{2}{m}) \frac{\int_X (n + \mu) e^\mu \frac{\omega^n}{n!}}{\int_X e^\mu \frac{\omega^n}{n!}} - \lambda \log \int_X e^\mu \frac{\omega^n}{n!}. 
\end{align*}
We can compute these integrals using the Duistremaat--Heckman measure $\mathrm{DH} = \mu_* (\omega^n/n!)$ on $\mathfrak{u} (1)^\vee = \mathbb{R}. \eta^\vee$. 

When $X = \mathbb{C}P^1$, the Duistremaat--Heckman measure is nothing but the Lebesgue measure restricted on $[-m \pi, m \pi] \subset \mathbb{R}$. 
So we explicitly compute $\bm{\mu}^\lambda$ for $\mathbb{C}P^1$ as 
\begin{align*} 
\bm{\mu}^\lambda (x. \eta) 
&= (\lambda - \frac{2}{m}) \frac{\int_{-m\pi}^{m\pi} (1+ xt) e^{xt} dt}{\int_{-m\pi}^{m\pi} e^{xt} dt} -\lambda \log \int_{-m\pi}^{m\pi} e^{xt} dt 
\\
&= (\lambda - \frac{2}{m}) \frac{m\pi x}{\tanh (m\pi x)}- \lambda \log (\frac{\sinh (m\pi x)}{m\pi x}) - \lambda \log (2\pi m). 
\end{align*}
Put $\chi = m\pi x$. 
Then the derivative is given by 
\[ \frac{d}{d\chi} \bm{\mu}^\lambda (\frac{\chi}{m\pi}. \eta) = \frac{1}{\chi (\sinh \chi)^2} \Big{(} \frac{2}{m} (\chi^2 - \chi \sinh \chi \cosh \chi) - \lambda (\chi^2 - (\sinh \chi)^2)  \Big{)}. \]

As long as $\lambda \le 4/m$, $x = 0$ is the unique critical point of $\bm{\mu}^\lambda$. 
However, once $\lambda$ exceeds $4/m$, $\bm{\mu}^\lambda$ yields three distinct critical points. 
In this case, non-zero critical points of $\bm{\mu}^\lambda$ maximizes $\bm{\mu}^\lambda$ (minimizes $\mathrm{Vol}^\lambda$), while the critical point $\xi = 0$ turns into a `metastable/supercooled' state. 

\subsubsection{$\mu^\lambda$-cscK metrics on $\mathbb{C}P^1$ for $\lambda > \lambda_{\mathrm{freeze}}$ which are not cscK metrics}
\label{Calabi ruled}

Any $U (1)$-invariant K\"ahler metric on $\mathbb{C}P^1$ can be written as 
\[ \omega = \frac{1}{2} u'' (\rho) d\rho \wedge d\theta \]
on the open set $\mathbb{C}^* \subset \mathbb{C}P^1$ for some strictly positive smooth function $u$ on $\mathbb{R}$, using the coordinate $(\rho, \theta) \in \mathbb{R} \times S^1 \mapsto (e^{\rho + \sqrt{-1} \theta}: 1)$. 
For this metric and $\xi = x. \eta = 2\pi x \frac{\partial}{\partial \theta} \in \mathfrak{u} (1)$, we can compute the ingredients of $\mu$-scalar curvature as follows: 
\begin{gather*}
\Ric (\omega) = - \frac{1}{2} (\log u'')'' d\rho \wedge d\theta, 
\quad s (\omega) = - (u'')^{-1} (\log u'')'', 
\\
\theta_\xi = -(2\pi x) u' + \mathrm{const.}, 
\quad 
\bar{\Box} \theta_\xi = (u'')^{-1} (2\pi x) u''', 
\quad 
\xi' \theta_\xi = (2 \pi x)^2 u''. 
\end{gather*}
Thus we get 
\begin{equation} 
s^\lambda_\xi (\omega) = - (u'')^{-1} (\log u'')'' + 2 (u'')^{-1} (2\pi x) u''' - (2\pi x)^2 u'' + \lambda (2\pi x) u'. 
\end{equation}

We put $I := \mathrm{Im} (u')$ and $\chi = 2\pi x$. 
We denote by $\rho: I \to \mathbb{R}$ the inverse map of $\tau := u': \mathbb{R} \to I$ and put $\varphi (\tau) := u'' (\rho (\tau))$. 
Using 
\begin{gather*} 
\frac{d}{d\rho} = \frac{d\tau}{d\rho} \frac{d}{d\tau} = \varphi \frac{d}{d\tau}, 
\quad
\frac{d^2}{d\rho^2} = \varphi \varphi' \frac{d}{d\tau} + \varphi^2 \frac{d^2}{d\tau^2}, 
\end{gather*}
we can reduce $s^\lambda_\xi (\omega)$ as follows: 
\[ s^\lambda_\xi (\omega) = - (\frac{d}{d\tau} - \chi)^2 \varphi + \lambda \chi \tau. \]

We can recover $u$ (modulo linear function) from $\varphi$ since they are related by the Legendre transform $U: I \to \mathbb{R}$ of $u$: if we put 
\[ U (\tau) := \rho (\tau) \tau - u (\rho (\tau)), \]
then we have $\varphi (\tau) = 1/U'' (\tau)$. 
Thus solving the equation of $\mu^\lambda_\xi$-cscK metric on $\mathbb{C}^*$ reduces to finding a positive function $\varphi$ on $I$ which solves the equation 
\[ - (\frac{d}{d\tau} - \chi)^2 \varphi + \lambda \chi \tau = c \]
for a constant $c$. 
When $\chi \neq 0$, the equation is 
\[ (\frac{d}{d\tau} - \chi)^2 (\varphi - \frac{\lambda}{\chi} \tau - \frac{2 \lambda - c}{\chi^2})  \]
we can see the solution is given by 
\begin{equation}
\label{moment profile} 
\varphi^\lambda_\chi (\tau) = a e^{\chi \tau} + b \tau e^{\chi \tau} + \frac{\lambda}{\chi} \tau + \frac{2\lambda - c}{\chi^2} 
\end{equation}
for some $a, b, c \in \mathbb{R}$. 

Now we impose boundary conditions on $\varphi$ to get a metric on $\mathbb{C}P^1$. 
We may assume $I = (0, 2m)$ for some $m$ by adding linear function to $u$. 
Since $\pi u'$ gives a moment map, we have $\int_{\mathbb{C}P^1} \omega = \pi \int_I d\tau = 2 \pi m$. 
To get a solution with $\omega \in 2 \pi c_1 (\mathcal{O} (1))$, we assume $m=1$. 
As usual Calabi ansatz (cf. \cite[Section 4.4]{Sze}), we can see that the following boundary conditions on $\varphi$ asserts that the metric $\frac{1}{2} u'' d\rho \wedge d\theta$ on $\mathbb{C}^*$ extends to $\mathbb{C}P^1$: 
\begin{gather*}
\varphi (0) = 0, \quad \varphi (2) = 0, 
\\
\varphi' (0) = -2, \quad \varphi (2) = 2. 
\end{gather*}
If we have a solution $\varphi$ satisfying this boundary condition, then $\varphi$ is automatically positive on $I$ since $\varphi$ has at most one inflection point: 
\[ \varphi'' (\tau) = (b\chi^2 \tau + a\chi^2 + b\chi) e^{\chi\tau}. \]
By the first three boundary conditions, we must have 
\begin{gather*}
a = \frac{2 \lambda \sinh \chi - 2\chi e^\chi}{\chi (\sinh \chi - \chi e^\chi)}, 
\quad 
b = \frac{(2-2\lambda) \sinh \chi}{\sinh \chi -\chi e^\chi} -\frac{\lambda}{\chi}
\\
c = \frac{2\lambda \chi \sinh \chi - 2\chi^2 e^\chi}{\sinh \chi - \chi e^\chi} + 2\lambda
\end{gather*}
We can reduce the last boundary condition $\varphi' (1) =2$ to the following equality on $\chi$: 
\begin{equation}
\label{lambda example} 
\lambda (\chi^2 - (\sinh \chi)^2) - 2(\chi^2 - \chi \sinh \chi \cosh \chi) = 0, 
\end{equation}
which is equivalent to $\frac{d}{d\chi} \bm{\mu}^\lambda (\chi. \eta) = 0$. 
From the observation in \ref{mu-volume functional example}, it has a solution $x \neq 0$ when $\lambda > 4$. 
Thus we get $\mu^\lambda_\xi$-cscK metrics in the K\"ahler class $c_1 (X)$ for a non-zero $\xi$ when $\lambda > 4 \pi$. 

To see the limiting behavior of these $\mu^\lambda$-cscK metrics as $\lambda \to + \infty$, we see $\lambda$ as a function on $\chi$ and observe the limit $\varphi^{\lambda (\chi)}_\chi$ as $|\chi| \to \infty$. 
Using $\lambda (\chi) = 2\chi + O (\chi^2 e^{-2 \chi})$ from (\ref{lambda example}), we get $a e^{\chi \tau}, b e^{\chi \tau}, (2\lambda + c)/\chi^2 \to 0$ and $\lambda/\chi \to 2$ for each $\tau \in [0,2)$. 
Thus we see $\varphi^{\lambda (\chi)}_\chi \to 2 \tau$ on $[0,2)$ as $\chi \to \infty$. 
The metric tensor $g$ corresponding to $\varphi$ is expressed as 
\[ g = \frac{1}{2} \varphi (\tau)^{-1} d\tau \otimes d\tau + \frac{1}{2} \varphi (\tau) d\theta \otimes d\theta \]
on $(0, 2) \times S^1$, which we identify a metric on $\mathbb{R} \times S^1 \cong \mathbb{C}^*$ via the diffeomorphism $U' (\tau) = \int_1^\tau \varphi (\tau)^{-1} d\tau: (0, 2) \to \mathbb{R}$. 
Thus the limit metric is expressed as  
\[ g = \frac{1}{4\tau} d\tau \otimes d\tau + \tau d\theta \otimes d\theta = dr \otimes dr + r^2 d\theta \otimes d\theta \]
on $(r, \theta) = (\sqrt{\tau}, \theta) \in (0, \sqrt{2}) \times S^1$, which is the flat disk of radius $\sqrt{2}$. 
Since $\varphi^{\lambda (\chi)}_\chi$ converges locally uniformly on $[0, 2)$, the diffeomorphisms $U'_\chi: [0, 2) \to [-\infty, \infty)$ converges to a smooth map $U'_\infty (\tau) = \frac{1}{2} \log \tau$, which is not a diffeomorphism onto $[-\infty, \infty)$. 

\subsection{$\mu$-cscK metrics on $\mathbb{P}_\Sigma (L \oplus \mathcal{O})$}

\subsubsection{The case $\lambda \ge 0$}

Let $L$ be an ample line bundle on a curve $\Sigma$ of degree $k \ge 1$. 
Let $F$ be a fibre of the ruled surface $X = \mathbb{P}_\Sigma (L \oplus \mathcal{O}) \to \Sigma$ and $B$ denote the section at infinity: $B := \{ (x, (0:1)) ~|~ x \in \Sigma \}$. 
The second cohomology $H^2 (X, \mathbb{R})$ is spanned by these divisors, whose intersections are given by 
\[ F \cdot F = 0, \quad F \cdot B = 1, \quad B \cdot B = k. \]
The K\"ahler cone is given by 
\[ \{ a F + b B ~|~ b > 0, ~ \frac{a}{b} > - \frac{k}{2} \}. \]
Now we show the following. 

\begin{prop}
Every K\"ahler class in the cone $\{ a F + b B ~|~ a, b > 0 \}$ admits a $\mu^\lambda$-cscK metric for every $\lambda \ge 0$. 
\end{prop}

Since the existence of $\mu^\lambda$-cscK metric depends only on the ray of K\"ahler class, we may assume the K\"ahler class is represented by $2 \pi (F + m B)$ for some $m \in (0, \infty)$. 

As in \cite[Section 4.4]{Sze}, we consider metrics of the form 
\[ p^* \omega_\Sigma + \sqddbar u \circ s \]
for a function $u: \mathbb{R} \to \mathbb{R}$, where $\omega_\Sigma$ is the K\"ahler--Einstein metric on $\Sigma$ with $\int_\Sigma \omega_\Sigma = 1$ and $s$ is the function $s: L \setminus \Sigma \to \mathbb{R}: z \mapsto \log |z|_h^2$ defined by a metric $h$ on $L$ with curvature $k \omega_\Sigma$. 
Taking a local trivialization $w$ of $L$ around $z_0 \in \Sigma$ so that $(\partial h/\partial z) (z_0) = 0$, we have 
\[ \omega = (1-k u') \omega_\Sigma + u'' \sqrt{-1} \frac{d w \wedge d\bar{w}}{|w|^2} \]
on the fibre $L_{z_0}$. 
Thus the metric is positive iff $1- k u' > 0$ and $u'' > 0$. 
Since $\int_F \omega = 2 \pi (u' (\infty) - u' (-\infty))$ and $\int_B \omega = 2 \pi (1- k u' (-\infty))$, we have 
\[ [\omega] = 2\pi ((1- ku' (\infty))F + (u' (\infty) - u' (-\infty)) B). \]
When $u' (\infty) = 0$ and $u' (-\infty) = -m$, we have $[\omega] = 2\pi (F + m B)$. 

We put $\tau := u': \mathbb{R} \to (-m, 0)$ and denote by $s: (-m, 0) \to \mathbb{R}$ its inverse map. 
Using the function $\varphi (\tau) := u'' (s (\tau))$ and $l_g := 2-2g$, we can express 
\begin{align*}
s (\omega) 
&= - \frac{1}{1-k\tau} ((1-k\tau) \varphi)'' + \frac{l_g}{1-k \tau}, 
\\
\theta_\xi 
&= -4 \pi x \tau, 
\\
\bar{\Box} \theta_\xi 
&= - 4\pi x \frac{k \varphi - (1-k\tau)\varphi'}{1-k\tau}, 
\\
\xi' \theta_\xi 
&= (-4 \pi x)^2 \varphi
\end{align*}
for $\xi = 2 \pi x \partial/\partial \theta$. 
Thus we have 
\[ s^\lambda_\xi (\omega) = - \frac{1}{1- k\tau} (\frac{d}{d\tau} - 4\pi x)^2 ((1-k\tau) \varphi) + 4\pi x \lambda \tau + \frac{l_g}{1-k \tau}. \]

Putting $\chi = 4\pi x$, the problem of the existence of $\mu^\lambda$-cscK metric under the Calabi ansatz reduces to solving the following equation 
\begin{equation}
\label{Calabi equation}
(\frac{d}{d\tau} - \chi)^2 ((1-k\tau) \varphi) = - \chi \lambda k \tau^2 + (\chi \lambda + kc) \tau + (l_g -c) 
\end{equation}
on $[-m, 0]$ together with the following boundary conditions (cf. \cite[Section 4.4]{Sze}): 
\begin{gather}
\label{Calabi boundary}
\varphi (0) = 0, 
\quad 
\varphi (-m) = 0, 
\\ \notag
\varphi' (0) = -1, 
\quad 
\varphi' (-m) = 1. 
\end{gather}
The solutions of (\ref{Calabi equation}) is given by 
\[ \varphi (\tau) = \frac{1}{1- k \tau} \Big{(} a e^{\chi \tau} + b \tau e^{\chi \tau} - \frac{\lambda k}{\chi} \tau^2 + \frac{\lambda \chi + kc - 4 \lambda k}{\chi^2} \tau + \frac{(2 \lambda + l_g - c)\chi + 2kc - 6\lambda k}{\chi^3} \Big{)} \]

Suppose we have a solution $\varphi$ with $\lambda \ge 0$, $x < 0$. 
Then since $((1- k\tau) \varphi (\tau))'' = (b\chi^2 \tau + a\chi^2 + 2b\chi) e^{\chi \tau} - \frac{2 \lambda k}{\chi}$, $\varphi$ satisfies one of the following: 
\begin{enumerate}
\item If $b \ge 0$, then $\varphi$ has at most inflection point. 

\item If $b < 0$, then $\varphi$ may have two inflection points $a, b \in (-m, 0)$. However, $\varphi$ is convex on the intervals $(-m, a)$, $(b, 0)$ and is concave on $(a, b)$. 
\end{enumerate}
In both cases, $\psi (\tau) = (1-k \tau) \varphi$ must be positive from the boundary conditions $\psi (0) = 0, \psi (-m) = 0$ and $\psi' (0) =-1, \psi( -m) = 1$. 

From the first three boundary conditions, we get 
\begin{gather*}
a = c \frac{\chi -2k}{\chi^3} + \frac{(-2\lambda - l_g) \chi + 6 \lambda k}{\chi^3}, 
\quad 
b= c \frac{-\chi + k}{\chi^2} + \frac{-\chi^2 + (\lambda + l_g) \chi - 2 \lambda k}{\chi^2}
\end{gather*}
and 
\begin{align*} 
c
&= \frac{(-m \chi^3 + m (\lambda + l_g) \chi^2 + (2 \lambda + l_g - 2\lambda k m)\chi - 6\lambda k)e^{-m\chi}}{(m\chi^2 + (1-mk) \chi - 2 k)e^{-m\chi} - (1+mk) \chi + 2k} 
\\
&\qquad \quad + \frac{(m^2 \lambda + m \lambda) \chi^2 - (4m \lambda k + 2\lambda + l_g) \chi + 6 \lambda k}{(m\chi^2 + (1-mk) \chi - 2 k)e^{-m\chi} - (1+mk) \chi + 2k}. 
\end{align*}

Regarding $a, b, c$ as a function on $x$, we can see that 
\begin{gather*}
\lim_{\chi \to 0} c (\chi) = \frac{6+3m l_g}{3m + m^2 k}, 
\\
\chi^3 a (\chi) = \Big{(} -2k \frac{6+3m l_g}{3m + m^2 k} + 6\lambda k \Big{)} + \Big{(} - 2 \lambda -l_g + \frac{6+3m l_g}{3m + m^2 k} \Big{)} \chi + O (\chi^2), 
\\
\chi^2 b (\chi) = \Big{(} k \frac{6+3m l_g}{3m + m^2 k} - 2\lambda k \Big{)} + \Big{(} \lambda + l_g - \frac{6+3m l_g}{3m + m^2 k} \Big{)} \chi + O (\chi^2) 
\end{gather*}
around $\chi = 0$. 
Using this, we can see 
\begin{align*} 
\lim_{\chi \to 0} \varphi'_\chi (-m) 
&= \lim_{\chi \to 0} \frac{1}{1+km} \Big{(} a \chi e^{-m\chi} + b e^{-m\chi} - m b\chi e^{-m\chi} + \frac{2 \lambda k}{\chi} m - a\chi - b -1 \Big{)}
\\
&= \lim_{\chi \to 0} \frac{1}{1+km} \Big{(} \frac{a\chi^3 + 2b \chi^2 - 2 \lambda k}{\chi} \frac{e^{-m \chi} - 1}{\chi} - b\chi^2 \frac{e^{-m\chi} + m \chi e^{-m\chi} -1}{\chi^2}  
\\
&\qquad \qquad \qquad \qquad + 2 \lambda k \frac{e^{-m\chi} + m\chi -1}{\chi^2} -1 \Big{)}
\\
&= \frac{1}{1+km} \Big{(} (l_g - \frac{6+3m l_g}{3m + m^2 k}) (-m) - (k \frac{6+3m l_g}{3m + m^2 k} - 2 \lambda k) (- \frac{m^2}{2}) 
\\
&\qquad \qquad \qquad \qquad + 2 \lambda k (\frac{m^2}{2}) -1 \Big{)}
\\
&= \frac{1}{2} \frac{k l_g m^2 + 4km + 6}{k^2 m^2 + 4km +3}. 
\end{align*}
In particular, $\varphi'_\chi (-m)$ extends continuously across $\chi=0$. 

From the above explicit calculus, we can see that $\lim_{\chi \to 0} \varphi'_\chi (-m) < 1$, which is equivalent to $mk ((2k -l_g)m + 4) > 0$, as we have $m > 0, k \ge 1, l_g \le 2$. 
Thus if we have $\lim_{\chi \to -\infty} \varphi'_\chi (-m) = + \infty$, then there must be some $\chi \in (-\infty, 0)$ satisfying $\varphi'_\chi (-m) = 1$, which solves $\varphi$ satisfying all the boundary conditions. 

As $\chi$ tends to $-\infty$, we can see the following: 
\begin{align*}
c(\chi)
&= -\chi + \frac{1- mk + m\lambda + ml_g}{m} 
\\
&\qquad + \frac{1}{m} \Big{(} - \frac{1- mk + m \lambda + ml_g}{m} (1-mk) + 2 \lambda + l_g - 2 \lambda k m - 2k \Big{)} \chi^{-1} + O (\chi^{-2}) 
\end{align*}
and $\chi a(\chi) \to -1, \chi b(\chi) \to -\frac{1}{m}$ $ (\chi a(\chi) - m\chi b(\chi)) \chi \to 0$.
Using this, we obtain 
\[ \varphi_\chi' (-m) \chi e^{m\chi} \to - \frac{1}{m}. \]
Thus $\varphi_\chi' (-m)$ tends to $+ \infty$ as $\chi$ goes to $-\infty$. 
Similarly, we can also see that $\varphi_\chi' (-m) \to 0$ as $\chi \to -\infty$, but we do not use this fact in this paper as we already have $\lim_{\chi \to 0} \varphi_\chi' (-m) < 1$. 

From the above observation, we get a positive solution $\varphi$ of (\ref{Calabi equation}) satisfying all the boundary conditions (\ref{Calabi boundary}), which shows the existence of a $\mu^\lambda$-cscK metric in the K\"ahler class $2 \pi (F + m B)$ for every $\lambda \ge 0$. 

\subsubsection{Connecting K\"ahler--Ricci soliton and extremal metric via $\mu$-cscK metrics}

Consider the case $X = \mathbb{P}_{\mathbb{C}P^1} (\mathcal{O} (1) \oplus \mathcal{O}) = \mathbb{C}P^2 \# \overline{\mathbb{C}P^2}$. 
We have $K_X = \mathcal{O}_{X/\Sigma} (-2) \otimes \pi^* (K_\Sigma \otimes \det (\mathcal{O} (1) \oplus \mathcal{O})^\vee)= 2 (kF -B) - l_g F - kF = -F - 2B$. 
It is known that there are both K\"ahler--Ricci soliton and extremal metric in the K\"ahler class $2 \pi (F + 2 B)$. 
Now we show that there exists $\mu^\lambda$-cscK metrics also for $(-\infty, 0)$ with $x < 0$, which converges to the extremal metric as $\lambda \to -\infty$ and to the $\mu^0$-cscK metric we constructed in section \ref{Calabi ruled} as $\lambda \to 0$. 
Thus we get a continuity path of $\mu$-cscK metrics which connects the K\"ahler--Ricci soliton ($\mu^1$-cscK) and the extremal metric. 

In this case, since $-\lambda /\chi <0$, $(1- \tau) \varphi (\tau)$ might have two inflection points $a, b \in (-2, 0)$ such that $(1- \tau) \varphi (\tau)$ is concave on $(-2, a)$, $(a, 0)$ and is convex on $(a, b)$, which might make $\varphi$ negative at some point in $(-2, 0)$. 
If this happens, we must have $((1- \tau) \varphi)''' (\tau_0) = 0$ at some point $\tau_0 \in (-2, 0)$. 
Since $((1- \tau) \varphi)''' = (b \chi^3 \tau + a\chi^3 + 3 b\chi^2)e^{\chi \tau}$, we have $\tau_0 = - \frac{a (\chi)}{b (\chi)} - \frac{3}{\chi}$ for $\chi < 0$ solving $\varphi'_\chi (-2) = 1$. 
To see the positivity of $\varphi$ on $(-2, 0)$, it suffices to show $\tau_0 > 0$. 
To achieve this, we explicitly compute the following: 
\begin{enumerate}
\item For $\chi < 0$ solving $\varphi'_\chi (-2) = 1$ with $\lambda \in (-\infty, 0)$, we have $\chi \in (-1, 0)$. Here the lower bound $-1$ is not effective. We actually have $\chi \in (-0.265..., 0)$. 
 
\item The function $\tau_0 (\chi) = - \frac{a (\chi)}{b (\chi)} - \frac{3}{\chi}$ is positive on $\chi \in (-1, 0)$ for every $\lambda < 0$. 
\end{enumerate}

For the second item, we explicitly write down as follows: 
\begin{align*} 
\frac{a (\chi)}{b (\chi)} + \frac{3}{\chi} 
&= \frac{\alpha (\chi) + \lambda \beta (\chi)}{\chi (\gamma (\chi) + \lambda \delta (\chi))}, 
\end{align*}
where 
\begin{align*} 
\alpha (\chi) 
&= (-2\chi^4 + \chi^3 + 2\chi^2 -6\chi) e^{-2\chi} + 9 \chi^3 - 14 \chi^2 + 6\chi, 
\\
\beta (\chi)
&= (4\chi^2 + 6\chi-6) e^{-2\chi} + 6\chi^3 + 5\chi^2 + 28 \chi -6, 
\\
\gamma (\chi) 
&= (-\chi^3 + 2\chi^2 -2\chi) e^{-2\chi} + 3\chi^3 - 6\chi^2 +2\chi, 
\\
\delta (\chi) 
&= (-\chi^2 + 4\chi -2) e^{-2\chi} + 7\chi^2 + 8\chi -2. 
\end{align*}
We can see that $\alpha (\chi), \gamma (\chi) > 0$ and $\beta (\chi), \delta (\chi) < 0$ for $\chi \in (-1, 0)$. 
So we have $\tau_0 (\chi) > 0$ for $\chi \in (-1, 0)$. 
Since $\alpha (\chi)$ is not necessarily positive for more smaller $\chi < 0$, we must bound $\chi$. 
In fact, we have $\alpha (-1.5) < 0$. 

To see that $\chi \in (-1, 0)$ for $\lambda < 0$, we observe $\lambda$ as a function on $\chi$. 
We can reduce $\varphi'_\chi (-m) = 1$ to the following equality: 
\[ \lambda = \chi \frac{(9\chi^2 -6\chi-2) e^{2\chi} + (-\chi^2 +2\chi -2) e^{-2\chi} + (-12\chi^3 + 16\chi^2 + 4\chi + 4)}{(9\chi^2 - 12\chi+2) e^{2\chi} + (\chi^2 -4\chi+2) e^{-2\chi} + (-12 \chi^4 + 16\chi^3 -2\chi^2 +16\chi -4)}. \] 
We can see that $\lambda$ is monotonically decreasing for $\chi \in (-\infty, 0)$, $\lambda (-1) > 0$ and $\lim_{\chi \to -0} \lambda (\chi) = -\infty$, thus we conclude that $\lambda (\chi) \le 0$ implies $\chi \in (-1, 0)$.

\end{document}